\documentclass[11pt]{article}
 
\usepackage[margin=0.5in]{geometry} 
\usepackage{amsmath,amsthm,amssymb}
\usepackage{mathrsfs}
\usepackage[T1]{fontenc}
\usepackage[english]{babel}
\usepackage{graphicx}
\usepackage{color} 

\usepackage{hyperref}
\usepackage{afterpage}

\usepackage{color}

\usepackage{atbegshi}
\AtBeginDocument{\AtBeginShipoutNext{\AtBeginShipoutDiscard}}
\usepackage{pdfpages} 
\usepackage{longtable} 
\usepackage{amsmath,amsthm,amsfonts,amssymb,relsize,bm} 
\usepackage{xparse,xcolor} 
\usepackage{graphicx,float} 
\usepackage{etoolbox}
\usepackage[shortlabels]{enumitem} 
\usepackage{tikz-cd} 
\usepackage{cite} 
\usepackage[normalem]{ulem}
\usepackage{comment}
\usepackage{tocloft}

\numberwithin{equation}{section}

\theoremstyle{plain}
\newtheorem{theorem}{Theorem}[section]
\newtheorem*{theorem*}{Theorem}
\newtheorem{lemma}{Lemma}[section]

\newtheorem{claim}{Claim}[section]
\newtheorem{conjecture}{Conjecture}[section]

\theoremstyle{definition}

\newtheorem*{definition*}{Definition}

\newtheorem{remark}{Remark}[section]

\begin{document}

\title{Estimates for the number of zeros of shifted combinations of completed Dirichlet series}

\author{{Pedro Ribeiro}} 
\thanks{


Department of Mathematics, Faculty of Sciences of University of Porto, Rua do Campo Alegre,  687; 4169-007 Porto (Portugal). 

\,\,\,\,\,E-mail: pedromanelribeiro1812@gmail.com}
\date{}
 
\maketitle

\begin{abstract} In a previous paper \cite{RYCE}, Yakubovich and the author of this article proved that certain shifted combinations of completed Dirichlet series have infinitely many zeros on the critical line. Here we provide some lower bounds for the number of critical zeros of a subclass of shifted combinations.
\end{abstract}

\tableofcontents 

\pagenumbering{arabic}
\newpage

\section{Introduction and main results}

Let $\eta(s):=\pi^{-s/2}\Gamma(s/2)\,\zeta(s)$. A. Dixit, N. Robles,
A. Roy and A. Zaharescu \cite{DRRZ} proved the following theorem.

\paragraph*{Theorem A:}
\textit{Let $(c_{j})_{j\in\mathbb{N}}$ be a sequence of non-zeros real numbers
such that $\sum_{j=1}^{\infty}\,|c_{j}|<\infty$. Also, let $\left(\lambda_{j}\right)_{j\in\mathbb{N}}$
be a bounded sequence of distinct real numbers that attains its bounds.
Then the function
\[
F(s)=\sum_{j=1}^{\infty}c_{j}\,\eta\left(s+i\lambda_{j}\right):=\sum_{j=1}^{\infty}c_{j}\,\pi^{-\frac{s+i\lambda_{j}}{2}}\Gamma\left(\frac{s+i\lambda_{j}}{2}\right)\,\zeta\left(s+i\lambda_{j}\right)
\]
has infinitely many zeros on the critical line $\text{Re}(s)=\frac{1}{2}$.}
\\

See also the introduction of \cite{DRRZ} for an excellent survey on
results about zeros of certain shifts of the Riemann zeta function. Based on an integral representation of Jacobi's transformation formula
due to Dixit {[}\cite{Dixit_theta}, p. 374, eq. (1.13){]},
A. Dixit, R. Kumar, B. Maji and A. Zaharescu \cite{DKMZ} later generalized the
aforementioned result and proved the more general theorem.

\paragraph*{Theorem B:}
\textit{Let $(c_{j})_{j\in\mathbb{N}}$ and $(\lambda_{j})_{j\in\mathbb{N}}$
be as in Theorem A. Also, let $\mathscr{R}$ denote the region of
the complex plane defined by $\mathscr{R}:=\left\{ z\in\mathbb{C}\,:\,|\text{Re}(z)|<\sqrt{\frac{\pi}{2}},\,|\text{Im}(z)|<\sqrt{\frac{\pi}{2}}\right\} $.
Then, for any $z\in\mathscr{R}$, the function
\begin{equation}
F_{z}(s)=\sum_{j=1}^{\infty}c_{j}\,\eta\left(s+i\lambda_{j}\right)\,\left\{ _{1}F_{1}\left(\frac{1-(s+i\lambda_{j})}{2};\,\frac{1}{2};\,\frac{z^{2}}{4}\right)\,+\,_{1}F_{1}\left(\frac{1-(\overline{s}-i\lambda_{j})}{2};\,\frac{1}{2};\,\frac{\overline{z}^{2}}{4}\right)\right\} \label{def function theroem B intro}
\end{equation}
has infinitely many zeros on the critical line $\text{Re}(s)=\frac{1}{2}$.}

\bigskip{}

The proof of Theorem B employed a variant of Hardy's method of studying
the moments of the real function $\eta\left(\frac{1}{2}+it\right)$
{[}\cite{hardy_note}, \cite{titchmarsh_zetafunction}, Chapter X{]}, as well as the elegant
transformation formula  
\begin{align}
2x^{1/4}\psi(x,z)-x^{-1/4}e^{-z^{2}/4} & =2e^{-z^{2}/4}x^{-1/4}\,\psi\left(\frac{1}{x},\,iz\right)-x^{1/4}\nonumber \\
=\frac{1}{2\pi}\,\intop_{-\infty}^{\infty}\pi^{-\frac{1}{4}-\frac{it}{2}}\Gamma\left(\frac{1}{4}+\frac{it}{2}\right)\,\zeta\left(\frac{1}{2}+it\right) & \,_{1}F_{1}\left(\frac{1}{4}+\frac{it}{2};\,\frac{1}{2};\,-\frac{z^{2}}{4}\right)\,x^{-\frac{it}{2}}\,dt,\label{Dixit integral formula jacobi psi}
\end{align}
where
\begin{equation}
\psi(x,z):=\sum_{n=1}^{\infty}e^{-\pi n^{2}x}\cos\left(\sqrt{\pi x}\,n\,z\right),\,\,\,\,\text{Re}(x)>0,\,\,\,z\in\mathbb{C}.\label{generalization of Jacobi formula with z intro}
\end{equation}
The first equality in (\ref{Dixit integral formula jacobi psi}) is,
of course, due to Jacobi but the integral representation appears for
the first time in {[}\cite{Dixit_theta}, p. 374, eq. (1.13){]}.

Together with Yakubovich \cite{RYCE}, the author of this paper extended
Theorem B to a class of Dirichlet series satisfying Hecke's functional
equation. A direct generalization of Theorem B concerns Dirichlet
series attached to positive powers of the $\theta-$function,
\begin{equation}
\theta(x):=\sum_{n=-\infty}^{\infty}e^{-\pi n^{2}x},\,\,\,\,\text{Re}(x)>0,\label{definition theta function in introduction}
\end{equation}
defined as follows: for any $\alpha>0$, $r_{\alpha}(n)$ is the arithmetical
function \cite{lagarias_reins} described by the expansion
\begin{equation}
\theta^{\alpha}(x)-1:=\sum_{n=1}^{\infty}r_{\alpha}(n)\,e^{-\pi nx},\,\,\,\,\,\text{Re}(x)>0.\label{r alpha definition iiiintro}
\end{equation}
For $\text{Re}(s)$ sufficiently large, one may consider the Dirichlet
series
\begin{equation}
\zeta_{\alpha}(s):=\sum_{n=1}^{\infty}\frac{r_{\alpha}(n)}{n^{s}},\label{zeta alpha definition intro}
\end{equation}
and motivate its study through the transformation formula for $\theta(x)$.  When $\alpha=k\in\mathbb{N}$,
it is effortless to see that $r_{\alpha}(n)$ reduces to the arithmetical
function counting the number of representations of $n$ as a sum of
$k$ squares. Also, when $\alpha=1$, $\zeta_{1}(s)$ reduces to $2\zeta(2s)$.
Like $\zeta_{k}(s)$ and $\zeta(2s)$, $\zeta_{\alpha}(s)$ satisfies
Hecke's functional equation
\begin{equation}
\eta_{\alpha}(s):=\pi^{-s}\Gamma(s)\,\zeta_{\alpha}(s)=\pi^{-\left(\frac{\alpha}{2}-s\right)}\Gamma\left(\frac{\alpha}{2}-s\right)\,\zeta_{\alpha}\left(\frac{\alpha}{2}-s\right):=\eta_{\alpha}\left(\frac{\alpha}{2}-s\right),\label{completed Dirichlet series definition}
\end{equation}
from which it is possible to conclude that $\eta_{\alpha}(s)$ is
real on the critical line $\text{Re}(s)=\frac{\alpha}{4}$. The extension
of Theorem B to the class of zeta functions $\zeta_{\alpha}(s)$ is
described by the following result (cf. {[}\cite{RYCE}, p. 6, Theorem
1.1{]}).

\paragraph*{Theorem C:}
\textit{Let $(c_{j})_{j\in\mathbb{N}}$ be a sequence of real numbers such
that $\sum_{j}|c_{j}|<\infty$ and $\left(\lambda_{j}\right)_{j\in\mathbb{N}}$
be a bounded sequence of real numbers attaining its bounds. Then,
for any $z$ satisfying the condition
\begin{equation}
z\in\mathscr{D}_{\alpha}:=\left\{ z\in\mathbb{C}\,:\,|\text{Re}(z)|<\sqrt{\frac{\pi\alpha}{2}},\,|\text{Im}(z)|<\sqrt{\frac{\pi\alpha}{2}}\right\} ,\label{condition rectangle z zeta alpha}
\end{equation}
the function
\begin{equation}
F_{z,\alpha}(s):=\sum_{j=1}^{\infty}c_{j}\,\pi^{-(s+i\lambda_{j})}\Gamma\left(s+i\lambda_{j}\right)\,\zeta_{\alpha}(s+i\lambda_{j})\,\left\{ _{1}F_{1}\left(\frac{\alpha}{2}-s-i\lambda_{j};\,\frac{\alpha}{2};\,\frac{z^{2}}{4}\right)+\,_{1}F_{1}\left(\frac{\alpha}{2}-\overline{s}+i\lambda_{j};\,\frac{\alpha}{2};\,\frac{\overline{z}^{2}}{4}\right)\right\} \label{function defining coooombinations-1}
\end{equation}
has infinitely many zeros on the critical line $\text{Re}(s)=\frac{\alpha}{4}$.}
\\

Note that Theorem B is a particular case of Theorem C (obtained when $\alpha=1$). One of the main ingredients in the proof of Theorem C is a generalization
of Dixit's integral formula (\ref{Dixit integral formula jacobi psi})
obtained in {[}\cite{RYCE}, p. 30, eq. (2.60){]}. This identity takes the form
\begin{align}
x^{\frac{\alpha}{4}}\psi_{\alpha}(x,z)-x^{-\alpha/4}\,e^{-\frac{z^{2}}{4}} & =e^{-\frac{z^{2}}{4}}x^{-\alpha/4}\,\psi_{\alpha}\left(\frac{1}{x},\,iz\right)-x^{\alpha/4}\nonumber \\
=\frac{1}{2\pi}\,\intop_{-\infty}^{\infty}\eta_{\alpha}\left(\frac{\alpha}{4}+it\right) & \,_{1}F_{1}\left(\frac{\alpha}{4}+it;\,\frac{\alpha}{2};\,-\frac{z^{2}}{4}\right)\,x^{-it}\,dt,\label{equation in the setting of zeta alpha}
\end{align}
where $\eta_{\alpha}(s)$ is the completed Dirichlet series (\ref{completed Dirichlet series definition})
and $\psi_{\alpha}(x,z)$ is the generalization of Jacobi's $\psi-$function,
\begin{equation}
\psi_{\alpha}(x,z)=2^{\frac{\alpha}{2}-1}\,\Gamma\left(\frac{\alpha}{2}\right)\,\left(\sqrt{\pi x}\,z\right)^{1-\frac{\alpha}{2}}\,\sum_{n=1}^{\infty}r_{\alpha}(n)\,n^{\frac{1}{2}-\frac{\alpha}{4}}\,e^{-\pi n\,x}\,J_{\frac{\alpha}{2}-1}(\sqrt{\pi\,n\,x}\,z),\,\,\,\,\,\text{Re}(x)>0,\,\,z\in\mathbb{C}.\label{Definition Jacobi final version}
\end{equation}

The transformation formula (\ref{equation in the setting of zeta alpha})
can be also taken in a general setting, with $\zeta_{\alpha}(s)$
being replaced by any Dirichlet series satisfying Hecke's functional equation.
For example, let $f(\tau)$ be a holomorphic cusp form with weight
$k\geq12$ for the full modular group whose Fourier expansion is given
by
\begin{equation}
f(\tau)=\sum_{n=1}^{\infty}a_{f}(n)\,e^{2\pi in\tau},\,\,\,\,\,\text{Im}(\tau)>0.\label{Fourier expansion cusp fooorms-1}
\end{equation}
If we construct the Dirichlet series associated to $f(\tau)$,
\begin{equation}
L(s,f)=\sum_{n=1}^{\infty}\frac{a_{f}(n)}{n^{s}},\,\,\,\,\,\,\text{Re}(s)>\frac{k+1}{2},\label{cusp form holomorphic def}
\end{equation}
we know that $L(s,f)$ can be analytically continued to an entire
function obeying Hecke's functional equation
\begin{equation}
\eta_{f}(s):=(2\pi)^{-s}\Gamma(s)\,L(s,f)=(-1)^{k/2}\,\left(2\pi\right)^{-(k-s)}\,\Gamma(k-s)\,L(k-s,f):=(-1)^{k/2}\eta_{f}\left(k-s\right).\label{functional equation cusp form}
\end{equation}

Analogously to $\psi_{\alpha}(x,z)$, which generalizes Jacobi's $\psi-$function,
we can construct a generalization of the cusp form $f(\tau)$ in the
form\footnote{Note that, when $z=0$, $\psi_{f}(x,0)=f(ix)$}
\begin{equation}
\psi_{f}(x,z):=(k-1)!\,\left(\sqrt{\frac{\pi x}{2}}\,z\right)^{1-k}\sum_{n=1}^{\infty}a_{f}(n)n^{\frac{1-k}{2}}\,e^{-2\pi n\,x}\,J_{k-1}\left(\sqrt{2\pi n\,x}\,z\right),\,\,\text{Re}(x)>0,\,\,z\in\mathbb{C}.\label{Jacobi theta function cusp forms}
\end{equation}
Just like $\psi_{\alpha}(x,z)$, $\psi_{f}(x,z)$ obeys to the following
transformation
\begin{align}
x^{\frac{k}{2}}\psi_{f}(x,z) & =(-1)^{k/2}\,e^{-\frac{z^{2}}{4}}x^{-k/2}\,\psi_{f}\left(\frac{1}{x},\,iz\right)\nonumber \\
=\frac{1}{2\pi}\,\intop_{-\infty}^{\infty}\left(2\pi\right)^{-\frac{k}{2}-it}\Gamma\left(\frac{k}{2}+it\right)\,L\left(\frac{k}{2}+it,f\right) & \,_{1}F_{1}\left(\frac{k}{2}+it;\,k;\,-\frac{z^{2}}{4}\right)\,x^{-it}\,dt\nonumber \\
:=\frac{1}{2\pi}\,\intop_{-\infty}^{\infty}\eta_{f}\left(\frac{k}{2}+it\right)\,_{1}F_{1} & \left(\frac{k}{2}+it;\,k;\,-\frac{z^{2}}{4}\right)\,x^{-it}dt,\label{Reflection formula for cusp forms}
\end{align}
where $\eta_{f}(s)$ is the completed Dirichlet series (\ref{functional equation cusp form}).
Using (\ref{Reflection formula for cusp forms}), we have been able
to establish the following Theorem {[}\cite{RYCE}, Theorem 1.4{]}.

\paragraph*{Theorem D:}
\textit{Let $f(\tau)$ be a cusp form of weight $k$ for the full modular
group with real Fourier coefficients $a_{f}(n)$. Consider the Dirichlet
series,
\begin{equation}
L(s,f)=\sum_{n=1}^{\infty}\frac{a_{f}(n)}{n^{s}},\,\,\,\,\,\text{Re}(s)>\frac{k+1}{2}.\label{L series cusp form p q-2}
\end{equation}
If $(c_{j})_{j\in\mathbb{N}}$ is a sequence of non-zero real numbers
such that $\sum_{j=1}^{\infty}\,|c_{j}|<\infty$, $\left(\lambda_{j}\right)_{j\in\mathbb{N}}$
is a bounded sequence of distinct real numbers\footnote{In the statement of Theorem 1.4. of \cite{RYCE} we require that the
sequence $\left(\lambda_{j}\right)_{j\in\mathbb{N}}$ attains its
bounds. However, in view of Remark 5.4 of \cite{RYCE}, this condition
is not necessary when we are working with combinations involving entire
Dirichlet series.} and $z$ satisfies the condition
\begin{equation}
z\in\mathscr{D}:=\left\{ z\in\mathbb{C}\,:\,|\text{Re}(z)|<2\sqrt{\pi},\,\,|\text{Im}(z)|<2\sqrt{\pi}\right\} ,\label{condition cusp case-1}
\end{equation}
then the function
\begin{equation}
G_{z,f}\left(s\right):=\sum_{j=1}^{\infty}c_{j}\,\left(2\pi\right)^{-s-i\lambda_{j}}\Gamma(s+i\lambda_{j})\,L\left(s+i\lambda_{j},f\right)\,\left\{ _{1}F_{1}\left(k-s-i\lambda_{j};\,k;\,\frac{z^{2}}{4}\right)+\,_{1}F_{1}\left(k-\overline{s}+i\lambda_{j};\,k;\,\frac{\overline{z}^{2}}{4}\right)\right\} \label{function cusp case-2}
\end{equation}
has infinitely many zeros at the critical line $\text{Re}(s)=\frac{k}{2}$.}
\\

Our main goal in this paper is to prove quantitative analogues of
Theorems C and D above for a subclass of shifted combinations. This is, when we impose some further restrictions
on $(\lambda_{j})_{j\in\mathbb{N}}$, we aim to find lower bounds
for the number of critical zeros of the functions $F_{z,\alpha}(s)$
and $G_{z,f}(s).$ We need to introduce the following assumptions:

\begin{enumerate}

\item{\textbf{Assumpion 1:}\label{assumption 1} If the shift $s\rightarrow s+i\lambda_{j}$ is introduced, then the
symmetric shift $s\rightarrow s-i\lambda_{j}$ needs to be introduced
as well and having the same weight $c_{j}$. Concerning Theorem C,
for example, this means that if the term
\[
c_{j}\,\eta_{\alpha}\left(s+i\lambda_{j}\right)\text{Re}\left\{ _{1}F_{1}\left(\frac{\alpha}{2}-s-i\lambda_{j};\,\frac{\alpha}{2};\,\frac{z^{2}}{4}\right)\right\} 
\]
belongs to the combination (\ref{function defining coooombinations-1}),
then the term
\[
c_{j}\,\eta_{\alpha}\left(s-i\lambda_{j}\right)\text{Re}\left\{ _{1}F_{1}\left(\frac{\alpha}{2}-s+i\lambda_{j};\,\frac{\alpha}{2};\,\frac{z^{2}}{4}\right)\right\} 
\]
also belongs to it. This assumption essentially says that we are enlarging
the sequences $\left(\lambda_{j}\right)_{j\in\mathbb{N}}$ in (\ref{function defining coooombinations-1})
and (\ref{function cusp case-2}) in such a way that each element
has a symmetric pair contained in it. Therefore, we can write it as a sequence over $\mathbb{Z}\setminus\{0\}$ in the form
$\left(\lambda_{j}\right)_{j\in\mathbb{Z}\setminus\{0\}}=\left(\lambda_{j}\right)_{j\in\mathbb{N}}\cup\left(-\lambda_{j}\right)_{j\in\mathbb{N}}$,
once we take the convention that $\lambda_{-j}:=-\lambda_{j}$. The same can be done to $(c_{j})_{j\in \mathbb{N}}$ under the convention $c_{-j}:=c_{j}$.
}

\item{\textbf{Assumption 2:}\label{assumption 2} In each one of the cases (\ref{function defining coooombinations-1}) and (\ref{function cusp case-2}), the parameter $z$ in the hypergeometric functions will
be a real number belonging to an interval contained in the regions
(\ref{condition rectangle z zeta alpha}) and (\ref{condition cusp case-1}).}

\end{enumerate}

Under these assumptions, we can write the shifted combinations (\ref{function defining coooombinations-1})
and (\ref{function cusp case-2}) in the symmetric forms
\begin{equation}
\tilde{F}_{z,\alpha}(s):=\sum_{j\neq0}c_{j}\,\pi^{-(s+i\lambda_{j})}\Gamma\left(s+i\lambda_{j}\right)\,\zeta_{\alpha}\left(s+i\lambda_{j}\right)\,\left\{ _{1}F_{1}\left(\frac{\alpha}{2}-s-i\lambda_{j};\,\frac{\alpha}{2};\,\frac{z^{2}}{4}\right)+\,_{1}F_{1}\left(\frac{\alpha}{2}-\overline{s}+i\lambda_{j};\,\frac{\alpha}{2};\,\frac{z^{2}}{4}\right)\right\} \label{symmetric function F}
\end{equation}
and
\begin{equation}
\tilde{G}_{z,f}(s):=\sum_{j\neq0}c_{j}\,\left(2\pi\right)^{-s-i\lambda_{j}}\Gamma(s+i\lambda_{j})\,L\left(s+i\lambda_{j},f\right)\,\left\{ _{1}F_{1}\left(k-s-i\lambda_{j};\,k;\,\frac{z^{2}}{4}\right)+\,_{1}F_{1}\left(k-\overline{s}+i\lambda_{j};\,k;\,\frac{z^{2}}{4}\right)\right\} .\label{symmetric function G}
\end{equation}
In both cases, the sum is taken over $\mathbb{Z}\setminus\{0\}$ and
we are under the conventions (see Assumption \ref{assumption 1}) $c_{-j}:=c_{j}$ and
$\lambda_{-j}:=-\lambda_{j}$.

Note also that $\tilde{F}_{z,\alpha}\left(\frac{\alpha}{4}+it\right)$
is a real-valued and even function of $t\in\mathbb{R}$. These properties
come immediately from the functional equation (\ref{completed Dirichlet series definition})
and assumptions \ref{assumption 1} and \ref{assumption 2}. In the same lines one can check that, when
the coefficients $a_{f}(n)$ are real, $i^{-\frac{k}{2}}\tilde{G}_{z,f}\left(\frac{k}{2}+it\right)$
is a real function of $t$. Furthermore, it is an even function if $k\equiv 0 \mod4$ and it is odd if $k\equiv 2 \mod 4$. 
Although the additional assumptions given above restrict some of the
general aspects of Theorems C and D, they prove to be very helpful
in deriving some estimates for the number of zeros of the functions
(\ref{symmetric function F}) and (\ref{symmetric function G}). We
begin with a theorem about the number of zeros of $\tilde{F}_{z,\alpha}\left(\frac{\alpha}{4}+it\right)$.
\\

\begin{theorem}\label{theorem 1.1}
Let $(c_{j})_{j\in\mathbb{N}}$ be a sequence of real numbers such
that $\sum_{j=1}^{\infty}|c_{j}|<\infty$ and $(\lambda_{j})_{j\in\mathbb{N}}$
be a bounded sequence of real numbers attaining its bounds. Suppose that $(c_{j})_{j\in\mathbb{N}}$ and $(\lambda_{j})_{j\in\mathbb{N}}$ can be extended to $\mathbb{Z}\setminus\{0\}$ in the form $c_{-j}=c_{j}$ and $\lambda_{-j}=-\lambda_{j}$. Assume
also that $z\in\mathbb{R}$ satisfies the condition
\begin{equation}
z\in\left[-\frac{1}{6}\sqrt{\frac{\pi\alpha}{2}},\frac{1}{6}\sqrt{\frac{\pi\alpha}{2}}\right].\label{condition first rectangle}
\end{equation}
Moreover, let $N_{\alpha,z}(T)$ be the number of zeros written in the form $s=\frac{\alpha}{4}+it,\,\,\,0\leq t\leq T$,
of the function
\begin{equation}
\tilde{F}_{z,\alpha}(s):=\sum_{j\neq0}c_{j}\,\pi^{-(s+i\lambda_{j})}\Gamma\left(s+i\lambda_{j}\right)\,\zeta_{\alpha}\left(s+i\lambda_{j}\right)\,\left\{ _{1}F_{1}\left(\frac{\alpha}{2}-s-i\lambda_{j};\,\frac{\alpha}{2};\,\frac{z^{2}}{4}\right)+\,_{1}F_{1}\left(\frac{\alpha}{2}-\overline{s}+i\lambda_{j};\,\frac{\alpha}{2};\,\frac{z^{2}}{4}\right)\right\} .\label{function defining coooombinations}
\end{equation}

Then there exists some $c>0$ such that
\begin{equation}
\liminf_{T\rightarrow\infty}\,\frac{N_{\alpha,z}(T)}{\sqrt{T}/\log(T)}\geq c.\label{result proved for zeta alpha}
\end{equation}
\end{theorem}

\bigskip{}

The quantitative estimate (\ref{result proved for zeta alpha}) is
somewhat general and it seems very difficult to improve on the power
of $T$ using the same method. In the final section of this paper,
we make some conjectures regarding a possible improvement. When the shifted
combination is trivial and reduces to the term $\lambda_{1}=0$, our
estimate (\ref{result proved for zeta alpha}) extends, for all $\alpha>0$,
a result of the author and Yakubovich {[}\cite{rysc_I}, Corollary 3.5{]}
(cf. Remark \ref{remark 1.1} below).

Our proof of Theorem \ref{theorem 1.1} employs a method developed by Fekete \cite{fekete_zeros},
which was essentially inspired by a lemma of Fej\'er \cite{fejer}. See
{[}\cite{titchmarsh_zetafunction}, p. 259{]} for a clear explanation of Fekete's
method, as well as an interesting paper by Berlowitz which uses the same idea {[}\cite{berlowitz}, p. 206, Lemma 2{]}. However, in order
to apply Fej\'er's lemma, it will be crucial in our argument to use the fact
that $\zeta_{\alpha}(s)$ has a simple pole located at $s=\frac{\alpha}{2}$.

Since $L(s,f)$ is an entire function of $s$, a quantitative analogue
of Theorem D will have to use a different idea. In any case, we have
been able to establish the following result.

\begin{theorem} \label{theorem 1.2}

Let $f(\tau)$ be a cusp form of weight $k$ for the full modular
group with real Fourier coefficients $a_{f}(n)$ and consider the
Dirichlet series, $L(s,f)$, attached to it.  Let $(c_{j})_{j\in\mathbb{N}}$
be a sequence of non-zero real numbers such that $\sum_{j=1}^{\infty}\,|c_{j}|<\infty$
and let $(\lambda_{j})_{j\in\mathbb{N}}$ be a bounded sequence of
distinct real numbers. Suppose that $(c_{j})_{j\in\mathbb{N}}$ and $(\lambda_{j})_{j\in\mathbb{N}}$ can be extended to $\mathbb{Z}\setminus\{0\}$ in the form $c_{-j}=c_{j}$ and $\lambda_{-j}=-\lambda_{j}$.  Assume also that $z$ is a real number satisfying
the condition
\begin{equation}
z\in\left[-\frac{\sqrt{\pi}}{3},\frac{\sqrt{\pi}}{3}\right].\label{interval z in cusp form case}
\end{equation}
If $N_{f,z}(T)$ denotes the number of zeros written in the form $s=\frac{k}{2}+it,\,\,0\leq t\leq T$,
of the function
\begin{equation}
\tilde{G}_{z,f}(s):=\sum_{j\neq0}c_{j}\,\left(2\pi\right)^{-s-i\lambda_{j}}\Gamma(s+i\lambda_{j})\,L\left(s+i\lambda_{j},f\right)\,\left\{ _{1}F_{1}\left(k-s-i\lambda_{j};\,k;\,\frac{z^{2}}{4}\right)+\,_{1}F_{1}\left(k-\overline{s}+i\lambda_{j};\,k;\,\frac{z^{2}}{4}\right)\right\} ,\label{function cusp case}
\end{equation}
then there exists some $d>0$ such that
\begin{equation}
\limsup_{T\rightarrow\infty}\frac{N_{f,z}(T)}{\sqrt{T}}\geq d.\label{lim sup estimate}
\end{equation}
\end{theorem}

\bigskip{}

Our proof of Theorem \ref{theorem 1.2} uses a variant of a method due to de la Vall\'ee Poussin \cite{Poussin_zeros}. The author of this paper has recently used this method to establish a quantitative estimate for the number of
critical zeros of $L-$functions attached to half-integral weight cusp forms \cite{RHALF}.

This paper is organized as follows. In the next section we give the
necessary technical lemmas to establish Theorem \ref{theorem 1.1}. The most important
of these, Lemma \ref{lemma 2.2}, is obtained via a generalization of the theta
transformation formula obtained in \cite{RYCE}. Section \ref{proof theorem 1.1 section} is devoted
to a proof of Theorem \ref{theorem 1.1}. Next, we follow a similar structure in
sections \ref{lemmas cusp forms section} and \ref{proof of theorem 1.2 section}. Finally, we end this paper with some conjectures
concerning further extensions of the main results here presented.
Before moving on, we introduce a couple of remarks which describe some interesting particular cases of our Theorems \ref{theorem 1.1} and \ref{theorem 1.2}, as well as other theorems that can be proved via the same methods.

\begin{remark}\label{remark 1.1}
The intervals (\ref{condition first rectangle}) and (\ref{interval z in cusp form case})
are respectively contained in the regions (\ref{condition rectangle z zeta alpha})
and (\ref{condition cusp case-1}) of Theorems C and D. As we shall
see below, it is possible to enlarge the length of the interval (\ref{condition first rectangle})
by a more careful choice of the parameter $\lambda$ in the proof
of Lemma \ref{lemma 2.2} and a more precise inequality for $|I_{\nu}(x)|$. See
Remark \ref{remark 2.1} below.
\end{remark}

\begin{remark}\label{remark 1.2}
When $z=0$, we deduce from Theorem \ref{theorem 1.1} that the arbitrary shifted
combination given by
\[
\sum_{j\neq0}c_{j}\,\pi^{-(s+i\lambda_{j})}\Gamma\left(s+i\lambda_{j}\right)\,\zeta_{\alpha}\left(s+i\lambda_{j}\right)
\]
has $\gg T^{1/2}/\log(T)$ zeros of the form $s=\frac{\alpha}{4}+it,$
$0\leq t\leq T$. Since $\zeta_{1}(s):=2\zeta(2s)$, this shows that
the function
\begin{equation}
F(s):=\sum_{j\neq0}c_{j}\,\pi^{-\frac{1}{2}(s+i\lambda_{j})}\Gamma\left(\frac{s+i\lambda_{j}}{2}\right)\,\zeta\left(s+i\lambda_{j}\right)\label{zeta case F}
\end{equation}
has $\gg T^{1/2}/\log(T)$ on the critical line $\text{Re}(s)=\frac{1}{2}$. Also, when $z\neq0$
and $\alpha=1$, Theorem \ref{theorem 1.1} yields an extension of Theorem B above.
The generality of this result, with respect to the shifted combination
and the explicit estimate for the number of critical zeros, seems
to be unnoticed in the literature. We should remark that Selberg's
outstanding result about the positive proportion of combinations of
$L-$functions with degree one \cite{selberg_class} does not seem to cover the case
(\ref{zeta case F}), because we are considering shifted combinations
of $\eta(s)$.
\end{remark}

\begin{remark}\label{remark 1.3}
Using an inductive method involving generalized Epstein zeta functions
\cite{rysc_I}, Yakubovich and the author of this paper proved that, when
$\alpha>4$, there always exist $T_{0}(\alpha)$ and $c:=c(\alpha)>0$
such that, for any $T\geq T_{0}(\alpha)$, there is a zero $\rho=\frac{\alpha}{4}+i\gamma$
of $\zeta_{\alpha}(s)$ with $\gamma\in[T,T+cT^{1/2}\log(T)]$. See
{[}\cite{rysc_I}, Corollary 3.5{]} for details. In particular, if $N_{\alpha}(T):=\#\left\{ 0\leq t\leq T\,:\,\zeta_{\alpha}\left(\frac{\alpha}{4}+it\right)=0\right\} $,
then 
\begin{equation}
\liminf_{T\rightarrow\infty}\,\frac{N_{\alpha}(T)}{\sqrt{T}/\log(T)}\geq c^{\prime}>0,\,\,\,\,\,\alpha>4.\label{lim inf new range}
\end{equation}
Our Theorem \ref{theorem 1.1} above generalizes the estimate (\ref{lim inf new range})
to shifted combinations of $\zeta_{\alpha}(s)$. Moreover, when $z=0$
and $c_{1}=1$, $\lambda_{1}=0$, $c_{2}=c_{3}=...=0$, it actually
extends (\ref{lim inf new range}) also to the range $0<\alpha\leq4$.
As remarked in \cite{rysc_I}, when $\alpha$ is any integer greater than
$3$, it is known that $N_{\alpha}(0,T)\asymp T$ \cite{Siegel_Contributions}
and so, in this case, (\ref{lim inf new range}) does not say anything
new. However, for non-integral $\alpha$, this result seems to be novel. 

\end{remark}

\begin{remark}\label{remark 1.4}
Analogues of our Theorem \ref{theorem 1.1} can be proved when $\zeta_{\alpha}(s)$ is replaced by any of the meromorphic Dirichlet series mentioned in
\cite{RYCE}. For example, it is still valid if we replace $\zeta_{\alpha}(s)$
by an Epstein zeta function $\zeta(s,Q)$ attached to an integral
quadratic form $Q(m,n)=Am^{2}+Bmn+Cn^{2}$ such that $\sqrt{4AC-B^{2}}\equiv2\mod4$. Of course, the admissible region for $z$ in this case would have to depend on the discriminant of $Q$.

\end{remark}

\begin{remark}
A lower bound for $d$ appearing in (\ref{lim sup estimate}) can
be explicitly calculated. A simple numerical computation (see the
proof of Lemma \ref{lemma 4.2} below) gives the value $d=\frac{1}{36\pi}$ as
admissible for $d$. 

\end{remark}

\begin{remark}
Since the method of proof of Theorem \ref{theorem 1.2} works well for any entire
Dirichlet series, there is an analogue of this result for some Dirichlet
$L-$functions. See [\cite{RYCE}, Theorem 1.3] for details. Furthermore, we can use this method to prove (\ref{lim sup estimate}) for a general combination involving
shifts of $\zeta_{\alpha}(s)$. Since $s\left(s-\frac{\alpha}{2}\right)\eta_{\alpha}(s)$
is an entire function of $s$, then our proof of Theorem \ref{theorem 1.2} can be
applied to study the number of zeros of the function 
\begin{equation}
\sum_{j\neq0}c_{j}\left(s+i\lambda_{j}\right)\left(s+i\lambda_{j}-\frac{\alpha}{2}\right)\,\eta_{\alpha}\left(s+i\lambda_{j}\right)\,\left\{ _{1}F_{1}\left(\frac{\alpha}{2}-s-i\lambda_{j};\,\frac{\alpha}{2};\,\frac{z^{2}}{4}\right)+\,_{1}F_{1}\left(\frac{\alpha}{2}-\overline{s}+i\lambda_{j};\,\frac{\alpha}{2};\,\frac{{z}^{2}}{4}\right)\right\} ,\label{entire combination zeta alpha remark}
\end{equation}
whenever $z$ satisfies (\ref{condition first rectangle}). If $\tilde{N}_{\alpha,z}(T)$ denotes the number of zeros of the above combination that can be written in the form $s=\frac{\alpha}{4}+it,$ $0\leq t\leq T$, then the following estimate takes place
\begin{equation}
\limsup_{T\rightarrow\infty}\,\frac{\tilde{N}_{\alpha,z}(T)}{\sqrt{T}}\geq\tilde{c},\label{lim sup zeta alpha case}
\end{equation}
for some $\tilde{c}>0$. In particular, when we reduce the combination
(\ref{entire combination zeta alpha remark}) to the case where $\lambda_{1}=0$
and $c_{j}=0$ for any $|j|\geq2$, we see that, besides (\ref{lim inf new range}),
we also have $N_{\alpha}(T)=\Omega\left(T^{\frac{1}{2}}\right)$.
\end{remark}

\begin{remark}
Using the properties of the slash operator, it is possible to establish
Theorem \ref{theorem 1.2} when $f(z)$ is a cusp form of weight $k$ on a congruence subgroup $\Gamma_{0}(N)$ with $N$ being a perfect square. We have
to put the additional condition that $\left(f|W_{N}\right)(z)=\pm f(z)$,
where $W_{N}$ is the Fricke involution. The same method also works
for half-integral weight cusp forms in $\Gamma_{0}(4N)$, $N$ being
a perfect square.
\end{remark}

\section{Lemmas for the proof of Theorem \ref{theorem 1.1}}
We start by recalling Stirling's formula for the Gamma function,
\begin{equation}
\Gamma(\sigma+it)=(2\pi)^{\frac{1}{2}}\,t^{\sigma+it-\frac{1}{2}}\,e^{-\frac{\pi t}{2}-it+\frac{i\pi}{2}(\sigma-\frac{1}{2})}\left(1+\frac{1}{12(\sigma+it)}+O\left(\frac{1}{t^{2}}\right)\right),\,\,\,\,t\rightarrow\infty,\label{Stirling exact form on Introduction}
\end{equation}
valid whenever $-\infty<\sigma_{1}\leq\sigma\leq\sigma_{2}<\infty$.
A similar formula can be written for $t<0$ as $t$ tends to $-\infty$
by using the fact that $\Gamma(\overline{s})=\overline{\Gamma(s)}$.
Of course, a direct consequence of this exact version is
\begin{equation}
|\Gamma(\sigma+it)|=(2\pi)^{\frac{1}{2}}\,|t|^{\sigma-\frac{1}{2}}\,e^{-\frac{\pi}{2}|t|}\left(1+O\left(\frac{1}{|t|}\right)\right),\,\,\,\,\,|t|\rightarrow\infty.\label{preliminary stirling}
\end{equation}

Let us now recall some basic facts about the Dirichet series $\zeta_{\alpha}(s)$.
It is well-known that the theta function $\vartheta_{3}(\tau):=\sum_{n\in\mathbb{Z}}e^{\pi in^{2}\tau}$
is a modular form of weight $\frac{1}{2}$ with a multiplier system
with respect to the theta group $\Gamma_{\theta}$ (see \cite{lagarias_reins},
p. 15-16 for details). Therefore, for any $\alpha>0$, $\vartheta_{3}^{\alpha}(\tau)$
is a modular form of weight $\alpha/2$ with a multiplier system on
the same group. By definition, $r_{\alpha}(n)$ are the Fourier coefficients
of the expansion of $\vartheta_{3}(\tau)$ at the cusp $i\infty$,
this is (cf. \cite{RYCE} and \cite{lagarias_reins} for a clearer explanation
of this expansion)
\begin{equation}
\vartheta_{3}^{\alpha}(ix):=\theta^{\alpha}(x)=1+\sum_{n=1}^{\infty}r_{\alpha}(n)\,e^{-\pi n\,x},\,\,\,\,\,x>0.\label{second definition varthet coefficients}
\end{equation}

\bigskip{}

The order of growth of $r_{\alpha}(n)$ as $n\rightarrow\infty$ is
determined by classical estimates due to Petersson and Lehner \cite{lagarias_reins, lehner}. These estimates show that
\begin{equation}
r_{\alpha}(n)\ll_{\alpha}\begin{cases}
n^{\alpha/2-1} & \alpha>4\\
n^{\alpha/2-1}\log(n) & \alpha=4\\
n^{\alpha/4} & 0<\alpha<4.
\end{cases}\label{estimate r alpha (n) useful}
\end{equation}
Thus, we can rigorously define the Dirichlet series (\ref{zeta alpha definition intro})
in the form
\begin{equation}
\zeta_{\alpha}(s):=\sum_{n=1}^{\infty}\frac{r_{\alpha}(n)}{n^{s}},\,\,\,\,\,\,\text{Re}(s)>\sigma_{\alpha}:=\begin{cases}
\frac{\alpha}{2} & \alpha\geq4\\
1+\frac{\alpha}{4} & 0<\alpha<4
\end{cases}.\label{sigma alpha definition in the Series zeta alpha}
\end{equation}
Mimicking Riemann's paper, one can easily show that $\zeta_{\alpha}(s)$
can be analytically continued to the entire complex plane as a meromorphic
function with a simple pole located at $s=\frac{\alpha}{2}$, whose
residue is $\text{Res}_{s=\alpha/2}\zeta_{\alpha}(s)=\frac{\pi^{\alpha/2}}{\Gamma(\alpha/2)}$.
Moreover, it satisfies Hecke's functional equation 
\begin{equation}
\eta_{\alpha}(s):=\pi^{-s}\Gamma(s)\,\zeta_{\alpha}(s)=\pi^{-\left(\frac{\alpha}{2}-s\right)}\Gamma\left(\frac{\alpha}{2}-s\right)\,\zeta_{\alpha}\left(\frac{\alpha}{2}-s\right):=\eta_{\alpha}\left(\frac{\alpha}{2}-s\right).\label{completed Dirichlet series definition in lemmata section}
\end{equation}
It also follows from the Phragm\'en-Lindel\"of principle that $\zeta_{\alpha}(s)$
obeys to the convex estimate
\begin{equation}
\zeta_{\alpha}\left(\sigma+it\right)\ll_{\alpha}|t|^{\sigma_{\alpha}-\sigma+\delta},\,\,\,\,\frac{\alpha}{2}-\sigma_{\alpha}-\delta<\sigma<\sigma_{\alpha}+\delta,\label{phragmen lindelof zeta alpha}
\end{equation}
for any $\delta>0$ and $\sigma_{\alpha}$ defined by (\ref{sigma alpha definition in the Series zeta alpha}).

In order to estimate the series given in (\ref{symmetric function F}),
we need an asymptotic formula for the confluent hypergeometric function
valid when $|s|\rightarrow\infty$. Following the reasoning in {[}\cite{Dixit_theta},
p. 379{]}, recall that the Whittaker function $M_{\lambda,\mu}(z)$
has the asymptotic formula, {[}\cite{NIST}, p.
341, eq. (13.21.1){]}
\begin{equation}
M_{\lambda,\mu}(z)=\frac{z^{1/4}}{\sqrt{\pi}}\lambda^{-\mu-\frac{1}{4}}\Gamma\left(2\mu+1\right)\,\cos\left(2\sqrt{\lambda z}-\frac{\pi}{4}-\mu\pi\right)+O\left(|\lambda|^{-\mu-\frac{3}{4}}\right)\label{Whittaker expansion}
\end{equation}
as $|\lambda|\rightarrow\infty$ and $z$ such that $|\arg(\lambda z)|<2\pi$.
Furthermore, since
\begin{equation}
M_{\lambda,\mu}(z)=z^{\mu+\frac{1}{2}}e^{-z/2}\,_{1}F_{1}\left(\mu-\lambda+\frac{1}{2};\,2\mu+1;\,z\right),\label{definition whitttaker}
\end{equation}
we see that, as $|t|\rightarrow\infty$ and fixed $z\in\mathbb{R}$,
the substitutions in (\ref{Whittaker expansion}) and (\ref{definition whitttaker})
give the bound
\begin{equation}
\left|_{1}F_{1}\left(\frac{\alpha}{4}-it;\,\frac{\alpha}{2};\,\frac{z^{2}}{4}\right)\right|=\left(\frac{|z|}{2}\right)^{-\frac{\alpha}{2}}e^{\frac{z^{2}}{8}}\left\{ \Gamma\left(\frac{\alpha}{2}\right)\sqrt{\frac{|z|}{2\pi}}\,|t|^{\frac{1-\alpha}{4}}\,\exp\left(\sqrt{\frac{|t|}{2}}\,|z|\right)+O\left(|t|^{-\frac{\alpha}{4}-\frac{1}{4}}\right)\right\} ,\,\,\,\,|t|\rightarrow\infty.\label{bound for confluent hypergeometric useful Rcomb}
\end{equation}
This estimate is more than enough to justify most of the steps in
this paper. By Stirling's formula and 
(\ref{phragmen lindelof zeta alpha}), one can see that
\begin{equation}
\left|\eta_{\alpha}\left(\frac{\alpha}{4}+it\right)\right|\ll_{\alpha}\,|t|^{A(\alpha)}\,e^{-\frac{\pi}{2}|t|},\,\,\,\,\,|t|\rightarrow\infty,\label{eta alpha convex estimates}
\end{equation}
where $A(\alpha)=\sigma_{\alpha}-\frac{1}{2}$, $\sigma_{\alpha}$
being given by (\ref{sigma alpha definition in the Series zeta alpha}).
When combined with (\ref{bound for confluent hypergeometric useful Rcomb}), the convex estimate (\ref{eta alpha convex estimates}) yields the bound 
\begin{equation}
\left|\tilde{F}_{z,\alpha}\left(\frac{\alpha}{4}+it\right)\right|\leq\sum_{j\neq0}\left|c_{j}\,\eta_{\alpha}\left(\frac{\alpha}{4}+i\,(t+\lambda_{j})\right)\text{Re}\left(\,_{1}F_{1}\left(\frac{\alpha}{4}-i\,(t+\lambda_{j});\,\frac{\alpha}{2};\,\frac{z^{2}}{4}\right)\right)\right|\ll_{\alpha,z}C_{\lambda}\sum_{j=1}^{\infty}|c_{j}|\,|t|^{B(\alpha)}\,e^{-\frac{\pi}{2}|t|+|z|\sqrt{|t|}},\label{estimate simple-1}
\end{equation}
where we have used the fact that $\left(\lambda_{j}\right)_{j\in\mathbb{N}}$
is a bounded sequence and $\left(c_{j}\right)_{j\in\mathbb{N}}\in\ell^{1}$.
The term $C_{\lambda}$ only stands for a positive constant which
depends on the bounds of the sequence $\left(\lambda_{j}\right)_{j\in\mathbb{N}}$.

Next, let us note that an explicit way of writing the first equality
in (\ref{Dixit integral formula jacobi psi}) is
\begin{align}
1+2^{\frac{\alpha}{2}-1}\,\Gamma\left(\frac{\alpha}{2}\right)\,\left(\sqrt{\pi x}\,z\right)^{1-\frac{\alpha}{2}}\,\sum_{n=1}^{\infty}r_{\alpha}(n)\,n^{\frac{1}{2}-\frac{\alpha}{4}}\,e^{-\pi n\,x}\,J_{\frac{\alpha}{2}-1}\left(\sqrt{\pi n\,x}\,z\right)\nonumber \\
=\frac{e^{-\frac{z^{2}}{4}}}{x^{\alpha/2}}\left\{ 1+2^{\frac{\alpha}{2}-1}\Gamma\left(\frac{\alpha}{2}\right)\,\left(\sqrt{\frac{\pi}{x}}\,z\right)^{1-\frac{\alpha}{2}}\sum_{n=1}^{\infty}r_{\alpha}(n)\,n^{\frac{1}{2}-\frac{\alpha}{4}}\,e^{-\frac{\pi n}{x}}\,I_{\frac{\alpha}{2}-1}\left(\sqrt{\frac{\pi n}{x}}\,z\right)\right\} ,\label{making point summation formula}
\end{align}
whenever $\text{Re}(x)>0$ and $z\in\mathbb{C}$. Since the summation
formula (\ref{making point summation formula}) transforms a generalized
$\psi-$function involving the Bessel functions of the first kind,
it will be useful in our next argument to have a bound for the modified
Bessel function, $I_{\nu}(z)$, $\nu>-1$. In the previous paper {[}\cite{RYCE},
p. 42{]} we have used the famous Hankel expansion for $I_{\nu}(z)$,
$|z|\rightarrow\infty$. However, for our alternative argument we
need a bound for $|I_{\nu}(z)|$ that is somewhat uniform in $\nu$
and $z$.

The following simple bound will be useful: first, let us note that,
if $\nu\geq0$ and $k\in\mathbb{N}_{0}$,
\[
\Gamma(k+\nu+1)=\Gamma(\nu+1)(\nu+1)...(\nu+k)\geq k!\,\Gamma(\nu+1),
\]
which gives, after the use of the power series of $I_{\nu}(z)$,
\begin{equation}
|I_{\nu}(z)|\leq\left(\frac{|z|}{2}\right)^{\nu}\,\sum_{k=0}^{\infty}\frac{|z|^{2k}}{2^{2k}k!\,\Gamma(k+\nu+1)}\leq\frac{(|z|/2)^{\nu}}{\Gamma(\nu+1)}\,\sum_{k=0}^{\infty}\frac{|z|^{2k}}{2^{2k}(k!)^{2}}\leq\frac{(|z|/2)^{\nu}}{\Gamma(\nu+1)}\left\{ \sum_{k=0}^{\infty}\frac{|z|^{k}}{2^{k}k!}\right\} ^{2}=\left(\frac{|z|}{2}\right)^{\nu}\frac{e^{|z|}}{\Gamma(\nu+1)},\,\,\,\,\,\nu\geq0.\label{first inequaliuty nu >0}
\end{equation}
If, on the other hand, we want to extend (\ref{first inequaliuty nu >0})
to $-1<\nu<0$, we begin to note that, for $k\geq1$,
\[
\Gamma(k+\nu+1)=\Gamma(\nu+2)\,(\nu+2)...(\nu+k)\geq(k-1)!\,\Gamma(\nu+2).
\]
Hence, if $-1<\nu<0$,
\begin{equation}
|I_{\nu}(z)| \leq\left(\frac{|z|}{2}\right)^{\nu}\,\sum_{k=0}^{\infty}\frac{|z|^{2k}}{2^{2k}k!\,\Gamma(k+\nu+1)}\leq\frac{\left(|z|/2\right)^{\nu}}{\Gamma(\nu+1)}+\frac{\left(|z|/2\right)^{\nu+2}}{\Gamma(\nu+2)}\,\sum_{k=0}^{\infty}\frac{|z|^{2k}}{2^{2k}(k!)^{2}}<\left(\frac{|z|}{2}\right)^{\nu}\frac{e^{|z|}}{\Gamma(\nu+2)}\left(1+\frac{|z|^{2}}{4}\right).\label{second inequality -1 <nu <0}
\end{equation}

To proceed, let us recall some functions  and facts already mentioned at the introduction. From this point on, we will let $\alpha>0$ and $r_{\alpha}(n)$ be defined as the coefficients
of the $q-$expansion of $\theta^{\alpha}(x)$. For $\text{Re}(x)>0$
and $z\in\mathbb{C}$, we let $\psi_{\alpha}(x,z)$ denote the analogue
of Jacobi's $\psi-$function,
\begin{equation}
\psi_{\alpha}(x,z)=2^{\frac{\alpha}{2}-1}\,\Gamma\left(\frac{\alpha}{2}\right)\,\left(\sqrt{\pi x}\,z\right)^{1-\frac{\alpha}{2}}\,\sum_{n=1}^{\infty}r_{\alpha}(n)\,n^{\frac{1}{2}-\frac{\alpha}{4}}\,e^{-\pi n\,x}\,J_{\frac{\alpha}{2}-1}(\sqrt{\pi\,n\,x}\,z).\label{Definition Jacobi in section lemmas}
\end{equation}
As stated above, the transformation formula for $\psi_{\alpha}(x,z)$,
(\ref{making point summation formula}), will play a major role in
this paper. We shall also need some auxiliary summation formulas given
in {[}\cite{RYCE}, pp. 37-40{]}. Analogously to $r_{\alpha}(n)$, we
can consider positive powers of the theta function
\begin{equation}
\vartheta_{2}(\tau)=2\,\sum_{n=0}^{\infty}e^{\pi i\tau\left(n+\frac{1}{2}\right)^{2}},\,\,\,\,\text{Im}(\tau)>0,\label{second theta nulll werte}
\end{equation}
from the Fourier expansion at the cusp $\tau=i\infty$. We can create
a new arithmetical function $\tilde{r}_{\alpha}(m)$ as the coefficient
coming from the expansion (see {[}\cite{RYCE}, p.36, eq. (3.5){]} for
details) 
\begin{align}
\vartheta_{2}^{\alpha}(ix) & :=\theta_{2}^{\alpha}(x)=\left(2\,\sum_{n=0}^{\infty}e^{-\pi x\left(n+\frac{1}{2}\right)^{2}}\right)^{\alpha}=\left(2\,e^{-\frac{\pi x}{4}}+2\,\sum_{n=1}^{\infty}e^{-\pi x(n+\frac{1}{2})^{2}}\right)^{\alpha}\nonumber \\
 & =2^{\alpha}e^{-\frac{\pi\alpha x}{4}}\left(1+\,\sum_{n=1}^{\infty}e^{-\pi x\left(n^{2}+n\right)}\right)^{\alpha}=2^{\alpha}\,e^{-\frac{\pi\alpha x}{4}}\,\sum_{j=0}^{\infty}\left(\begin{array}{c}
\alpha\\
j
\end{array}\right)\,\left(\sum_{n=1}^{\infty}e^{-\pi x\left(n^{2}+n\right)}\right)^{j}\nonumber \\
 & :=\sum_{m=0}^{\infty}\tilde{r}_{\alpha}(m)\,e^{-\pi\left(m+\frac{\alpha}{4}\right)x},\,\,\,\,\,\,\,\,\,\,\,\text{Re}(x)>0.\label{definition ralpha tilde}
\end{align}
Analogously to $r_{\alpha}(n)$, the coefficients $\tilde{r}_{\alpha}(n)$
grow polynomially with $n$ (cf.{[}\cite{RYCE}, Lemma
3.1{]}). Note also that $\tilde{r}_{\alpha}(0):=2^{\alpha}$ by the
construction (\ref{definition ralpha tilde}). Using Jacobi's transformation
formula for $\vartheta_{2}(\tau)$, one can derive a summation formula
connecting $\tilde{r}_{\alpha}(n)$ with $r_{\alpha}(n)$. The next
lemma, given in {[}\cite{RYCE}, p. 40, Lemma 3.3{]}, establishes this
correspondence.

\begin{lemma}[\cite{RYCE}, Lemma 3.3] \label{lemma 2.1}
Let $r_{\alpha}(n)$ be the coefficients of the series expansion of
$\theta^{\alpha}(x)-1$, (\ref{r alpha definition iiiintro}), and $\tilde{r}_{\alpha}(n)$ be defined by
(\ref{definition ralpha tilde}). Then, for $\text{Re}(x)>0$ and
$y\in\mathbb{C}$, the following identity holds
\begin{align}
\sum_{n=1}^{\infty}(-1)^{n}\,r_{\alpha}(n)\,n^{\frac{1}{2}-\frac{\alpha}{4}}\,e^{-\pi nx}\,J_{\frac{\alpha}{2}-1}(y\,\sqrt{\pi\,n}) & =-\frac{y^{\frac{\alpha}{2}-1}\pi^{\frac{\alpha}{4}-\frac{1}{2}}}{2^{\frac{\alpha}{2}-1}\Gamma\left(\frac{\alpha}{2}\right)}\nonumber \\
+\,\frac{e^{-\frac{y^{2}}{4x}}}{x}\,\sum_{n=0}^{\infty}\tilde{r}_{\alpha}(n)\,\left(n+\frac{\alpha}{4}\right)^{\frac{1}{2}-\frac{\alpha}{4}}\,e^{-\frac{\pi}{x}\,\left(n+\frac{\alpha}{4}\right)} & I_{\frac{\alpha}{2}-1}\left(\frac{\sqrt{\pi(n+\frac{\alpha}{4})}\,y}{x}\right).\label{final formula for 1f1 theorem-1}
\end{align}
\end{lemma}

\bigskip{}

The previous lemma was necessary in \cite{RYCE} to prove that, for
$z$ under the condition (\ref{condition rectangle z zeta alpha})
and any $m\in\mathbb{N}_{0}$,
\begin{equation}
\frac{d^{m}}{d\omega^{m}}\left(1+\psi_{\alpha}\left(e^{2i\omega},z\right)\right)\rightarrow0,\,\,\,\,\,\text{as \,\,\,\ensuremath{\omega\rightarrow\frac{\pi}{4}^{-}.}}\label{tend to zero as omega tends to pi/4}
\end{equation}
This curious transformation formula will now play an essential role
in giving a more precise version of (\ref{tend to zero as omega tends to pi/4}).
This is done in the next lemma, where we shall estimate uniformly
the derivatives of the function $\psi_{\alpha}\left(ie^{-2iu},z\right)$,
$0<u<\frac{\pi}{4}$.

\begin{lemma}\label{lemma 2.2}
Let $\psi_{\alpha}(x,z)$ be the generalized Jacobi's $\psi-$function
defined by (\ref{Definition Jacobi final version}) and assume that
$z\in\mathbb{R}$ satisfies the condition
\begin{equation}
-\frac{1}{6}\sqrt{\frac{\pi\alpha}{2}}\leq z\leq\frac{1}{6}\sqrt{\frac{\pi\alpha}{2}}.\label{condition 1/6}
\end{equation}
Then there are two positive constants $A$ and $C$ (depending only
on $\alpha$) such that, for any $0<u<\frac{\pi}{4}$,
\begin{equation}
\left|\frac{d^{k}}{du^{k}}\left(1+\psi_{\alpha}\left(ie^{-2iu},z\right)\right)\right|<C\,\frac{2^{7k}k!}{u^{\frac{\alpha}{2}+k}}e^{-\frac{A}{u}},\,\,\,\,\alpha>2\label{bound derivative !}
\end{equation}
and
\begin{equation}
\left|\frac{d^{k}}{du^{k}}\left(1+\psi_{\alpha}\left(ie^{-2iu},z\right)\right)\right|<C\,\frac{2^{7k}k!}{u^{\frac{\alpha}{2}+k+2}}e^{-\frac{A}{u}},\,\,\,\,0\leq\alpha\leq2.\label{second bound derivative 0 alpha 2}
\end{equation}
\end{lemma}

\begin{proof}

Let us fix $0<u_{0}<\frac{\pi}{4}$. We shall prove (\ref{bound derivative !})
only, as the bound (\ref{second bound derivative 0 alpha 2}) can
be similarly deduced. At the end of the proof we just outline the
difference in getting (\ref{second bound derivative 0 alpha 2}).
The derivative of the function $\psi_{\alpha}\left(ie^{-2iu},z\right)$
at the point $u_{0}$ is taken by integrating along a circle with
center $u_{0}$ and having radius $\lambda u_{0}$, $0<\lambda<1$.
Throughout this proof, the parameter $\lambda$ will be arbitrary
up to the point where the condition (\ref{condition 1/6}) enters
in the argument. For any $w\in D_{\lambda u_{0}}(u_{0}):=\text{int}(C_{\lambda u_{0}}(u_{0}))$,
it is simple to check that $\text{Re}\left(ie^{-2iw}\right)>0$, so
that $\psi_{\alpha}\left(ie^{-2iw},z\right)$ is analytic inside the
circle $C_{\lambda u_{0}}(u_{0})$ {[}\cite{RYCE}, Corollary 2.2{]}.
Hence, by Cauchy's formula,
\begin{equation}
\left[\frac{d^{k}}{du^{k}}\left(1+\psi_{\alpha}\left(ie^{-2iu},z\right)\right)\right]_{u=u_{0}}=\frac{k!}{2\pi i}\,\intop_{C_{\lambda u_{0}}(u_{0})}\,\frac{1+\psi_{\alpha}\left(ie^{-2iw},z\right)}{(w-u_{0})^{k+1}}\,dw.\label{Cauchy's formula at beginning!}
\end{equation}
Now, let us bound trivially the integral above: to do it, it will
be crucial to employ the transformation formula (\ref{final formula for 1f1 theorem-1}).
Indeed, since $ie^{-2iw}=i+2e^{-iw}\sin(w)$ and, by definition (\ref{Definition Jacobi final version}), 
\[
1+\psi_{\alpha}(ie^{-2iw},z):=1+2^{\frac{\alpha}{2}-1}\,\Gamma\left(\frac{\alpha}{2}\right)\,\left(\sqrt{\pi}\,e^{i\frac{\pi}{4}}e^{-iw}\,z\right)^{1-\frac{\alpha}{2}}\,\sum_{n=1}^{\infty}(-1)^{n}r_{\alpha}(n)\,n^{\frac{1}{2}-\frac{\alpha}{4}}\,e^{-2\pi ne^{-iw}\sin(w)}\,J_{\frac{\alpha}{2}-1}\left(\sqrt{\pi n}\,e^{i\left(\frac{\pi}{4}-w\right)}\,z\right),
\]
an application of (\ref{final formula for 1f1 theorem-1})
with $x=2e^{-iw}\sin(w)$ and $y=e^{i\left(\frac{\pi}{4}-w\right)}\,z$
gives 
\begin{align}
1+\psi_{\alpha}(ie^{-2iw},z) & =2^{\frac{\alpha}{2}-1}\Gamma\left(\frac{\alpha}{2}\right)\left(\sqrt{\pi}\,e^{i\left(\frac{\pi}{4}-w\right)}z\right)^{1-\frac{\alpha}{2}}\nonumber \\
\times\,\frac{e^{-\frac{ie^{-iw}z^{2}}{8\sin(w)}}}{2e^{-iw}\sin(w)} & \,\sum_{n=0}^{\infty}\,\tilde{r}_{\alpha}(n)e^{-\frac{\pi i}{2}\left(n+\frac{\alpha}{4}\right)}\,\left(n+\frac{\alpha}{4}\right)^{\frac{1}{2}-\frac{\alpha}{4}}\,e^{-\frac{\pi}{2\tan(w)}\,\left(n+\frac{\alpha}{4}\right)}I_{\frac{\alpha}{2}-1}\left(\frac{\sqrt{\pi(n+\frac{\alpha}{4})\,}e^{\frac{i\pi}{4}}z}{2\sin(w)}\right).\label{The transformation formula in accordance to RYCE}
\end{align}

We shall bound $|1+\psi_{\alpha}\left(ie^{-2iw},z\right)|$ by estimating
the second expression on the right-hand side of (\ref{The transformation formula in accordance to RYCE}).
We will track each factor independently. Clearly, for any $w\in C_{\lambda u_{0}}(u_{0})$,
\[
\left|2^{\frac{\alpha}{2}-1}\,\Gamma\left(\frac{\alpha}{2}\right)\,\left(\sqrt{\pi}\,e^{i\frac{\pi}{4}}e^{-iw}z\right)^{1-\frac{\alpha}{2}}\right|\leq2^{\frac{\alpha}{2}-1}\Gamma\left(\frac{\alpha}{2}\right)\pi^{\frac{1}{2}-\frac{\alpha}{4}}e^{\frac{\pi}{4}\left|1-\frac{\alpha}{2}\right|}|z|^{1-\frac{\alpha}{2}},
\]
where we have noted that $\lambda u_{0}\leq\frac{\pi}{4}$. Next,
\begin{align*}
\left|\frac{e^{-\frac{ie^{-iw}z^{2}}{8\sin(w)}}}{2e^{-iw}\sin(w)}\right| & =\frac{e^{-\frac{z^{2}}{8}-\text{Im}(w)}\exp\left(\frac{z^{2}}{8}\text{Im}\left(\frac{1}{\tan(w)}\right)\right)}{2\sqrt{\sin^{2}\left(\text{Re}(w)\right)+\sinh^{2}\left(\text{Im}(w)\right)}}\\
 & =\frac{e^{-\frac{z^{2}}{8}-\text{Im}(w)}\exp\left(-\frac{z^{2}}{8}\cdot\frac{\sinh(2\text{Im}(w))}{\cosh(2\text{Im}(w))-\cos(2\text{Re}(w))}\right)}{2\sqrt{\sin^{2}\left(\text{Re}(w)\right)+\sinh^{2}\left(\text{Im}(w)\right)}}\\
 & \leq\frac{e^{-\frac{z^{2}}{8}+\lambda u_{0}}}{2\sin\left(\text{Re}(w)\right)}\,\exp\left(-\frac{z^{2}}{8}\cdot\frac{\sinh(2\text{Im}(w))}{\cosh(2\text{Im}(w))-\cos(2\text{Re}(w))}\right),
\end{align*}
where in the last step we just have used the fact that, for any $w\in C_{\lambda u_{0}}(u_{0})$,
$|\text{Im}(w)|\leq\lambda u_{0}$. Moreover, we have used the explicit
expression
\begin{equation}
\frac{1}{\tan(w)}=\frac{\sin(2\text{Re}(w))-i\,\sinh(2\text{Im}(w))}{\cosh(2\text{Im}(w))-\cos(2\text{Re}(w))}.\label{cotangent real and imaginary parts}
\end{equation}
Using now the elementary Jordan inequality $\sin(x)>\frac{2x}{\pi}$,
$0<x<\frac{\pi}{2}$, and the fact that $(1-\lambda)u_{0}\leq\text{Re}(w)\leq(1+\lambda)u_{0}<\frac{\pi}{2}$,
we obtain
\begin{align}
\left|\frac{e^{-\frac{ie^{-iw}z^{2}}{8\sin(w)}}}{2e^{-iw}\sin(w)}\right| & <\frac{\pi e^{-\frac{z^{2}}{8}+\lambda u_{0}}}{4\text{Re}(w)}\,\exp\left(-\frac{z^{2}}{8}\cdot\frac{\sinh(2\text{Im}(w))}{\cosh(2\text{Im}(w))-\cos(2\text{Re}(w))}\right)\nonumber \\
 & \leq\frac{\pi e^{-\frac{z^{2}}{8}+\lambda u_{0}}}{4(1-\lambda)u_{0}}\,\exp\left(-\frac{z^{2}}{8}\cdot\frac{\sinh(2\text{Im}(w))}{\cosh(2\text{Im}(w))-\cos(2\text{Re}(w))}\right),\,\,\,\,w\in C_{\lambda u_{0}}(u_{0}).\label{The bound on the first term}
\end{align}

We now address each term of the series on the right-hand side of (\ref{The transformation formula in accordance to RYCE}).
Since we first aim at proving (\ref{bound derivative !}), we assume
here that $\alpha>2$, so that $\frac{\alpha}{2}-1>0$. Using the
first uniform inequality for $I_{\nu}(z)$, (\ref{first inequaliuty nu >0}), we have that the term involving the Bessel function admits the bound
\begin{align}
\left|I_{\frac{\alpha}{2}-1}\left(\frac{\sqrt{\pi(n+\frac{\alpha}{4})\,}e^{\frac{i\pi}{4}}z}{2\sin(w)}\right)\right| & \leq\frac{\pi^{\frac{\alpha}{4}-\frac{1}{2}}\left(n+\frac{\alpha}{4}\right)^{\frac{\alpha}{4}-\frac{1}{2}}|z|^{\frac{\alpha}{2}-1}}{2^{\alpha-2}\Gamma\left(\frac{\alpha}{2}\right)|\sin(w)|^{\frac{\alpha}{2}-1}}\,\exp\left(\frac{\sqrt{\pi\left(n+\frac{\alpha}{4}\right)}\,|z|}{2\,|\sin(w)|}\right)\nonumber \\
 & \leq\frac{\pi^{\frac{\alpha}{4}-\frac{1}{2}}\left(n+\frac{\alpha}{4}\right)^{\frac{\alpha}{4}-\frac{1}{2}}|z|^{\frac{\alpha}{2}-1}}{2^{\alpha-2}\Gamma\left(\frac{\alpha}{2}\right)\sin^{\frac{\alpha}{2}-1}\left(\text{Re}(w)\right)}\,\exp\left(\frac{\sqrt{\frac{\pi}{2}\left(n+\frac{\alpha}{4}\right)}\,|z|}{\sqrt{\cosh(2\text{Im}(w))-\cos(2\text{Re}(w))}}\right)\nonumber \\
 & \leq\frac{\pi^{\frac{3\alpha}{4}-\frac{3}{2}}\left(n+\frac{\alpha}{4}\right)^{\frac{\alpha}{4}-\frac{1}{2}}|z|^{\frac{\alpha}{2}-1}}{2^{\frac{3\alpha}{2}-3}\Gamma\left(\frac{\alpha}{2}\right)\left(1-\lambda\right)^{\frac{\alpha}{2}-1}u_{0}^{\frac{\alpha}{2}-1}}\,\exp\left(\frac{\sqrt{\frac{\pi}{2}\left(n+\frac{\alpha}{4}\right)}\,|z|}{\sqrt{\cosh(2\text{Im}(w))-\cos(2\text{Re}(w))}}\right).\label{Final bound Bessel term}
\end{align}
In the second inequality above we have used the fact that $|\sin(w)|\geq|\sin\left(\text{Re}(w)\right)|$
and in the third inequality we have once more invoked Jordan's inequality
$\sin(x)>\frac{2x}{\pi}$, $0<x<\frac{\pi}{2}$, together with the
fact that $(1-\lambda)u_{0}\leq\text{Re}(w)\leq(1+\lambda)u_{0}<\frac{\pi}{2}$.
Returning to (\ref{The transformation formula in accordance to RYCE}) and recalling once more (\ref{cotangent real and imaginary parts}), from (\ref{The bound on the first term}) and  (\ref{Final bound Bessel term}) we deduce
\begin{equation}
\left|1+\psi_{\alpha}(ie^{-2iw},z)\right|<d_{\alpha}\,\frac{\pi^{\frac{\alpha}{2}}e^{-\frac{z^{2}}{8}+\lambda u_{0}}}{2^{\alpha}\left(1-\lambda\right)^{\frac{\alpha}{2}}u_{0}^{\frac{\alpha}{2}}}\,\sum_{n=0}^{\infty}|\tilde{r}_{\alpha}(n)|\,\exp\left\{ -\frac{\pi}{2}\frac{\sin\left(2\text{Re}(w)\right)}{\cosh\left(2\text{Im}(w)\right)-\cos\left(2\text{Re}(w)\right)}\,P_{w}\left(\sqrt{n+\frac{\alpha}{4}}\right)\right\} ,\label{first inequality in the combination}
\end{equation}
where $d_{\alpha}$ is some constant depending on $\alpha$ and $P_{w}\left(X\right)$
is the real-valued polynomial
\begin{equation}
P_{w}(X):=X^{2}-\sqrt{\frac{2}{\pi}}\frac{|z|}{\sin(2\text{Re}(w))}\,\sqrt{\cosh\left(2\text{Im}(w)\right)-\cos\left(2\text{Re}(w)\right)}\,X+\frac{z^{2}}{4\pi}\,\frac{\sinh(2\text{Im}(w))}{\sin(2\text{Re}(w))}.\label{Polynomial at last}
\end{equation}
The inequality (\ref{first inequality in the combination}) is valid for any $w\in C_{\lambda u_{0}}(u_{0})$. Taking out the first term of the series on (\ref{first inequality in the combination}) and using the fact that $\tilde{r}_{\alpha}(0):=2^{\alpha}$, we get from (\ref{first inequality in the combination})
\begin{align}
\left|1+\psi_{\alpha}(ie^{-2iw},z)\right| & <\frac{d_{\alpha}\,\pi^{\frac{\alpha}{2}}e^{-\frac{z^{2}}{8}+\frac{u_{0}}{2}}}{\left(1-\lambda\right)^{\frac{\alpha}{2}}u_{0}^{\frac{\alpha}{2}}}\exp\left[-\frac{\pi}{2}\frac{\sin\left(2\text{Re}(w)\right)}{\cosh\left(2\text{Im}(w)\right)-\cos\left(2\text{Re}(w)\right)}\,P_{w}\left(\frac{\sqrt{\alpha}}{2}\right)\right]\nonumber \\
 & \times\sum_{n=0}^{\infty}\left|\frac{\tilde{r}_{\alpha}(n)}{\tilde{r}_{\alpha}(0)}\right|\exp\left[-Q_{w}(n)\right].\label{bound for the modulus in fact}
\end{align}
The exponent showing up in the infinite series is explicitly
\[
Q_{w}(n):=\frac{\pi}{2}\frac{\sin\left(2\text{Re}(w)\right)}{\cosh\left(2\text{Im}(w)\right)-\cos\left(2\text{Re}(w)\right)}\left[P_{w}\left(\sqrt{n+\frac{\alpha}{4}}\right)-P_{w}\left(\frac{\sqrt{\alpha}}{2}\right)\right],
\]
where $P_{w}(X)$ is the polynomial given by (\ref{Polynomial at last}). We will show below that there is a choice of $\lambda$ such that, for any $z$ satisfying (\ref{condition 1/6}),  $Q_{w}(n)>0,\,\,\forall n\in\mathbb{N}$.  Since $|\tilde{r}_{\alpha}(n)|<A_{\alpha}\,n^{\frac{\alpha}{2}}$ (see [\cite{RYCE},
p. 38]), the infinite series that appears on the right-hand side of
(\ref{bound for the modulus in fact}) is convergent. The remaining part of our proof will be to find a uniform bound for it, with a constant depending on  $\lambda$ and $\alpha$.
We want to find an upper bound for
\[
\sum_{n=0}^{\infty}\left|\frac{\tilde{r}_{\alpha}(n)}{\tilde{r}_{\alpha}(0)}\right|\,\exp\left[-\frac{\pi}{2}\frac{\sin\left(2\text{Re}(w)\right)}{\cosh\left(2\text{Im}(w)\right)-\cos\left(2\text{Re}(w)\right)}\left\{ n-\sqrt{\frac{2}{\pi}}\,|z|\frac{\sqrt{\cosh\left(2\text{Im}(w)\right)-\cos\left(2\text{Re}(w)\right)}}{\sin(2\text{Re}(w))}\left(\sqrt{n+\frac{\alpha}{4}}-\frac{\sqrt{\alpha}}{2}\right)\right\} \right].
\]
Since $(1-\lambda)u_{0}\leq\text{Re}(w)\leq(1+\lambda)u_{0}$, we
know from Jordan's inequality 
\begin{equation*}
\cosh(2\text{Im}(w))-\cos(2\text{Re}(w)) \geq1-\cos\left(2(1-\lambda)u_{0}\right)=\intop_{0}^{2(1-\lambda)u_{0}}\sin(t)\,dt>\frac{2}{\pi}\,\intop_{0}^{2(1-\lambda)u_{0}}t\,dt=\frac{4(1-\lambda)^{2}}{\pi}u_{0}^{2},
\end{equation*}
which implies
\begin{equation}
\frac{1}{\sqrt{\cosh\left(2\text{Im}(w)\right)-\cos\left(2\text{Re}(w)\right)}}<\frac{\sqrt{\pi}}{2(1-\lambda)u_{0}}.\label{inequality for fraction cosh}
\end{equation}
On the other hand, we can find the upper bound 
\begin{align}
\cosh(2\text{Im}(w))-\cos(2\text{Re}(w)) & =\intop_{0}^{2\text{Im}(w)}\sinh(t)\,dt+\intop_{0}^{2\text{Re}(w)}\sin(t)\,dt\leq\intop_{0}^{2\text{Im}(w)}t\,e^{\frac{t^{2}}{6}}dt+\intop_{0}^{2\text{Re}(w)}t\,dt\nonumber \\
 & \leq2\text{Im}(w)^{2}\,e^{\frac{2}{3}\text{Im}(w)^{2}}+2\text{Re}(w)^{2}<4e^{\frac{2\lambda^{2}}{3}u_{0}^{2}}\left(1+\lambda\right)^{2}u_{0}^{2},\label{inequality first for cosh}
\end{align}
where we have used the fact that $(1-\lambda)u_{0}\leq\text{Re}(w)\leq(1+\lambda)u_{0}$,
$-\lambda u_{0}\leq\text{Im}(w)\leq\lambda u_{0}$. We have also used
the known inequality
\begin{equation}
\sinh(x)\leq x\,e^{\frac{x^{2}}{6}},\,\,\,\,x>0.\label{inequality sinh (x) important!}
\end{equation}
Since $(1-\lambda)u_{0}\leq\text{Re}(w)\leq(1+\lambda)u_{0}<\frac{\pi}{2}$,
another application of Jordan's inequality gives 
\begin{equation}
\sin\left(2\text{Re}(w)\right)>\begin{cases}
\frac{4}{\pi}\text{Re}(w), & 0<\text{Re}(w)<\frac{\pi}{4}\\
2-\frac{4}{\pi}\text{Re}(w), & \frac{\pi}{4}\leq\text{Re}(w)<\frac{\pi}{2}
\end{cases}\,\,\geq\frac{4}{\pi}\left(1-\lambda\right)u_{0}.\label{final inequality sin(2Re)}
\end{equation}

Thus, combining (\ref{inequality first for cosh}) with (\ref{final inequality sin(2Re)}) and recalling that $0<u_{0}<\frac{\pi}{4}$, we get the inequality
\begin{equation}
\frac{\sin(2\text{Re}(w))}{\cosh(2\text{Im}(w))-\cos(2\text{Re}(w))}>\frac{\left(1-\lambda\right)}{\pi e^{\frac{2\lambda^{2}}{3}u_{0}^{2}}\left(1+\lambda\right)^{2}u_{0}}>\frac{\left(1-\lambda\right)}{\pi e^{\frac{\pi^{2}\lambda^{2}}{24}}\left(1+\lambda\right)^{2}u_{0}}.\label{lower bound for the first term exponential}
\end{equation}
Thus, the infinite series on the right-hand side of (\ref{bound for the modulus in fact}) admits the bound
\begin{align*}
\sum_{n=0}^{\infty}\left|\frac{\tilde{r}_{\alpha}(n)}{\tilde{r}_{\alpha}(0)}\right|\,&\exp\left[-\frac{\pi}{2}\frac{\sin\left(2\text{Re}(w)\right)}{\cosh\left(2\text{Im}(w)\right)-\cos\left(2\text{Re}(w)\right)}\left\{ n-\frac{\sqrt{\frac{2}{\pi}}\,|z|\left(\sqrt{n+\frac{\alpha}{4}}-\frac{\sqrt{\alpha}}{2}\right)}{\sin(2\text{Re}(w))}\sqrt{\cosh\left(2\text{Im}(w)\right)-\cos\left(2\text{Re}(w)\right)}\right\} \right]\\
&<\sum_{n=0}^{\infty}\left|\frac{\tilde{r}_{\alpha}(n)}{\tilde{r}_{\alpha}(0)}\right|\,\exp\left[-\frac{1}{u_{0}}\left\{\frac{n(1-\lambda)}{2e^{\frac{\pi^{2}\lambda^{2}}{24}}\left(1+\lambda\right)^{2}}-\frac{\pi\,|z|}{2\sqrt{2}(1-\lambda)}\left(\sqrt{n+\frac{\alpha}{4}}-\frac{\sqrt{\alpha}}{2}\right)\right\} \right].
\end{align*}
Using the condition (\ref{condition 1/6}), one sees that the exponent in the previous expression is bounded by
\begin{align*}
\exp\left[-\frac{1}{u_{0}}\left\{ \frac{n(1-\lambda)}{2e^{\frac{\pi^{2}\lambda^{2}}{24}}\left(1+\lambda\right)^{2}}-\frac{\pi\,|z|}{2\sqrt{2}(1-\lambda)}\left(\sqrt{n+\frac{\alpha}{4}}-\frac{\sqrt{\alpha}}{2}\right)\right\} \right]\\
\leq\exp\left[-\frac{1}{u_{0}}\left\{ \frac{n(1-\lambda)}{2e^{\frac{\pi^{2}\lambda^{2}}{24}}\left(1+\lambda\right)^{2}}-\frac{\pi^{\frac{3}{2}}\sqrt{\alpha}}{24(1-\lambda)}\left(\sqrt{n+\frac{\alpha}{4}}-\frac{\sqrt{\alpha}}{2}\right)\right\} \right]\\
\leq\exp\left[-\frac{n}{u_{0}}\left\{ \frac{1-\lambda}{2e^{\frac{\pi^{2}\lambda^{2}}{24}}\left(1+\lambda\right)^{2}}-\frac{\pi^{\frac{3}{2}}}{24(1-\lambda)}\right\} \right],
\end{align*}
where in the last inequality we have used the mean value theorem for
the function $f(x)=\sqrt{x+\frac{\alpha}{4}}$. If we select the value $\lambda=0.01$, we have that  
\[
\frac{1-\lambda}{2e^{\frac{\pi^{2}\lambda^{2}}{24}}\left(1+\lambda\right)^{2}}-\frac{\pi^{\frac{3}{2}}}{24(1-\lambda)}>\frac{1}{4},
\]
and so, since $0<u_{0}<\frac{\pi}{4}$,
\begin{align}
\sum_{n=0}^{\infty}\left|\frac{\tilde{r}_{\alpha}(n)}{\tilde{r}_{\alpha}(0)}\right|\,\exp\left[-\frac{1}{u_{0}}\left\{ \frac{n(1-\lambda)}{2e^{\frac{\pi^{2}\lambda^{2}}{24}}\left(1+\lambda\right)^{2}}-\frac{\pi\,|z|}{2\sqrt{2}(1-\lambda)}\left(\sqrt{n+\frac{\alpha}{4}}-\frac{\sqrt{\alpha}}{2}\right)\right\} \right]\nonumber \\
\leq\sum_{n=0}^{\infty}\left|\frac{\tilde{r}_{\alpha}(n)}{\tilde{r}_{\alpha}(0)}\right|\,\exp\left[-\frac{n}{4u_{0}}\right]<\frac{A_{\alpha}}{2^{\alpha}}\sum_{n=0}^{\infty}n^{\frac{\alpha}{2}}\,\exp\left[-\frac{n}{\pi}\right]\leq M_{\alpha}.\label{bound for the series!}
\end{align}
Combining (\ref{bound for the series!}) with (\ref{bound for the modulus in fact}) and using the fact that $0<u_{0}<\frac{\pi}{4}$, the following bound holds (note that we are now selecting $\lambda$ as equal to $0.01$)
\begin{equation}
\left|1+\psi_{\alpha}(ie^{-2iw},z)\right|<\frac{d_{\alpha}M_{\alpha}e^{\frac{\pi}{8}}\,\pi^{\frac{\alpha}{2}}\,e^{-\frac{z^{2}}{8}}}{(2u_{0})^{\frac{\alpha}{2}}}\,\exp\left[-\frac{\pi}{2}\frac{\sin\left(2\text{Re}(w)\right)}{\cosh\left(2\text{Im}(w)\right)-\cos\left(2\text{Re}(w)\right)}\,P_{w}\left(\frac{\sqrt{\alpha}}{2}\right)\right],\label{intermediate bound}
\end{equation}
which is almost (\ref{bound derivative !}) with $k=0$. To conclude
the proof in this direction, we just need to bound uniformly the term
\[
\exp\left[-\frac{\pi}{2}\frac{\sin\left(2\text{Re}(w)\right)}{\cosh\left(2\text{Im}(w)\right)-\cos\left(2\text{Re}(w)\right)}\,P_{w}\left(\frac{\sqrt{\alpha}}{2}\right)\right],\,\,\,\,\,\,w\in C_{\lambda u_{0}}(u_{0})
\]
under the hypothesis on $z$ (\ref{condition 1/6}) and the
choice $\lambda=0.01$. This can be done by appealing to (\ref{inequality sinh (x) important!}),
(\ref{inequality for fraction cosh}) and (\ref{lower bound for the first term exponential}),
which give
\begin{align}
\exp&\left[-\frac{\pi}{2}\frac{\sin\left(2\text{Re}(w)\right)}{\cosh\left(2\text{Im}(w)\right)-\cos\left(2\text{Re}(w)\right)}\,P_{w}\left(\frac{\sqrt{\alpha}}{2}\right)\right]
<\exp\left[-\frac{\alpha}{8}\,\frac{\left(1-\lambda\right)}{e^{\frac{\pi^{2}\lambda^{2}}{24}}\left(1+\lambda\right)^{2}u_{0}}+\sqrt{\frac{\alpha}{2}}\,\frac{\pi\,|z|}{4(1-\lambda)u_{0}}+\frac{\pi z^{2}\lambda e^{\frac{\pi^{2}\lambda^{2}}{24}}}{16(1-\lambda)^{2}u_{0}}\right]\nonumber \\
&\leq\exp\left[-\frac{\alpha}{u_{0}}\left\{ \frac{1-\lambda}{8e^{\frac{\pi^{2}\lambda^{2}}{24}}\left(1+\lambda\right)^{2}}-\frac{\pi^{\frac{3}{2}}}{48(1-\lambda)}-\frac{\pi^{2}\lambda e^{\frac{\pi^{2}\lambda^{2}}{24}}}{1152(1-\lambda)^{2}}\right\} \right].\label{bound isolated exponential term!}
\end{align}
However, for $\lambda=0.01$, a numerical confirmation gives
\[
\frac{1-\lambda}{8e^{\frac{\pi^{2}\lambda^{2}}{24}}\left(1+\lambda\right)^{2}}-\frac{\pi^{\frac{3}{2}}}{48(1-\lambda)}-\frac{\pi^{2}\lambda e^{\frac{\pi^{2}\lambda^{2}}{24}}}{1152(1-\lambda)^{2}}>0.004,
\]
and so, returning to (\ref{intermediate bound}), we get from the previous inequality that
\begin{equation}
\left|1+\psi_{\alpha}(ie^{-2iw},z)\right|<\frac{d_{\alpha}M_{\alpha}e^{\frac{\pi}{8}}\,\pi^{\frac{\alpha}{2}}e^{-\frac{z^{2}}{8}}}{(2u_{0})^{\frac{\alpha}{2}}}\exp\left[-\frac{0.04\,\alpha}{u_{0}}\right]=\frac{C}{u_{0}^{\alpha/2}}\,e^{-\frac{A}{u_{0}}}\label{the case k=00003D0 yey}
\end{equation}
where $C$ and $A:=0.04\,\alpha$ only depend on $\alpha$. This proves
(\ref{bound derivative !}) for $k=0$. For the remaining cases, let
us invoke Cauchy's integral formula (\ref{Cauchy's formula at beginning!})
and use (\ref{the case k=00003D0 yey}) to get
\begin{equation}
\left|\left[\frac{d^{k}}{du^{k}}\psi_{\alpha}\left(ie^{-2iu},z\right)\right]_{u=u_{0}}\right| \leq\frac{k!}{2\pi\left(\lambda u_{0}\right)^{k+1}}\,\intop_{C_{\lambda u_{0}}(u_{0})}\,|1+\psi_{\alpha}\left(ie^{-2iw},z\right)|\,|dw|<C\,\frac{k!}{u_{0}^{\frac{\alpha}{2}+k}\lambda^{k}}e^{-\frac{A}{u_{0}}}.\label{steps to get bound derivativeees}
\end{equation}
Finally, under the choice $\lambda=0.01$, we know $\frac{1}{\lambda}<2^{7}$, 
and so we derive our estimate (\ref{bound derivative !}) in its final
form.

\bigskip{}

To get the second formula (\ref{second bound derivative 0 alpha 2}),
we use the bound (\ref{second inequality -1 <nu <0}) for $I_{\nu}(z)$, $-1<\nu<0$, and we invoke the same kind of inequalities as in (\ref{Final bound Bessel term}) to obtain
\begin{align}
\left|I_{\frac{\alpha}{2}-1}\left(\frac{\sqrt{\pi(n+\frac{\alpha}{4})\,}e^{\frac{i\pi}{4}}z}{2\sin(w)}\right)\right| & <\frac{\pi^{\frac{\alpha}{4}-\frac{1}{2}}\left(n+\frac{\alpha}{4}\right)^{\frac{\alpha}{4}-\frac{1}{2}}|z|^{\frac{\alpha}{2}-1}}{2^{\alpha-2}\Gamma\left(\frac{\alpha}{2}+1\right)|\sin(w)|^{\frac{\alpha}{2}-1}}\exp\left(\frac{\sqrt{\frac{\pi}{2}\left(n+\frac{\alpha}{4}\right)}\,|z|}{\sqrt{\cosh(2\text{Im}(w))-\cos(2\text{Re}(w))}}\right)\left(1+\frac{\pi^{3}(n+\frac{\alpha}{4})|z|^{2}}{64(1-\lambda)^{2}u_{0}^{2}}\right)\nonumber \\
 & \leq\mathcal{A}\,\frac{\pi^{\frac{\alpha}{4}-\frac{1}{2}}\left(n+\frac{\alpha}{4}\right)^{\frac{\alpha}{4}+\frac{1}{2}}|z|^{\frac{\alpha}{2}+1}}{2^{\frac{3\alpha}{2}-3}\Gamma\left(\frac{\alpha}{2}+1\right)\left(1-\lambda\right)^{\frac{\alpha}{2}+1}u_{0}^{\frac{\alpha}{2}+1}}\exp\left(\frac{\sqrt{\frac{\pi}{2}\left(n+\frac{\alpha}{4}\right)}\,|z|}{\sqrt{\cosh(2\text{Im}(w))-\cos(2\text{Re}(w))}}\right),\label{bound for other range of alpha}
\end{align}
for some absolutely large constant $\mathcal{A}$. There is no formal
difference between (\ref{bound for other range of alpha}) and (\ref{Final bound Bessel term}),
the only exception being in the power of $1/u_{0}$. Thus, the same
argument carries through with an extra factor of $1/u_{0}^{2}$. This
gives (\ref{second bound derivative 0 alpha 2}) and finishes the
proof. 
\end{proof}

\begin{remark}\label{remark 2.1}
Of course, the condition (\ref{condition first rectangle}) may be
improved if we take a more careful choice of the parameter $\lambda$
in the previous proof. Moreover, when $\nu>-\frac{1}{2}$, the uniform
bounds (\ref{first inequaliuty nu >0}) and (\ref{second inequality -1 <nu <0})
can be replaced by the better inequality,
\begin{equation}
\left|I_{\nu}(z)\right|\leq\left(\frac{|z|}{2}\right)^{\nu}\frac{e^{|\text{Re}(z)|}}{\Gamma(\nu+1)},\,\,\,\,\,\nu>-\frac{1}{2},\label{Bound better for Bessel with reeeeal parrrt}
\end{equation}
which can be obtained via the integral representation {[}\cite{NIST},
p. 252, eq. (10.32.2){]}
\[
I_{\nu}\left(z\right)=\frac{\left(z/2\right)^{\nu}}{\sqrt{\pi}\Gamma\left(\nu+\frac{1}{2}\right)}\,\intop_{-1}^{1}\left(1-t^{2}\right)^{\nu-\frac{1}{2}}e^{zt}\,dt,\,\,\,\,\,\nu>-\frac{1}{2}.
\]
Using (\ref{Bound better for Bessel with reeeeal parrrt}) instead
of (\ref{first inequaliuty nu >0}) in the proof above, one may enlarge
the interval $\left[-\frac{1}{6}\sqrt{\frac{\pi\alpha}{2}},\frac{1}{6}\sqrt{\frac{\pi\alpha}{2}}\right]$
in the condition (\ref{condition 1/6}) to $\left[-\frac{1}{4}\sqrt{\frac{\pi\alpha}{2}},\frac{1}{4}\sqrt{\frac{\pi\alpha}{2}}\right]$
when $\alpha\geq1$. Therefore, there are some cases where technical
improvements lead to stronger results than those stated in our Theorems \ref{theorem 1.1} and \ref{theorem 1.2}.
\end{remark}

We now give another lemma that will be crucial in the proof of our
result. 

\begin{lemma} \label{lemma 2.3}
Let $h:\,\mathbb{C}\longmapsto\mathbb{C}$ be an analytic function.
If $z\in\left[-\frac{1}{6}\sqrt{\frac{\pi\alpha}{2}},\frac{1}{6}\sqrt{\frac{\pi\alpha}{2}}\right]$
then, for every $p\in\mathbb{N}_{0}$, the following relation holds
\begin{equation}
\frac{d^{2p}}{du^{2p}}\left\{ h(u)\,\left(e^{-z^{2}/8}+e^{z^{2}/8}\,\psi_{\alpha}\left(ie^{-2iu},z\right)\right)\right\} =-2\sinh\left(\frac{z^{2}}{8}\right)h^{(2p)}(u)+\mathcal{A}_{\alpha,p}(u,z),\label{bound composite derrrivative}
\end{equation}
where, for any $0<u<\frac{\pi}{4}$, $\mathcal{A}_{\alpha,p}(u,z)$
satisfies the inequalities
\begin{equation}
|\mathcal{A}_{\alpha,p}(u,z)|<D\frac{2^{14p}(2p)!e^{-\frac{A}{u}}}{u^{\frac{\alpha}{2}+2p}}\left\Vert h\right\Vert _{L^{\infty}(C_{1}(0))},\,\,\,\,\alpha>2,\label{bound for Aup}
\end{equation}
\begin{equation}
|\mathcal{A}_{\alpha,p}(u,z)|<D\frac{2^{14p}(2p)!e^{-\frac{A}{u}}}{u^{\frac{\alpha}{2}+2+2p}}\left\Vert h\right\Vert _{L^{\infty}(C_{1}(0))},\,\,\,\,0\leq\alpha\leq2,\label{bound for Aup for other range}
\end{equation}
with $A$ and $D$ depending only on $\alpha$.
\end{lemma}
\begin{proof}
First observe that
\begin{align}
\frac{d^{2p}}{du^{2p}}\left\{ h(u)\,\left(e^{-z^{2}/8}+e^{z^{2}/8}\,\psi_{\alpha}\left(ie^{-2iu},z\right)\right)\right\}  & =-2\sinh\left(\frac{z^{2}}{8}\right)h^{(2p)}(u)+e^{\frac{z^{2}}{8}}\,\frac{d^{2p}}{du^{2p}}\,\left\{ h(u)\,\left(1+\psi_{\alpha}\left(ie^{-2iu},z\right)\right)\right\} \nonumber \\
 & :=-2\sinh\left(\frac{z^{2}}{8}\right)h^{(2p)}(u)+\mathcal{A}_{\alpha,p}(u,z).\label{formula that defines A alpha, p}
\end{align}
By the well-known Cauchy's estimates for the coefficients
of analytic functions and (\ref{bound derivative !}), we obtain 
\begin{align*}
\left|e^{\frac{z^{2}}{8}}\,\frac{d^{2p}}{du^{2p}}\,\left\{ h(u)\,\left(1+\psi_{\alpha}\left(ie^{-2iu},z\right)\right)\right\} \right| & \leq e^{\frac{\pi\alpha}{576}}\,\sum_{k=0}^{2p}\left(\begin{array}{c}
2p\\
k
\end{array}\right)\,\left|h^{(2p-k)}(u)\right|\,\left|\frac{d^{k}}{du^{k}}\left(1+\psi_{\alpha}\left(ie^{-2iu},z\right)\right)\right|\\
 & \leq e^{\frac{\pi\alpha}{576}}\left\Vert h\right\Vert _{L^{\infty}(D(0,1))}\,\sum_{k=0}^{2p}\left(\begin{array}{c}
2p\\
k
\end{array}\right)(2p-k)!\,\left|\frac{d^{k}}{du^{k}}\left(1+\psi_{\alpha}\left(ie^{-2iu},z\right)\right)\right|\\
 & <Ce^{\frac{\pi\alpha}{576}}\,\left\Vert h\right\Vert _{L^{\infty}(D(0,1))}\frac{(2p)!e^{-\frac{A}{u}}}{u^{\frac{\alpha}{2}}}\,\sum_{k=1}^{2p}\frac{2^{7k}}{u^{k}}\\
 & <C^{\prime}\,\left\Vert h\right\Vert _{L^{\infty}(C_{1}(0))}\frac{(2p)!e^{-\frac{A}{u}}}{u^{\frac{\alpha}{2}+2p}}\,\frac{(128)^{2p+1}-u^{2p+1}}{128-u}\\
 & <C^{\prime\prime}\,\left\Vert h\right\Vert _{L^{\infty}(C_{1}(0))}\frac{2^{14p}(2p)!e^{-\frac{A}{u}}}{u^{\frac{\alpha}{2}+2p}}
\end{align*}
where we have invoked the maximum modulus principle.  This completes the proof
of (\ref{bound for Aup}). The proof of (\ref{bound for Aup for other range}) is analogous. 
\end{proof}

\bigskip{}

To conclude this section, let us recall the useful notation introduced
in \cite{DRRZ}. From now on, it will be convenient to write the terms of
the sequence $\left(\lambda_{j}\right)_{j\in\mathbb{N}}$ in polar
coordinates 
\begin{equation}
\frac{i\alpha}{2}-2\lambda_{j}:=r_{j}\,e^{i\theta_{j}},\,\,\,0<\theta_{j}<\pi.\label{polar coordinates}
\end{equation}
According to our assumption \ref{assumption 1}, we need to extend this definition to negative
indices: since $\lambda_{-j}=-\lambda_{j}$, then we have that $r_{-j}=r_{j}$
and $\theta_{-j}=\pi-\theta_{j}$.

Using these polar coordinates, we shall introduce a simple lemma which
gives an integral representation for the moments of $\tilde{F}_{z,\alpha}\left(\frac{\alpha}{4}+it\right)$.
Its proof follows closely the steps given in {[}\cite{RYCE}, pp. 44-45{]},
so we shall omit some details.

\begin{lemma}\label{lemma 2.4}
Let $-\frac{\pi}{4}<\omega<\frac{\pi}{4}$, $p\in\mathbb{N}_{0}$
and $\tilde{F}_{z,\alpha}\left(\frac{\alpha}{4}+it\right)$ be the
function defined by (\ref{function defining coooombinations}). Then
the following integral representation holds
\begin{align}
\intop_{0}^{\infty}t^{2p}\,\tilde{F}_{z,\alpha}\left(\frac{\alpha}{4}+it\right)\,\cosh\left(2\omega t\right)\,dt & =-\frac{8\pi}{2^{2p}}\,\sum_{j=1}^{\infty}c_{j}r_{j}^{2p}\left[\cos\left(2p\theta_{j}\right)\,\cos\left(\frac{\omega\alpha}{2}\right)\cosh\left(2\omega\lambda_{j}\right)+\sin\left(2p\theta_{j}\right)\,\sin\left(\frac{\omega\alpha}{2}\right)\sinh\left(2\omega\lambda_{j}\right)\right]\nonumber \\
+\frac{2\pi\,e^{z^{2}/8}}{2^{2p}}\,\text{Re} & \left(\frac{d^{2p}}{d\omega^{2p}}\left\{ \sum_{j\neq0}c_{j}e^{\frac{i\omega\alpha}{2}-2\omega\lambda_{j}}\left(e^{-z^{2}/8}+e^{z^{2}/8}\,\psi_{\alpha}\left(e^{2i\omega},z\right)\right)\right\} \right),\label{integral representation lemma 4}
\end{align}
where $\psi_{\alpha}(x,z)$ is the generalized Jacobi's $\psi-$function
(\ref{Definition Jacobi final version}) and $(r_{j},\theta_{j})$
are the polar coordinates attached to $\left(\lambda_{j}\right)_{j\in\mathbb{N}}$
defined by (\ref{polar coordinates}).
\end{lemma}

\begin{proof}
We start by using the integral representation given at the introduction
(\ref{equation in the setting of zeta alpha}), replacing there $x$
by $e^{2i\omega}$, $-\frac{\pi}{4}<\omega<\frac{\pi}{4}$: we obtain
the formula
\begin{equation}
\frac{1}{2\pi}\,\intop_{-\infty}^{\infty}\eta_{\alpha}\left(\frac{\alpha}{4}+it\right)\,_{1}F_{1}\left(\frac{\alpha}{4}+it;\,\frac{\alpha}{2};\,-\frac{z^{2}}{4}\right)\,e^{2\omega t}\,dt=e^{\frac{i\omega\alpha}{2}}\,\psi_{\alpha}\left(e^{2i\omega},z\right)-e^{-\frac{i\omega\,\alpha}{2}}e^{-z^{2}/4}.\label{formula with omega}
\end{equation}
Using Kummer's identity {[}\cite{roy}, p. 191, eq. (4.1.11){]},
\begin{equation}
_{1}F_{1}\left(a;\,c;\,x\right)=e^{x}\,_{1}F_{1}\left(c-a;\,c;\,-x\right),\label{Kummer confluent transformation}
\end{equation}
an equivalent version of (\ref{formula with omega}) reads
\begin{equation}
\frac{e^{-z^{2}/8}}{2\pi}\,\intop_{-\infty}^{\infty}\eta_{\alpha}\left(\frac{\alpha}{4}+it\right)\,_{1}F_{1}\left(\frac{\alpha}{4}-it;\,\frac{\alpha}{2};\,\frac{z^{2}}{4}\right)\,e^{2\omega t}\,dt=e^{\frac{i\omega\alpha}{2}}\,e^{z^{2}/8}\,\psi_{\alpha}\left(e^{2i\omega},z\right)-e^{-\frac{i\omega\,\alpha}{2}}e^{-z^{2}/8}.\label{after kummer in the whole proof}
\end{equation}
In {[}\cite{RYCE}, p. 43, eq. (3.42){]}, we have used (\ref{after kummer in the whole proof})
and the estimates (\ref{estimate simple-1}) to obtain the following
formula
\begin{align*}
\intop_{-\infty}^{\infty}t^{2p}\,F_{z,\alpha}\left(\frac{\alpha}{4}+it\right)\,e^{2\omega t}\,dt & =-\frac{8\pi}{2^{2p}}\,\sum_{j=1}^{\infty}c_{j}\,r_{j}^{2p}e^{-2\omega\lambda_{j}}\,\cos\left(2p\,\theta_{j}+\frac{\omega\alpha}{2}\right)\\
+\frac{4\pi}{2^{2p}}\,\text{Re} & \left(e^{z^{2}/8}\,\frac{d^{2p}}{d\omega^{2p}}\left\{ \sum_{j=1}^{\infty}c_{j}e^{\frac{i\omega\alpha}{2}-2\omega\lambda_{j}}\left(e^{-z^{2}/8}+e^{z^{2}/8}\,\psi_{\alpha}\left(e^{2i\omega},z\right)\right)\right\} \right),
\end{align*}
where $0<\omega<\frac{\pi}{4}$, $z\in\mathbb{C}$, $\psi_{\alpha}(x,z)$
is the generalized theta function (\ref{Definition Jacobi final version})
and $F_{z,\alpha}(s)$ is the function given by (cf. (\ref{function defining coooombinations-1})
above)
\[
F_{z,\alpha}(s):=\sum_{j=1}^{\infty}c_{j}\,\eta_{\alpha}\left(s+i\lambda_{j}\right)\,\left\{ _{1}F_{1}\left(\frac{\alpha}{2}-s-i\lambda_{j};\,\frac{\alpha}{2};\,\frac{z^{2}}{4}\right)+\,_{1}F_{1}\left(\frac{\alpha}{2}-\overline{s}+i\lambda_{j};\,\frac{\alpha}{2};\,\frac{\overline{z}^{2}}{4}\right)\right\} .
\]
According to assumption \ref{assumption 1}, our symmetric shifted combination,
$\tilde{F}_{z,\alpha}(s)$, has the same expression as $F_{z,\alpha}(s)$
with an additional extension to the negative integers (see (\ref{symmetric function F}) above). Therefore, assuming also that $z\in\mathbb{R}$ (see assumption \ref{assumption 2}), we immediately find that
\begin{align}
\intop_{-\infty}^{\infty}t^{2p}\,\tilde{F}_{z,\alpha}\left(\frac{\alpha}{4}+it\right)\,e^{2\omega t}\,dt & =-\frac{8\pi}{2^{2p}}\,\sum_{j\neq0}c_{j}\,r_{j}^{2p}e^{-2\omega\lambda_{j}}\,\cos\left(2p\,\theta_{j}+\frac{\omega\alpha}{2}\right)\nonumber \\
+\frac{4\pi}{2^{2p}}e^{z^{2}/8}\,\text{Re} & \left(\frac{d^{2p}}{d\omega^{2p}}\left\{ \sum_{j\neq0}c_{j}e^{\frac{i\omega\alpha}{2}-2\omega\lambda_{j}}\left(e^{-z^{2}/8}+e^{z^{2}/8}\,\psi_{\alpha}\left(e^{2i\omega},z\right)\right)\right\} \right).\label{formula at beginning}
\end{align}
But it is clear that the first infinite series admits the expression
\begin{align}
\sum_{j\neq0}c_{j}\,r_{j}^{2p}e^{-2\omega\lambda_{j}}\,\cos\left(2p\,\theta_{j}+\frac{\omega\alpha}{2}\right)
=\sum_{j=1}^{\infty}c_{j}r_{j}^{2p}\left[e^{-2\omega\lambda_{j}}\cos\left(2p\,\theta_{j}+\frac{\omega\alpha}{2}\right)+e^{2\omega\lambda_{j}}\cos\left(2p\,\theta_{j}-\frac{\omega\alpha}{2}\right)\right]\nonumber \\
=2\,\sum_{j=1}^{\infty}c_{j}r_{j}^{2p}\left[\cos\left(2p\theta_{j}\right)\,\cos\left(\frac{\omega\alpha}{2}\right)\cosh\left(2\omega\lambda_{j}\right)+\sin\left(2p\theta_{j}\right)\,\sin\left(\frac{\omega\alpha}{2}\right)\sinh\left(2\omega\lambda_{j}\right)\right],\label{trivial elementary cosh simplifications}
\end{align}
because $\lambda_{-j}=-\lambda_{j}$, $r_{-j}=r_{j}$ and $\theta_{-j}=\pi-\theta_{j}$
by hypothesis. Since $\tilde{F}_{z,\alpha}\left(\frac{\alpha}{4}+it\right)$
is a real and an even function of $t$, using the previous expression
allows to rewrite (\ref{formula at beginning}) in the form
\begin{align*}
\intop_{0}^{\infty}t^{2p}\,\tilde{F}_{z,\alpha}\left(\frac{\alpha}{4}+it\right)\,\cosh\left(2\omega t\right)\,dt & =-\frac{8\pi}{2^{2p}}\,\sum_{j=1}^{\infty}c_{j}r_{j}^{2p}\left[\cos\left(2p\theta_{j}\right)\,\cos\left(\frac{\omega\alpha}{2}\right)\cosh\left(2\omega\lambda_{j}\right)+\sin\left(2p\theta_{j}\right)\,\sin\left(\frac{\omega\alpha}{2}\right)\sinh\left(2\omega\lambda_{j}\right)\right]\\
+\frac{2\pi\,e^{z^{2}/8}}{2^{2p}}\,\text{Re} & \left(\frac{d^{2p}}{d\omega^{2p}}\left\{ \sum_{j\neq0}c_{j}e^{\frac{i\omega\alpha}{2}-2\omega\lambda_{j}}\left(e^{-z^{2}/8}+e^{z^{2}/8}\,\psi_{\alpha}\left(e^{2i\omega},z\right)\right)\right\} \right),
\end{align*}
which completes the proof.
\end{proof}

\section{Proof of Theorem \ref{theorem 1.1}}\label{proof theorem 1.1 section}

\subsection{Outline of the Proof}
At the core of our method is the following lemma, which is an immediate
consequence of a Theorem of Fej\'er \cite{titchmarsh_zetafunction,fejer,fekete_zeros}.
\begin{lemma}\label{Fejer lemma}
Let $\left(p_{n}\right)_{n\in\mathbb{N}_{0}}$ be a subsequence of
$\mathbb{N}_{0}$ with $p_{0}:=0$. Then the number of variations
of sign in the interval $[0,a]$ of a real continuous function $f(x)$
is not less than the number of variations of sign of the following
sequence
\[
\mathcal{F}_{p_{n}}(a):=\begin{cases}
f(0) & n=0\\
\intop_{0}^{a}f(t)\,t^{p_{n}-1}dt & n\geq1.
\end{cases}
\]
\end{lemma}

\bigskip{}

Our proof of Theorem \ref{theorem 1.1} is essentially divided in four parts: in
the first of these, we employ lemmas \ref{lemma 2.3} and \ref{lemma 2.4} to rewrite (\ref{integral representation lemma 4})
in a form which is useful to apply the zero counting method suggested
by the previous lemma.
By the conditions of our Theorem \ref{theorem 1.1}, the sequence $\left(\lambda_{j}\right)_{j\in\mathbb{N}}$
attains its bounds and is made of distinct elements. Hence, we know
that exists some $M\in\mathbb{N}$ such that
\begin{equation}
|\lambda_{M}|=\max_{j\in\mathbb{N}}\{|\lambda_{j}|\},\,\,\,\,\,\lambda_{j}\neq\lambda_{M}\,\,\,\text{for}\,\,\,j\neq M.\label{lambda M condition attaining bounds}
\end{equation}
This implies that $r_{j}<r_{M}$ for $j\neq M$ by definition of the
polar coordinates (\ref{polar coordinates}). During the proof of
our result, we will assume without any loss of generality that $\lambda_{M}<0$.\footnote{There is no loss of generality because we can interchange $\lambda_{j}$
with $-\lambda_{j}$ in the function $\tilde{F}_{z,\alpha}\left(s\right)$.} According to the representation (\ref{polar coordinates}), this
assumption implies that $\theta_{M}\in\left(0,\frac{\pi}{2}\right)$.

Using the existence of $\lambda_{M}$ given in (\ref{lambda M condition attaining bounds}), the first section of our proof is devoted to show that, for $0<u<\frac{\pi}{4},$
\begin{align}
\intop_{0}^{\infty}t^{2p}\,\tilde{F}_{z,\alpha}\left(\frac{\alpha}{4}+it\right)\,\cosh\left(\left(\frac{\pi}{2}-2u\right)t\right)\,dt & =-\frac{8\pi c_{M}r_{M}^{2p}}{2^{2p}}\left(1+e^{z^{2}/8}\sinh\left(\frac{z^{2}}{8}\right)\right)\mathcal{G}_{p}\left(\theta_{M}\right) \nonumber \\
\times\left\{ 1+\frac{\mathcal{B}_{\alpha,p}(u)}{2c_{M}r_{M}^{2p}\,\mathcal{G}_{p}\left(\theta_{M}\right)}+\tilde{E}\left(X,z\right)+\tilde{H}\left(X,z\right)\right\}& +\frac{2\pi\,e^{z^{2}/8}}{2^{2p}}\,\text{Re}\left[\mathcal{A}_{\alpha,p}(u,z)\right],\label{first formula outlined in the proof introduction}
\end{align}
where $\mathcal{A}_{\alpha,p}(u,z)$ and the terms inside the braces
will be specified later. Moreover, $\theta_{M}$ is the angular coordinate of $\frac{i\alpha}{2}-2\lambda_{M}:=r_{M}e^{i\theta_{M}}$ in the convention (\ref{polar coordinates}) and $(\mathcal{G}_{p}(\theta_{M}))_{p\in\mathbb{N}}$
is the sequence given by
\begin{equation}
\left(\mathcal{G}_{p}(\theta_{M})\right)_{p\in\mathbb{N}}:=\cos(2p\theta_{M})\,\cos\left(\frac{\pi\alpha}{8}\right)\cosh\left(\frac{\pi\lambda_{M}}{2}\right)+\sin(2p\theta_{M})\,\sin\left(\frac{\pi\alpha}{8}\right)\sinh\left(\frac{\pi\lambda_{M}}{2}\right).\label{definition of sequnece Gp}
\end{equation}
In the second part we construct a sequence of integers $\left(q_{n}\right)_{n\in\mathbb{N}}$
such that, for every $n\in\mathbb{N}$, $\mathcal{G}_{q_{n}}(\theta_{M})$
and $\mathcal{G}_{q_{n+1}}(\theta_{M})$ have always different signs.
In the third part of our proof, we prove that the terms
\[
\mathcal{A}_{\alpha,p}(u,z),\,\frac{\mathcal{B}_{\alpha,p}(u)}{2c_{M}r_{M}^{2p}\,\mathcal{G}_{p}\left(\theta_{M}\right)},\,\tilde{E}\left(X,z\right),\,\tilde{H}\left(X,z\right)
\]
are very small if we pick $u$ small enough and take $p$ as a very
large integer. Finally, in the fourth part of our proof we combine
our collected data with Fej\'er's lemma \ref{Fejer lemma} to reach the desired conclusion.

We should remark that, throughout this proof, we will only work with
the assumption that $\alpha>2$, so that we will apply bounds (\ref{bound derivative !})
and (\ref{bound for Aup}) respectively. Of course, in the range $0\leq\alpha\le2$
the proof is entirely analogous but the auxiliary computations are
slightly different.

\subsection{A suitable integral representation}\label{suitable integral section}

 Let $0<u<\frac{\pi}{4}$ and take $\omega=\frac{\pi}{4}-u$ on the
integral representation (\ref{integral representation lemma 4}).
If $\mathcal{I}_{2p}(u)$ denotes the integral
\[
\mathcal{I}_{2p}(u):=\intop_{0}^{\infty}t^{2p}\,\tilde{F}_{z,\alpha}\left(\frac{\alpha}{4}+it\right)\,\cosh\left(\left(\frac{\pi}{2}-2u\right)t\right)\,dt,\,\,\,\,0<u<\frac{\pi}{4},
\]
then, according to (\ref{integral representation lemma 4}), $\mathcal{I}_{2p}(u)$ can be explicitly given by 
\begin{align}
-\frac{8\pi}{2^{2p}}&\,\sum_{j=1}^{\infty}c_{j}r_{j}^{2p}\left[\cos\left(2p\theta_{j}\right)\,\cos\left(\frac{\alpha}{2}\left(\frac{\pi}{4}-u\right)\right)\cosh\left(2\lambda_{j}\left(\frac{\pi}{4}-u\right)\right)+\sin\left(2p\theta_{j}\right)\,\sin\left(\frac{\alpha}{2}\left(\frac{\pi}{4}-u\right)\right)\sinh\left(2\lambda_{j}\left(\frac{\pi}{4}-u\right)\right)\right]\nonumber \\
&+\frac{2\pi\,e^{z^{2}/8}}{2^{2p}}\,\text{Re}\left(\frac{d^{2p}}{du^{2p}}\left\{ e^{\frac{i\alpha}{2}\left(\frac{\pi}{4}-u\right)}\sum_{j\neq0}c_{j}e^{-2\lambda_{j}\,\left(\frac{\pi}{4}-u\right)}\,\left(e^{-z^{2}/8}+e^{z^{2}/8}\,\psi_{\alpha}\left(ie^{-2iu},z\right)\right)\right\} \right)=\mathcal{I}_{2p}(u).\label{rewriting I2p(u)}
\end{align}
We shall rewrite (\ref{rewriting I2p(u)}) by appealing to Lemma \ref{lemma 2.3}. Consider the following function
\begin{equation}
h_{\alpha}(u)=e^{\frac{i\alpha}{2}\left(\frac{\pi}{4}-u\right)}\sum_{j\neq0}c_{j}e^{-2\lambda_{j}\,\left(\frac{\pi}{4}-u\right)}.\label{particular analytic function}
\end{equation}
Since $\sum|c_{j}|<\infty$, it is clear that $h_{\alpha}(u)$ is
an analytic function of $u$. Therefore, by (\ref{formula that defines A alpha, p})
and the notation (\ref{polar coordinates}),
\begin{align}
\frac{d^{2p}}{du^{2p}}\left\{ \left(e^{-z^{2}/8}+e^{z^{2}/8}\,\psi_{\alpha}\left(ie^{-2iu},z\right)\right)\,e^{\frac{i\alpha}{2}\left(\frac{\pi}{4}-u\right)}\sum_{j\neq0}c_{j}e^{-2\lambda_{j}\,\left(\frac{\pi}{4}-u\right)}\right\}  & =-2\sinh\left(\frac{z^{2}}{8}\right)h_{\alpha}^{(2p)}(u)+\mathcal{A}_{\alpha,p}(u,z)\nonumber \\
=-2\sinh\left(\frac{z^{2}}{8}\right)\,e^{\frac{i\alpha}{2}\left(\frac{\pi}{4}-u\right)}\,\sum_{j\neq0}c_{j}r_{j}^{2p}e^{2ip\theta_{j}}\, & e^{-2\left(\frac{\pi}{4}-u\right)\lambda_{j}}+\mathcal{A}_{\alpha,p}(u,z).\label{formula after application of lemma derivative}
\end{align}
We will now find a suitable bound for $\mathcal{A}_{\alpha,p}(u,z)$,
which will invoke Lemma \ref{lemma 2.3}. Since $\left(\lambda_{j}\right)_{j\in\mathbb{N}}$
is bounded by hypothesis, there exists some $\mu$ such that $|\lambda_{j}|<\mu$
for every $j\in\mathbb{N}$: hence
\begin{equation}
\left\Vert h_{\alpha}\right\Vert _{L^{\infty}(C_{0}(1))}\leq e^{\frac{\alpha}{2}}\sum_{j\neq0}|c_{j}|e^{\left(2-\frac{\pi}{2}\right)\lambda_{j}}<e^{\left(4-\pi+\alpha\right)\mu}\,\sum_{j\neq0}|c_{j}|\leq\mathcal{M},\label{bound for the h alpha Linfinity}
\end{equation}
where $\mathcal{M}$ only depends on the sequence the value of the
convergent series $\sum_{j\neq0}|c_{j}|:=2\sum_{j=1}^{\infty}|c_{j}|$. Recalling (\ref{bound for Aup}),
we arrive at the following bound
\begin{equation}
|\mathcal{A}_{\alpha,p}(u,z)|<D\frac{2^{14p}(2p)!e^{-\frac{A}{u}}}{u^{\frac{\alpha}{2}+2p}}\left\Vert h\right\Vert _{L^{\infty}(C_{1}(0))}\leq\mathscr{D}\frac{2^{14p}(2p)!e^{-\frac{A}{u}}}{u^{\frac{\alpha}{2}+2p}},\label{bound for Aalpha,p}
\end{equation}
for some $\mathscr{D}$ depending only on the sequence $\left(c_{j}\right)_{j\in\mathbb{N}}$
and $\alpha$. Returning to the second term of (\ref{rewriting I2p(u)})
and using (\ref{formula after application of lemma derivative}),
we find that 
\begin{align}
e^{z^{2}/8}\text{Re}\left(\,\frac{d^{2p}}{du^{2p}}\left\{ \left(e^{-z^{2}/8}+e^{z^{2}/8}\,\psi_{\alpha}\left(ie^{-2iu},z\right)\right)\,e^{\frac{i\alpha}{2}\left(\frac{\pi}{4}-u\right)}\sum_{j\neq0}c_{j}e^{-2\left(\frac{\pi}{4}-u\right)\lambda_{j}}\right\} \right)\nonumber \\
=e^{z^{2}/8}\,\text{Re}\left[\mathcal{A}_{\alpha,p}(u,z)\right]-2e^{\frac{z^{2}}{8}}\sinh\left(\frac{z^{2}}{8}\right)\,\sum_{j\neq0}c_{j}r_{j}^{2p}e^{-2\left(\frac{\pi}{4}-u\right)\lambda_{j}}\,\cos\left(2p\theta_{j}+\frac{\pi\alpha}{8}-\frac{\alpha}{2}u\right).\label{the real part explicit}
\end{align}
We can rewrite the infinite series on the previous expression by repeating the same computations
leading to (\ref{trivial elementary cosh simplifications}): we see
that 
\begin{align}
2\,\sum_{j=1}^{\infty}c_{j}r_{j}^{2p}\left\{ \cos\left(2p\theta_{j}\right)\,\cos\left(\frac{\alpha}{2}\left(\frac{\pi}{4}-u\right)\right)\cosh\left(2\left(\frac{\pi}{4}-u\right)\lambda_{j}\right)+\sin\left(2p\theta_{j}\right)\,\sin\left(\frac{\alpha}{2}\left(\frac{\pi}{4}-u\right)\right)\sinh\left(2\left(\frac{\pi}{4}-u\right)\lambda_{j}\right)\right\} \nonumber \protect\\
=\sum_{j\protect\neq0}c_{j}r_{j}^{2p}e^{-2\left(\frac{\pi}{4}-u\right)\lambda_{j}}\,\cos\left(2p\theta_{j}+\frac{\pi\alpha}{8}-\frac{\alpha}{2}u\right),\label{suitable simplification}
\end{align}
and so, after combining (\ref{suitable simplification}) with (\ref{rewriting I2p(u)}),
we derive the representation
\begin{align}
&\mathcal{I}_{2p}(u)=\frac{2\pi\,e^{z^{2}/8}}{2^{2p}}\,\text{Re}\left[\mathcal{A}_{\alpha,p}(u,z)\right]-\frac{8\pi}{2^{2p}}\left(1+e^{z^{2}/8}\sinh\left(\frac{z^{2}}{8}\right)\right)\nonumber\\
&\times\sum_{j=1}^{\infty}c_{j}\,r_{j}^{2p}\left\{ \cos\left(2p\theta_{j}\right)\,\cos\left(\frac{\alpha}{2}\left(\frac{\pi}{4}-u\right)\right)\cosh\left(2\left(\frac{\pi}{4}-u\right)\lambda_{j}\right)+\sin\left(2p\theta_{j}\right)\,\sin\left(\frac{\alpha}{2}\left(\frac{\pi}{4}-u\right)\right)\sinh\left(2\left(\frac{\pi}{4}-u\right)\lambda_{j}\right)\right\}.\label{beginning of the proof after lemma}
\end{align}
Of course, by (\ref{the real part explicit}) and (\ref{rewriting I2p(u)}), an equivalent way of writing (\ref{beginning of the proof after lemma}) is returning to the sum over $j\in\mathbb{Z}\setminus\{0\}$ and get
\begin{align}
\mathcal{I}_{2p}(u) & =-\frac{4\pi}{2^{2p}}\left(1+e^{z^{2}/8}\sinh\left(\frac{z^{2}}{8}\right)\right)\,\sum_{j\neq0}c_{j}\,r_{j}^{2p}e^{-2\left(\frac{\pi}{4}-u\right)\lambda_{j}}\,\cos\left(2p\,\theta_{j}+\frac{\pi\alpha}{8}-\frac{\alpha}{2}u\right)\nonumber \\
 & +\frac{2\pi\,e^{z^{2}/8}}{2^{2p}}\,\text{Re}\left[\mathcal{A}_{\alpha,p}(u,z)\right].\label{simplification almost endinggggg}
\end{align}

We will now analyze the infinite series on the first term on the right
side of (\ref{beginning of the proof after lemma}) and approximate
it by its value for $u=0$. For this analysis it will be easier to manipulate the more compact expression (\ref{simplification almost endinggggg}). By the mean value theorem and the existence
of an element $\lambda_{M}$ satisfying (\ref{lambda M condition attaining bounds}),
we know that
\begin{equation}
\left|\sum_{j\neq0}c_{j}\,r_{j}^{2p}e^{-2\left(\frac{\pi}{4}-u\right)\lambda_{j}}\,\left\{ \cos\left(2p\,\theta_{j}+\frac{\pi\alpha}{8}-\frac{\alpha}{2}u\right)-\cos\left(2p\,\theta_{j}+\frac{\pi\alpha}{8}\right)\right\} \right|
\leq\frac{\alpha}{2}e^{\frac{\pi}{2}|\lambda_{M}|}r_{M}^{2p}\sum_{j\neq0}|c_{j}|\,u\leq C_{\alpha}\,r_{M}^{2p}\,u,\label{a mean value estimate}
\end{equation}
where, by the definition (\ref{polar coordinates}), $\frac{i\alpha}{2}-2\lambda_{M}=r_{M}e^{i\theta_{M}}$.
Thus, we can write the infinite series in (\ref{simplification almost endinggggg}) in the approximate form
\begin{align}
\sum_{j\neq0}c_{j}\,r_{j}^{2p}e^{-2\left(\frac{\pi}{4}-u\right)\lambda_{j}}&\,\cos\left(2p\,\theta_{j}+\frac{\pi\alpha}{8}-\frac{\alpha}{2}u\right) =\sum_{j\neq0}c_{j}\,r_{j}^{2p}e^{-2\left(\frac{\pi}{4}-u\right)\lambda_{j}}\cos\left(2p\,\theta_{j}+\frac{\pi\alpha}{8}\right)+\mathcal{B}_{\alpha,p}(u)\nonumber \\
=2\,\sum_{j=1}^{\infty}c_{j}\,r_{j}^{2p} & \left\{ \cos\left(2p\theta_{j}\right)\,\cos\left(\frac{\pi\alpha}{8}\right)\cosh\left(\frac{\pi}{2}\lambda_{j}\right)+\sin\left(2p\theta_{j}\right)\,\sin\left(\frac{\pi\alpha}{8}\right)\sinh\left(\frac{\pi}{2}\lambda_{j}\right)\right\} +\mathcal{B}_{\alpha,p}(u),\label{series on (111)}
\end{align}
where, according to (\ref{a mean value estimate}), $\mathcal{B}_{\alpha,p}(u)$
satisfies the estimate
\begin{equation}
|\mathcal{B}_{\alpha,p}(u)|\leq C_{\alpha}\,r_{M}^{2p}\,u.\label{Balpha p estimate as u tends to zero}
\end{equation}
Assume that $p\in\mathbb{N}$ is such that
\begin{equation}
\mathcal{G}_{p}(\theta_{M}):=\cos\left(2p\theta_{M}\right)\,\cos\left(\frac{\pi\alpha}{8}\right)\cosh\left(\frac{\pi}{2}\lambda_{M}\right)+\sin\left(2p\theta_{M}\right)\,\sin\left(\frac{\pi\alpha}{8}\right)\sinh\left(\frac{\pi}{2}\lambda_{M}\right)\neq0.\label{Gp first deffffition}
\end{equation}
Then, for every $p$ satisfying (\ref{Gp first deffffition}), we
can write (\ref{series on (111)}) as 
\begin{align}
\sum_{j\neq0}c_{j}\,r_{j}^{2p}e^{-2\left(\frac{\pi}{4}-u\right)\lambda_{j}}\,\cos\left(2p\,\theta_{j}+\frac{\pi\alpha}{8}-\frac{\alpha}{2}u\right) & =2\,\sum_{j=1}^{\infty}c_{j}r_{j}^{2p}\,\mathcal{G}_{p}\left(\theta_{j}\right)+\mathcal{B}_{\alpha,p}(u)\nonumber \\
=2c_{M}r_{M}^{2p}\,\mathcal{G}_{p}\left(\theta_{M}\right) & \left\{ 1+\frac{\mathcal{B}_{\alpha,p}(u)}{2c_{M}r_{M}^{2p}\,\mathcal{G}_{p}\left(\theta_{M}\right)}+\tilde{E}\left(X,z\right)+\tilde{H}\left(X,z\right)\right\} ,\label{putting term outside with maximum}
\end{align}
where
\begin{equation}
\tilde{E}(X,z):=\sum_{j\neq M,\,j\leq X}\frac{c_{j}}{c_{M}}\,\left(\frac{r_{j}}{r_{M}}\right)^{2p}\,\frac{\cos\left(2p\theta_{j}\right)\,\cos\left(\frac{\pi\alpha}{8}\right)\cosh\left(\frac{\pi}{2}\lambda_{j}\right)+\sin\left(2p\theta_{j}\right)\,\sin\left(\frac{\pi\alpha}{8}\right)\sinh\left(\frac{\pi}{2}\lambda_{j}\right)}{\cos\left(2p\theta_{M}\right)\,\cos\left(\frac{\pi\alpha}{8}\right)\cosh\left(\frac{\pi}{2}\lambda_{M}\right)+\sin\left(2p\theta_{M}\right)\,\sin\left(\frac{\pi\alpha}{8}\right)\sinh\left(\frac{\pi}{2}\lambda_{M}\right)}\label{E(X,z)}
\end{equation}
and
\begin{equation}
\tilde{H}\left(X,z\right):=\sum_{j\neq M,\,j>X}\frac{c_{j}}{c_{M}}\,\left(\frac{r_{j}}{r_{M}}\right)^{2p}\,\frac{\cos\left(2p\theta_{j}\right)\,\cos\left(\frac{\pi\alpha}{8}\right)\cosh\left(\frac{\pi}{2}\lambda_{j}\right)+\sin\left(2p\theta_{j}\right)\,\sin\left(\frac{\pi\alpha}{8}\right)\sinh\left(\frac{\pi}{2}\lambda_{j}\right)}{\cos\left(2p\theta_{M}\right)\,\cos\left(\frac{\pi\alpha}{8}\right)\cosh\left(\frac{\pi}{2}\lambda_{M}\right)+\sin\left(2p\theta_{M}\right)\,\sin\left(\frac{\pi\alpha}{8}\right)\sinh\left(\frac{\pi}{2}\lambda_{M}\right)},\label{H(X,z)}
\end{equation}
for some large parameter $X$ that will be chosen later.  Using (\ref{simplification almost endinggggg}),
we arrive at the expression
\begin{align}
\mathcal{I}_{2p}(u) & =-\frac{8\pi c_{M}r_{M}^{2p}}{2^{2p}}\left(1+e^{z^{2}/8}\sinh\left(\frac{z^{2}}{8}\right)\right)\mathcal{G}_{p}\left(\theta_{M}\right)\left\{ 1+\frac{\mathcal{B}_{\alpha,p}(u)}{2c_{M}r_{M}^{2p}\,\mathcal{G}_{p}\left(\theta_{M}\right)}+\tilde{E}\left(X,z\right)+\tilde{H}\left(X,z\right)\right\} \nonumber \\
 & +\frac{2\pi\,e^{z^{2}/8}}{2^{2p}}\,\text{Re}\left[\mathcal{A}_{\alpha,p}(u,z)\right],\label{integral representation with Gp simplified}
\end{align}
which is exactly the one claimed at the introduction of our proof,
(\ref{first formula outlined in the proof introduction}). 

In the remaining parts of our proof, the strategy will be to see that
$\left(\mathcal{G}_{p}(\theta_{M})\right)_{p\in\mathbb{N}}$ is the
dominant term in (\ref{putting term outside with maximum}) and in
(\ref{integral representation with Gp simplified}). This will be concluded
after showing that $\mathcal{A}_{\alpha,p}(u,z)$, $\mathcal{B}_{\alpha,p}(u)$,
$\tilde{E}(X,z)$ and $\tilde{H}(X,z)$ are very small for infinitely
many values of $p$ and for sufficiently small $u$. Since (\ref{definition of sequnece Gp})
will change its sign infinitely often for a suitable sequence of integers
$\left(q_{n}\right)_{n\in\mathbb{N}}$, our proof will be concluded
after this observation. 

\subsection{Studying the sign changes of $\left(\mathcal{G}_{p}(\theta_{M})\right)_{p\in\mathbb{N}}$}

We will now show that we can always construct a sequence $(q_{n})_{n\in\mathbb{N}}\subset\mathbb{N}$
for which $\mathcal{G}_{q_{n}}(\theta_{M})$ and $\mathcal{G}_{q_{n+1}}(\theta_{M})$
always have a different sign. To that end, we divide our construction
in two cases: $\frac{\theta_{M}}{\pi}\notin\mathbb{Q}$ or $\frac{\theta_{M}}{\pi}\in\mathbb{Q}$.

\paragraph*{$1^{\text{st}}$ case: $\theta_{M}/\pi\protect\notin\mathbb{Q}$}
Since we are assuming that $\lambda_{M}<0$ (so that $\theta_{M}\in(0,\frac{\pi}{2})$),
we may write the sequence (\ref{definition of sequnece Gp}) as follows
\[
\left(\mathcal{G}_{p}\left(\theta_{M}\right)\right)_{p\in\mathbb{N}}=\cos(2p\theta_{M})\,\cos\left(\frac{\pi\alpha}{8}\right)\cosh\left(\frac{\pi|\lambda_{M}|}{2}\right)-\sin(2p\theta_{M})\,\sin\left(\frac{\pi\alpha}{8}\right)\sinh\left(\frac{\pi|\lambda_{M}|}{2}\right).
\]
We will study the sign changes of the previous sequence through the
study of the function
\[
G(\phi):=\cos\left(\frac{\pi\alpha}{8}\right)\cosh\left(\frac{\pi|\lambda_{M}|}{2}\right)\,\cos(\phi)-\sin\left(\frac{\pi\alpha}{8}\right)\sinh\left(\frac{\pi|\lambda_{M}|}{2}\right)\,\sin(\phi).
\]
In order to study $G(\phi)$, let us assume that $8k<\alpha<8k+4$,
for some $k\in\mathbb{N}_{0}$. The remaining cases $8k+4<\alpha<8k+8$ and $\alpha\equiv0\mod4$ are analogous and will be sketched at the
end of the argument.  Under the assumption $8k<\alpha<8k+4$, we
know that $\cos\left(\frac{\pi\alpha}{8}\right)$ and $\sin\left(\frac{\pi\alpha}{8}\right)$
must have the same sign and, without loss of generality, we assume
it to be positive.

Let us define the sets
\begin{equation}
\mathcal{A}^{+}:=\left(\arctan\left(\cot\left(\frac{\pi\alpha}{8}\right)\coth\left(\frac{\pi|\lambda_{M}|}{2}\right)\right)-\pi,\arctan\left(\cot\left(\frac{\pi\alpha}{8}\right)\coth\left(\frac{\pi|\lambda_{M}|}{2}\right)\right)\right)
\end{equation}
and
\begin{equation}
\mathcal{A}^{-}:=\left(\arctan\left(\cot\left(\frac{\pi\alpha}{8}\right)\coth\left(\frac{\pi|\lambda_{M}|}{2}\right)\right),\arctan\left(\cot\left(\frac{\pi\alpha}{8}\right)\coth\left(\frac{\pi|\lambda_{M}|}{2}\right)\right)+\pi\right).
\end{equation}
Clearly, if $\phi\in\mathcal{A}^{+}$ then $G(\phi)>0$, while $G(\phi)<0$
when $\phi\in\mathcal{A}^{-}$.We now consider the subintervals of
$\mathcal{A}^{+}$ and $\mathcal{A}^{-}$ defined by 
\begin{equation}
\mathcal{I}^{+}:=\left(\arctan\left(\cot\left(\frac{\pi\alpha}{8}\right)\coth\left(\frac{\pi|\lambda_{M}|}{2}\right)\right)-\frac{\pi}{2}-\theta_{M};\,\arctan\left(\cot\left(\frac{\pi\alpha}{8}\right)\coth\left(\frac{\pi|\lambda_{M}|}{2}\right)\right)-\frac{\pi}{2}+\theta_{M}\right),\label{I+}
\end{equation}
\begin{equation}
\mathcal{I}^{-}:=\left(\arctan\left(\cot\left(\frac{\pi\alpha}{8}\right)\coth\left(\frac{\pi|\lambda_{M}|}{2}\right)\right)+\frac{\pi}{2}-\theta_{M};\,\arctan\left(\cot\left(\frac{\pi\alpha}{8}\right)\coth\left(\frac{\pi|\lambda_{M}|}{2}\right)\right)+\frac{\pi}{2}+\theta_{M}\right).\label{I-}
\end{equation}
Then $\mathcal{I}^{+}$ (resp. $\mathcal{I}^{-}$) is an arc in the
center of $\mathcal{A}^{+}$ (resp. $\mathcal{A}^{-}$) with length
$2\theta_{M}$. We can therefore write the partition
\begin{equation}
\mathcal{A}^{+}=\mathcal{A}_{1}^{+}\cup\mathcal{I}^{+}\cup\mathcal{A}_{2}^{+},\label{great partition}
\end{equation}
where
\[
\mathcal{A}_{1}^{+}=\left(\arctan\left(\cot\left(\frac{\pi\alpha}{8}\right)\coth\left(\frac{\pi|\lambda_{M}|}{2}\right)\right)-\pi,\arctan\left(\cot\left(\frac{\pi\alpha}{8}\right)\coth\left(\frac{\pi|\lambda_{M}|}{2}\right)\right)-\frac{\pi}{2}-\theta_{M}\right),
\]
while
\[
\mathcal{A}_{2}^{+}=\left(\arctan\left(\cot\left(\frac{\pi\alpha}{8}\right)\coth\left(\frac{\pi|\lambda_{M}|}{2}\right)\right)-\frac{\pi}{2}+\theta_{M},\arctan\left(\cot\left(\frac{\pi\alpha}{8}\right)\coth\left(\frac{\pi|\lambda_{M}|}{2}\right)\right)\right).
\]
Analogously, we can find the partition of $\mathcal{A}^{-}$,
\[
\mathcal{A}^{-}=\mathcal{A}_{1}^{-}\cup\mathcal{I}^{-}\cup\mathcal{A}_{2}^{-},
\]
where
\[
\mathcal{A}_{1}^{-}=\left(\arctan\left(\cot\left(\frac{\pi\alpha}{8}\right)\coth\left(\frac{\pi|\lambda_{M}|}{2}\right)\right),\arctan\left(\cot\left(\frac{\pi\alpha}{8}\right)\coth\left(\frac{\pi|\lambda_{M}|}{2}\right)\right)+\frac{\pi}{2}-\theta_{M}\right)
\]
and
\[
\mathcal{A}_{2}^{-}=\left(\arctan\left(\cot\left(\frac{\pi\alpha}{8}\right)\coth\left(\frac{\pi|\lambda_{M}|}{2}\right)\right)+\frac{\pi}{2}+\theta_{M},\arctan\left(\cot\left(\frac{\pi\alpha}{8}\right)\coth\left(\frac{\pi|\lambda_{M}|}{2}\right)\right)+\pi\right).
\]
\begin{figure}
\begin{center}
\includegraphics{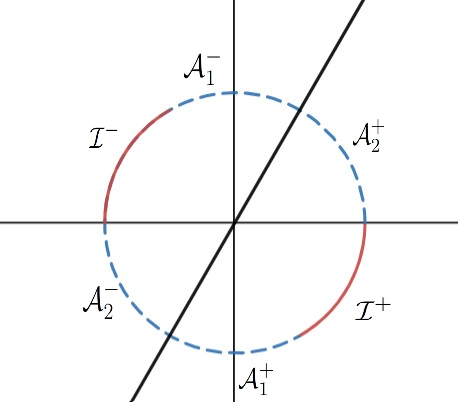}\caption{The partition of the unit circle into $\mathcal{A}^{+}$ and $\mathcal{A}^{-}$
and the representations of the subsets $\mathcal{A}_{1}^{\pm}$, $\mathcal{A}_{2}^{\pm}$
and $\mathcal{I}^{\pm}$. The line dividing the unit circle is defined
by the equation $Y=\arctan\left(\cot\left(\frac{\pi\alpha}{8}\right)\coth\left(\frac{\pi|\lambda_{M}|}{2}\right)\right)\,X$.}
\end{center}
\end{figure}
\\

Since $\theta_{M}/\pi\notin\mathbb{Q}$, by Kronecker's lemma the
set of points $\left\{ n\theta_{M}/\pi\right\} _{n\in\mathbb{N}}$
is dense in $(0,1)$. Thus, there exist two infinite sequences $\left(q_{n}^{+}\right)_{n\in\mathbb{N}}$
and $\left(q_{n}^{-}\right)_{n\in\mathbb{N}}$ such that
\begin{equation}
\left(q_{n}^{+}\right)_{n\in\mathbb{N}}:=\left\{ q\in\mathbb{N}\,:\,2q\theta_{M}\in\mathcal{I}^{+}\right\} ,\,\,\,\,\left(q_{n}^{-}\right)_{n\in\mathbb{N}}:=\left\{ q\in\mathbb{N}\,:\,2q\theta_{M}\in\mathcal{I}^{-}\right\} .\label{construction of sequences important for the proof}
\end{equation}
The main claim of this section establishes that $\left(q_{n}^{+}\right)_{n\in\mathbb{N}}$
and $\left(q_{n}^{-}\right)_{n\in\mathbb{N}}$ interlace.

\begin{claim}\label{claim 3.1}
For every $n\in\mathbb{N}$, we have that
\[
q_{n}^{+}<q_{n}^{-}<q_{n+1}^{+}<q_{n+1}^{-}<q_{n+2}^{+}<\ldots.
\]
\end{claim}

\begin{proof}
First, let us observe that, if $a\in\left(q_{n}^{+}\right)_{n\in\mathbb{N}}$
then $a+1\notin\left(q_{n}^{+}\right)_{n\in\mathbb{N}}$, because
$2(a+1)\theta_{M}=2a\theta_{M}+2\theta_{M}$ and the length of $\mathcal{I}^{+}$
is, by construction, equal to $2\theta_{M}$. Therefore, either $2(a+1)\theta_{M}\in\mathcal{I}^{-}$
(i.e., $a+1\in\left(q_{n}^{-}\right)_{n\in\mathbb{N}}$) or $2(a+1)\theta_{M}\in\mathcal{A}_{2}^{+}\cup\mathcal{A}_{1}^{-}$.
This actually happens because $2\theta_{M}<\pi$ and, since $2a\theta_{M}\in\mathcal{I}^{+}$
by hypothesis, then
\[
2a\theta_{M}+2\theta_{M}<\arctan\left(\cot\left(\frac{\pi\alpha}{8}\right)\coth\left(\frac{\pi|\lambda_{M}|}{2}\right)\right)+\frac{\pi}{2}+\theta_{M},
\]
which assures that $2(a+1)\theta_{M}\notin\mathcal{A}_{2}^{-}$. Since
$a\in(q_{n}^{+})_{n\in\mathbb{N}}$ then $a=q_{m_{0}}^{+}$ for some
$m_{0}$. If $a+1\in(q_{n}^{-})_{n\in\mathbb{N}}$, then we have that
$q_{m_{0}}^{+}<q_{m_{0}}^{-}$. On the other hand, if $2(a+1)\theta_{M}\in\mathcal{A}_{2}^{+}\cup\mathcal{A}_{1}^{-}$
then there is some $j\geq2$ such that $2(a+j)\theta_{M}\in\mathcal{I}^{-}$
because the length of $\mathcal{I}^{-}$ is equal to $2\theta_{M}$.
Therefore, the translation $a+j$, $j\in\mathbb{N}_{0}$, cannot go
from $q_{m_{0}}^{+}$ to $q_{m_{0}+1}^{+}$ without passing first
by an element of the sequence $\left(q_{n}^{-}\right)_{n\in\mathbb{N}}$.
This completes the proof of the claim.
\end{proof}

\bigskip{}

The interlacing property of $q_{n}^{+}$ and $q_{n}^{-}$ shows that
we can construct a new sequence
\[
\left(q_{n}\right)_{n\in\mathbb{N}}:=\left(q_{n}^{+}\right)_{n\in\mathbb{N}}\cup\left(q_{n}^{-}\right)_{n\in\mathbb{N}}
\]
in such a way that $q_{2n-1}:=q_{n}^{+}$ and $q_{2n}:=q_{n}^{-}$.
Moreover, we have that $\mathcal{G}_{q_{m}}\left(\theta_{M}\right)$
and $\mathcal{G}_{q_{m+1}}\left(\theta_{M}\right)$ have opposite
signs. 

Finally, we remark that the cases $8k+4<\alpha<8k+8$ and $\alpha\equiv0\mod4$
are analogous. In the first case, $\cos\left(\frac{\pi\alpha}{8}\right)$
and $\sin\left(\frac{\pi\alpha}{8}\right)$ have opposite signs. However,
a similar strategy as the one given above works and the only modification
is replacing $\mathcal{A}_{1}^{+},\mathcal{I}^{+}\text{ and }\mathcal{A}_{2}^{+}$
respectively by $\mathcal{A}_{1}^{-},\mathcal{I}^{-}\text{ and }\mathcal{A}_{2}^{-}$.

On the other hand, if $\alpha=8k$ or $\alpha=8k+4$ for some $k\in\mathbb{N}_{0}$,
then
\begin{equation}
\left(\mathcal{G}_{p}(\theta_{M})\right)_{p\in\mathbb{N}}=(-1)^{k}\,\cosh\left(\frac{\pi|\lambda_{M}|}{2}\right)\,\cos\left(2p\theta_{M}\right),\,\,\,\text{if \,}\alpha=8k\label{first case}
\end{equation}
and
\begin{equation}
\left(\mathcal{G}_{p}(\theta_{M})\right)_{p\in\mathbb{N}}=(-1)^{k-1}\,\sinh\left(\frac{\pi|\lambda_{M}|}{2}\right)\,\sin\left(2p\theta_{M}\right),\,\,\,\text{if }\,\alpha=8k+4.\label{second case}
\end{equation}
In any of these cases a similar construction of the sets $\mathcal{A}_{1}^{\pm},$
$\mathcal{A}_{2}^{\pm}$ and $\mathcal{I}^{\pm}$ can be made. 

\bigskip{}

\paragraph*{$2^{\text{nd}}$ case: $\theta_{M}/\pi\in\mathbb{Q}$} The conclusions given above in the case $\theta_{M}/\pi\in\mathbb{Q}$
are analogous but instead of arguing via Kronecker's lemma, the sequences
(\ref{construction of sequences important for the proof}) can be
explicitly constructed in this case. Since $|\mathcal{I}^{+}|=|\mathcal{I}^{-}|=2\theta_{M}$,
we know that there exist $q_{1}$ and $q_{2}>q_{1}$ such that $2q_{1}\theta_{M}\in\mathcal{I}^{+}$
and $2q_{2}\theta_{M}\in\mathcal{I}^{-}$. Since $\theta_{M}=\frac{a}{b}\pi$,
$\left(a,b\right)=1$, then $2q_{1}\theta_{M}\equiv2\left(q_{1}+nb\right)\theta_{M}\mod2\pi$
and $2q_{2}\theta_{M}\equiv2\left(q_{2}+nb\right)\theta_{M}\mod2\pi$.
Hence, if we construct the sequence $\left(q_{n}\right)_{n\in\mathbb{N}}$
in such a way that\footnote{Note that the interlacing property established by Claim \ref{claim 3.1} is obviously valid in this case.}
\begin{equation}
q_{2n-1}:=q_{1}+nb,\,\,\,\,q_{2n}:=q_{2}+nb,\label{construction of the sequence rational case}
\end{equation}
then $\mathcal{G}_{q_{n}}(\theta_{M})\cdot\mathcal{G}_{q_{n+1}}(\theta_{M})<0,$
i.e., $\mathcal{G}_{q_{n}}\left(\theta_{M}\right)$ and $\mathcal{G}_{q_{n+1}}\left(\theta_{M}\right)$
have opposite signs. 

\bigskip{}
\bigskip{}

\subsection{The dominance of $\left(\mathcal{G}_{p}(\theta_{M})\right)_{p\in\mathbb{N}}$.}

Having constructed a sequence $\left(q_{n}\right)_{n\in\mathbb{N}}$
for which $\left(\mathcal{G}_{q_{n}}(\theta_{M})\right)_{n\in\mathbb{N}}$
is always alternating its sign, we are now ready to proceed with the
proof. Our next claim establishes that the sequence $\mathcal{G}_{q_{n}}(\theta_{M})$
dominates, in some sense, the second term (involving $\mathcal{A}_{\alpha,p}(u,z)$)
appearing in (\ref{integral representation with Gp simplified}).

\begin{claim}\label{claim 3.2}
There exists a sufficiently large $\eta_{0}$ such that, for any $z$
satisfying (\ref{condition 1/6}), $\mathcal{A}_{\alpha,p}(\cdot,z)$
satisfies 
\begin{equation}
e^{z^{2}/8}\,\left|\text{Re}\left[\mathcal{A}_{\alpha,p}\left(\frac{1}{\eta_{0}\,p\log(p)},z\right)\right]\right|<|c_{M}|\,r_{M}^{2p}\left(1+e^{z^{2}/8}\sinh\left(\frac{z^{2}}{8}\right)\right)\left|\mathcal{G}_{p}\left(\theta_{M}\right)\right|\label{bound for A alpha p with u depending on p}
\end{equation}
for any $p\in\left(q_{n}\right)_{n\in\mathbb{N}}$.
\end{claim}

\begin{proof}
By (\ref{bound for Aalpha,p}), we have that
\[
|\mathcal{A}_{\alpha,p}(u,z)|\leq\mathscr{D}\frac{2^{14p}(2p)!e^{-\frac{A}{u}}}{u^{\frac{\alpha}{2}+2p}},\,\,\,\,0<u<\frac{\pi}{4},
\]
so that the inequality (\ref{bound for A alpha p with u depending on p})
holds for any $p\in\left(q_{n}\right)_{n\in\mathbb{N}}$ if   
\begin{equation}
\frac{\mathscr{D}2^{14p}(2p)!\,e^{z^{2}/8}}{|c_{M}|r_{M}^{2p}\,u^{\frac{\alpha}{2}+2p}\left(1+e^{z^{2}/8}\sinh\left(\frac{z^{2}}{8}\right)\right)\left|\mathcal{G}_{p}\left(\theta_{M}\right)\right|}<e^{A/u}\label{condition to hold}
\end{equation}
holds for any $p\in(q_{n})_{n\in\mathbb{N}}$. By construction of
the sets $\mathcal{I}^{+}$ and $\mathcal{I}^{-}$ (see (\ref{I+}) and
(\ref{I-}) above), if $p\in(q_{n})_{n\in\mathbb{N}}$ then
\begin{equation}
\left|\mathcal{G}_{p}\left(\theta_{M}\right)\right|=\left|\cos\left(2p\theta_{M}\right)\,\cos\left(\frac{\pi\alpha}{8}\right)\cosh\left(\frac{\pi}{2}\lambda_{M}\right)+\sin\left(2p\theta_{M}\right)\,\sin\left(\frac{\pi\alpha}{8}\right)\sinh\left(\frac{\pi}{2}\lambda_{M}\right)\right|\geq\epsilon_{0},\label{lower bound for terrrrms all the terrrms}
\end{equation}
for some $\epsilon_{0}>0$ (only depending on $\theta_{M}$). Therefore,
since $e^{z^{2}/8}<e^{\frac{\pi\alpha}{576}}$ (by our condition
(\ref{condition first rectangle})), (\ref{condition to hold}) is plainly
satisfied if, for any $p\in(q_{n})_{n\in\mathbb{N}}$, 
\begin{equation}
\frac{\mathscr{D}\exp\left(\frac{\pi\alpha}{576}\right)}{\epsilon_{0}|c_{M}|}\,\frac{2^{14p}(2p)!}{r_{M}^{2p}u^{\frac{\alpha}{2}+2p}}<e^{A/u}.\label{both sides of sufficient condition already to take the logarithm}
\end{equation}
Letting $u=\frac{1}{\eta_{0}\,p\log(p)}$ and taking the logarithm
on both sides of (\ref{both sides of sufficient condition already to take the logarithm}),
we see that (\ref{both sides of sufficient condition already to take the logarithm})
holds if 
\begin{equation}
\frac{1}{Ap\log(p)}\log\left(\frac{\mathscr{D}\exp\left(\frac{\pi\alpha}{576}\right)}{\epsilon_{0}|c_{M}|}\right)+\frac{2\,\log\left(128/r_{M}\right)}{A\log(p)}+\frac{\log(2p)!}{Ap\log(p)}+\frac{1}{A\log(p)}\left(\log(\eta)+\log(p)+\log(\log(p))\right)\left(2+\frac{\alpha}{2p}\right)<\eta_{0}.\label{Final sufficient condition with all the logs}
\end{equation}
From Stirling's formula, we know that $\log(2p)!=O\left(p\log(p)\right)$,
from which it follows that (\ref{Final sufficient condition with all the logs})
holds for every $p\in(q_{n})_{n\in\mathbb{N}}$ if $\eta_{0}$ is
a sufficiently large number. This completes the proof of (\ref{bound for A alpha p with u depending on p}).
\end{proof}

\bigskip{}

We have proved that the second term of (\ref{integral representation with Gp simplified})
is dominated by the term involving $\mathcal{G}_{p}\left(\theta_{M}\right)$
once $p$ belongs to a suitable sequence of integers. Similar to what
we have done in {[}\cite{RYCE}, p. 47{]}, a bound for the functions
$\tilde{E}(X,z)$ and $\tilde{H}(X,z)$ also holds.

\begin{claim}\label{claim 3.3}
There exists a sufficiently large $N_{0}$ such that, for any $p\in\left(q_{n}\right)_{n\geq N_{0}}$
and $z$ satisfying (\ref{condition first rectangle}), the following
inequalities hold
\begin{equation}
|\tilde{E}(X,z)|<\frac{1}{6},\,\,\,|\tilde{H}\left(X,z\right)|<\frac{1}{6}\label{bound E and H}
\end{equation}
and, for any $\eta>0$,
\begin{equation}
\frac{\left|\mathcal{B}_{\alpha,p}\left(\frac{1}{\eta p\log(p)}\right)\right|}{2\,|c_{M}|r_{M}^{2p}\left|\mathcal{G}_{p}\left(\theta_{M}\right)\right|}<\frac{1}{6},\label{Bound Balpha p}
\end{equation}
with $\mathcal{B}_{\alpha,p}(u)$ is defined by (\ref{series on (111)}).
\end{claim}

\begin{proof}
Since $p\in(q_{n})_{n\geq N_{0}}$, then (\ref{lower bound for terrrrms all the terrrms})
holds for some $\epsilon_{0}\geq0$. Thus, since $(c_{j})_{j\in\mathbb{N}}\in\ell^{1}$, we can choose a sufficiently large $X\geq X_{0}$ such that
\begin{align*}
|\tilde{H}(X,z)| & \leq\frac{1}{\epsilon_{0}}\sum_{j\neq M,\,j>X}\frac{|c_{j}|}{|c_{M}|}\,\left(\frac{r_{j}}{r_{M}}\right)^{2p}\,\left|\cos\left(2p\theta_{j}\right)\,\cos\left(\frac{\pi\alpha}{8}\right)\cosh\left(\frac{\pi}{2}\lambda_{j}\right)+\sin\left(2p\theta_{j}\right)\,\sin\left(\frac{\pi\alpha}{8}\right)\sinh\left(\frac{\pi}{2}\lambda_{j}\right)\right|\\
 & \leq\frac{2e^{\frac{\pi}{2}\lambda_{M}}}{\epsilon_{0}|c_{M}|}\,\sum_{j\neq M,j>X}\,|c_{j}|<\frac{1}{6}.
\end{align*}
On the other hand, we may bound $\tilde{E}(X,z)$ as follows
\[
\left|\tilde{E}(X,z)\right|\leq\frac{2e^{\frac{\pi}{2}\lambda_{M}}}{\epsilon_{0}|c_{M}|}\mu_{X}^{2p}\sum_{j\neq M,\,j\leq X}|c_{j}|,
\]
where 
\[
\mu_{X}=\max_{j\leq X}\left\{ \frac{r_{j}}{r_{M}}\right\} .
\]
By property (\ref{lambda M condition attaining bounds}), we know
that $\mu_{X}<1$ and so, for $N_{0}$ sufficiently large and $p\in\left(q_{n}\right)_{n\geq N_{0}}$,
\[
|\tilde{E}(X,z)|<\frac{1}{6},
\]
which proves (\ref{bound E and H}). To get (\ref{Bound Balpha p}),
we just need to use (\ref{Balpha p estimate as u tends to zero})
with $u=\frac{1}{\eta p\,\log(p)}$,
\[
\frac{\left|\mathcal{B}_{\alpha,p}\left(\frac{1}{\eta p\log(p)}\right)\right|}{2\,|c_{M}|r_{M}^{2p}\left|\mathcal{G}_{p}\left(\theta_{M}\right)\right|}\leq\frac{C_{\alpha}}{2\epsilon_{0}|c_{M}|\eta p\,\log(p)}<\frac{1}{6},
\]
whenever $p\geq q_{N_{0}}$ and $N_{0}$ is sufficiently large. This
completes the proof of the claim.

\end{proof}

\subsection{Conclusion of the argument}
Finally, we return to the integral representation (\ref{integral representation with Gp simplified}),
\begin{align}
\intop_{0}^{\infty}t^{2p}\,\tilde{F}_{z,\alpha}\left(\frac{\alpha}{4}+it\right)\,\cosh\left(\left(\frac{\pi}{2}-2u\right)t\right)\,dt & =-\frac{8\pi c_{M}r_{M}^{2p}}{2^{2p}}\left(1+e^{z^{2}/8}\sinh\left(\frac{z^{2}}{8}\right)\right)\mathcal{G}_{p}\left(\theta_{M}\right) \nonumber \\
\times \left\{ 1+\frac{\mathcal{B}_{\alpha,p}(u)}{2c_{M}r_{M}^{2p}\,\mathcal{G}_{p}\left(\theta_{M}\right)}+\tilde{E}\left(X,z\right)+\tilde{H}\left(X,z\right)\right\} & +\frac{2\pi\,e^{z^{2}/8}}{2^{2p}}\,\text{Re}\left[\mathcal{A}_{\alpha,p}(u,z)\right],\label{integral representation after all the claims and steps}
\end{align}
and we take there $p\in\left(q_{n}\right)_{N_{0}\leq n\leq N}$, where
$N\geq2N_{0}+2$ and $N_{0}$ is a sufficiently large number for which
(\ref{bound E and H}) and (\ref{Bound Balpha p}) hold. Replacing
$u$ by $\frac{1}{\eta q_{N}\log(q_{N})}$ in (\ref{integral representation after all the claims and steps}),
it follows from our choices of $N_{0}$ and $X$ in the previous claim
that 
\begin{equation}
1+\frac{\mathcal{B}_{\alpha,p}(u)}{2c_{M}r_{M}^{2p}\,\mathcal{G}_{p}\left(\theta_{M}\right)}+\tilde{E}\left(X,z\right)+\tilde{H}\left(X,z\right)>\frac{1}{2},\label{first inequallliiity}
\end{equation}
by (\ref{bound E and H}) and (\ref{Bound Balpha p}). Moreover, according
to Claim \ref{claim 3.2}, namely (\ref{bound for A alpha p with u depending on p}),
we also have the inequality
\begin{equation}
\frac{2\pi\,e^{z^{2}/8}}{2^{2p}}\,\left|\text{Re}\left[\mathcal{A}_{\alpha,p}\left(\frac{1}{\eta\,p\log(p)},z\right)\right]\right|<\frac{1}{4}\cdot\frac{8\pi|c_{M}|\,r_{M}^{2p}}{2^{2p}}\left(1+e^{z^{2}/8}\sinh\left(\frac{z^{2}}{8}\right)\right)\left|\mathcal{G}_{p}\left(\theta_{M}\right)\right|,\label{second inequallllliiiiittttyyyy}
\end{equation}
for any $\eta\geq\eta_{0}$ and $\eta_{0}$ sufficiently large (such
that (\ref{Final sufficient condition with all the logs}) holds).
Combining (\ref{first inequallliiity}) with (\ref{second inequallllliiiiittttyyyy})
one concludes that, for any $p\in\left(q_{n}\right)_{N_{0}\leq n\leq N}$
and $\eta_{0}$ chosen as in the proof of Claim \ref{claim 3.2}, the sequence of
moments
\begin{equation}
\left(M_{p}\right)_{p\in(q_{n})_{N_{0}\leq n\leq N}}:=\intop_{0}^{\infty}t^{2p}\,\tilde{F}_{z,\alpha}\left(\frac{\alpha}{4}+it\right)\,\cosh\left(\left(\frac{\pi}{2}-\frac{2}{\eta_{0}q_{N}\log(q_{N})}\right)t\right)\,dt\label{Moments definition}
\end{equation}
must have the same sign as the sequence defined by
\begin{equation}
\left(s_{p}\right)_{p\in(q_{n})_{N_{0}\leq n\leq N}}:=-\frac{8\pi c_{M}r_{M}^{2p}}{2^{2p}}\left(1+e^{z^{2}/8}\sinh\left(\frac{z^{2}}{8}\right)\right)\mathcal{G}_{p}(\theta_{M}).\label{sp sequence defined here}
\end{equation}
Moreover, from (\ref{first inequallliiity}) and (\ref{second inequallllliiiiittttyyyy}),
the modulus of $M_{p}$, (\ref{Moments definition}), satisfies the
inequality
\begin{align}
\left|\intop_{0}^{\infty}t^{2p}\,\tilde{F}_{z,\alpha}\left(\frac{\alpha}{4}+it\right)\,\cosh\left(\left(\frac{\pi}{2}-\frac{2}{\eta_{0}q_{N}\log(q_{N})}\right)t\right)\,dt\right| & >\frac{1}{4}\cdot\frac{8\pi|c_{M}|r_{M}^{2p}}{2^{2p}}\left(1+e^{z^{2}/8}\sinh\left(\frac{z^{2}}{8}\right)\right)\left|\mathcal{G}_{p}(\theta_{M})\right|\nonumber \\
 & =\frac{2\pi|c_{M}|r_{M}^{2p}}{2^{2p}}\left(1+e^{z^{2}/8}\sinh\left(\frac{z^{2}}{8}\right)\right)\left|\mathcal{G}_{p}(\theta_{M})\right|.\label{inequality for the modulus}
\end{align}
In order to be able to finally apply Fej\'er's lemma, we establish the
following simple bound.

\begin{claim}
There exists $\eta_{1}>\eta_{0}$ such that, for any $\eta\geq\eta_{1}$ and 
$z$ satisfying (\ref{condition first rectangle}), the following inequality takes
place
\begin{equation}
\left|\,\intop_{\eta^{2}\,q_{N}^{2}\log^{2}(q_{N})}^{\infty}t^{2p}\,\tilde{F}_{z,\alpha}\left(\frac{\alpha}{4}+it\right)\,\cosh\left(\left(\frac{\pi}{2}-\frac{2}{\eta_{0}q_{N}\log(q_{N})}\right)t\right)\,dt\,\right|<\frac{2\pi|c_{M}|r_{M}^{2p}}{2^{2p}}\left(1+e^{z^{2}/8}\sinh\left(\frac{z^{2}}{8}\right)\right)\left|\mathcal{G}_{p}(\theta_{M})\right|,\label{inequality claim 4}
\end{equation}
whenever $p\in\left(q_{n}\right)_{N_{0}\leq n\leq N}$.
\end{claim}

\begin{proof}
Using (\ref{estimate simple-1}), it is simple to check that, for
a sufficiently large $\eta_{1}$ and $\eta\geq\eta_{1}$, that the
left-hand side of (\ref{inequality claim 4}) can be bounded as follows
\begin{align*}
\intop_{\eta^{2}q_{N}^{2}\log^{2}(q_{N})}^{\infty}t^{2p}\,\tilde{F}_{z,\alpha}\left(\frac{\alpha}{4}+it\right)\,\cosh\left(\left(\frac{\pi}{2}-\frac{2}{\eta_{0}q_{N}\log(q_{N})}\right)t\right)\,dt\\
\leq e^{-\frac{\eta^{2}}{\eta_{0}}q_{N}\log(q_{N})}\intop_{\eta^{2}q_{N}^{2}\log^{2}(q_{N})}^{\infty}t^{2q_{N}}\,\left|\tilde{F}_{z,\alpha}\left(\frac{\alpha}{4}+it\right)\right|\,\exp\left(\left(\frac{\pi}{2}-\frac{1}{\eta_{0}q_{N}\log(q_{N})}\right)t\right)\,dt\\
\ll_{\alpha,z,\Lambda}e^{-\frac{\eta^{2}}{\eta_{0}}q_{N}\log(q_{N})}\intop_{\eta^{2}q_{N}^{2}\log^{2}(q_{N})}^{\infty}t^{2q_{N}+B(\alpha)}\,\exp\left(-\frac{t}{\eta_{0}q_{N}\log(q_{N})}+|z|\sqrt{t}\right)\,\,dt\\
\ll_{\alpha,z,\Lambda}e^{-\frac{\eta^{2}}{\eta_{0}}q_{N}\log(q_{N})}\intop_{\eta^{2}q_{N}^{2}\log^{2}(q_{N})}^{\infty}t^{2q_{N}+B(\alpha)}\,\exp\left(-\frac{t}{2\eta_{0}q_{N}\log(q_{N})}\right)\,dt\\
\ll_{\alpha,z,\Lambda}\left(2\eta_{0}q_{N}\log(q_{N})\right)^{2q_{N}+B(\alpha)+1}e^{-\frac{\eta^{2}}{\eta_{0}}q_{N}\log(q_{N})}\,\left(2q_{N}+[B(\alpha)]+1\right)!\,.
\end{align*}
In the above inequalities, $\ll_{\alpha,z,\Lambda}$ stands for a
constant depending on $\alpha,z$ and the pair of sequences $\left(c_{j},\lambda_{j}\right)_{j\in\mathbb{Z}\setminus\{0\}}$. 
By Stirling's formula, we have that the latter expression is bounded
by
\[
\exp\left\{ -\frac{\eta^{2}}{\eta_{0}}q_{N}\log(q_{N})+\left(2q_{N}+B(\alpha)+1\right)\log\left(2\eta_{0}q_{N}\log(q_{N})\right)+\left(2q_{N}+[B(\alpha)]+1\right)\log\left(2q_{N}+[B(\alpha)]+1\right)+O\left(q_{N}\right)\right\} 
\]
and so it is clear that, for $\eta\geq\eta_{1}>\eta_{0}$ sufficiently
large, the dominant term in the exponential will be $-\frac{\eta^{2}}{\eta_{0}}q_{N}\log(q_{N})$
and this proves (\ref{inequality claim 4}).
\end{proof}

\bigskip{}

By the previous claim and (\ref{inequality for the modulus}), we found that for each $p\in\left(q_{n}\right)_{N_{0}\leq n\leq N}$,
there exists some sufficiently large number $\eta_{1}>\eta_{0}$ such
that the sequence of moments
\[
\left(\tilde{M}_{p}\right)_{p\in(q_{n})_{N_{0}\leq n\leq N}}:=\intop_{0}^{\eta_{1}^{2}q_{N}^{2}\log^{2}(q_{N})}t^{2p}\,\tilde{F}_{z,\alpha}\left(\frac{\alpha}{4}+it\right)\,\cosh\left(\left(\frac{\pi}{2}-\frac{2}{\eta_{0}q_{N}\log(q_{N})}\right)t\right)\,dt
\]
must have the same sign as the sequence $\left(s_{p}\right)_{p\in\left(q_{n}\right)_{N_{0}\leq n\leq N}}$
defined by (\ref{sp sequence defined here}). According to Claim \ref{claim 3.1},
$s_{\ell}$ and $s_{\ell+1}$ have always distinct signs, and so the
sequence $\left(s_{p}\right)_{p\in\left(q_{n}\right)_{N_{0}\leq n\leq N}}$
has exactly $N-N_{0}-1$ sign changes. By Fej\'er's theorem, we then
have that
\begin{equation}
N_{\alpha,z}\left(\eta_{1}^{2}\,q_{N}^{2}\log^{2}(q_{N})\right)\geq N-N_{0}-1>\frac{N}{2},\label{almost at the end}
\end{equation}
since $N\geq2N_{0}+2$ by hypothesis. To conclude, we just need to
connect $q_{N}$ with $N$. We divide the final argument in two cases
(recall that we are assuming that $0<\theta_{M}<\frac{\pi}{2}$).
\begin{enumerate}
\item If $\theta_{M}/\pi\in\mathbb{Q}$, say $\theta_{M}=\pi\frac{a}{b}$,
then the explicit construction (\ref{construction of the sequence rational case})
shows that, for $N\geq3$, $N>\frac{q_{N}-q_{2}}{b}\geq\frac{q_{N}}{2b}$.
Thus, (\ref{almost at the end}) implies
\[
N_{\alpha,z}\left(\eta_{1}^{2}\,q_{N}^{2}\log^{2}(q_{N})\right)>\frac{q_{N}}{4b},
\]
proving (\ref{result proved for zeta alpha}).
\item If $\theta_{M}/\pi\notin\mathbb{Q}$ then, by Kronecker's lemma, the
sequence $\left(\left\{ n\,\frac{2\theta_{M}}{\pi}\right\} \right)_{n\in\mathbb{N}}$
is uniformly distributed on the interval $[0,1]$, which means that
\begin{equation}
\lim_{X\rightarrow\infty}\frac{\#\left\{ 1\leq m\leq X\,:\,2m\theta_{M}\in\left(\mathcal{I}^{+}\cup\mathcal{I}^{-}\right)\right\} }{X}=\frac{|\mathcal{I^{+}\cup I^{-}}|}{2\pi}=\frac{2\theta_{M}}{\pi}.\label{definition uniform distribution}
\end{equation}
Thus, taking $X=q_{N}$ in (\ref{definition uniform distribution})
we obtain, for $N\geq2N_{0}+2$ sufficiently large, 
\[
N_{\alpha,z}\left(\eta_{1}^{2}\,q_{N}^{2}\log^{2}(q_{N})\right)>\frac{\theta_{M}}{\pi}q_{N},
\]
which proves (\ref{result proved for zeta alpha}). $\blacksquare$
\end{enumerate}

\section{Lemmas for the proof of Theorem \ref{theorem 1.2}}\label{lemmas cusp forms section}

Analogously to (\ref{estimate simple-1}), we can use the estimate
(\ref{bound for confluent hypergeometric useful Rcomb}) to prove
that
\begin{equation}
\left|\tilde{G}_{z,f}\left(\frac{k}{2}+it\right)\right| \leq\sum_{j\neq0}\left|c_{j}\,\eta_{f}\left(\frac{k}{2}+i\,(t+\lambda_{j})\right)\text{Re}\left(\,_{1}F_{1}\left(\frac{k}{2}-i\,(t+\lambda_{j});\,k;\,\frac{z^{2}}{4}\right)\right)\right|\ll_{k,z}C_{\lambda}\,\sum_{j=1}^{\infty}|c_{j}|\,|t|^{B(k)}e^{-\frac{\pi}{2}|t|+|z|\sqrt{|t|}},\label{bound 1F1 and G cusp forms}
\end{equation}
for some exponent $B(k)$ depending on the weight $k$ (which comes
from convex estimates for $L(s,f)$). The next lemma is an analogue
of Lemma \ref{lemma 2.1}. For a proof, we refer to {[}\cite{RYCE}, p. 34, eq. (2.78){]}.

\begin{lemma}\label{lemma 4.1}
Let $f(\tau)$ be a holomorphic cusp form of weight $k$ for the full
modular group and $a_{f}(n)$ be its Fourier coefficients. Then for
$\text{Re}(x)>0$ and any $y\in\mathbb{C}$, the following transformation
formula takes place 
\begin{equation}
\sum_{n=1}^{\infty}a_{f}(n)\,n^{\frac{1-k}{2}}\,e^{-2\pi n\,x}\,J_{k-1}\left(\sqrt{2\pi n}\,y\right)=\frac{(-1)^{k/2}\,e^{-\frac{y^{2}}{4x}}}{x}\,\sum_{n=1}^{\infty}a_{f}(n)\,n^{\frac{1-k}{2}}\,e^{-\frac{2\pi n}{x}}\,I_{k-1}\left(\sqrt{2\pi n}\,\,\frac{y}{x}\right)\label{summation formula cusp forms lemma}
\end{equation}
or, equivalently,
\begin{equation}
\psi_{f}(x,z)=(-1)^{k/2}\,e^{-\frac{z^{2}}{4}}x^{-k}\,\psi_{f}\left(\frac{1}{x},\,iz\right),\,\,\,\,\text{Re}(x)>0,\,z\in\mathbb{C},\label{direct identity cusp form twisted}
\end{equation}
where $\psi_{f}(x,z)$ is the analogue of Jacobi's $\psi-$function
(\ref{Jacobi theta function cusp forms}), 
\begin{equation}
\psi_{f}(x,z):=(k-1)!\,\left(\sqrt{\frac{\pi x}{2}}\,z\right)^{1-k}\sum_{n=1}^{\infty}a_{f}(n)n^{\frac{1-k}{2}}\,e^{-2\pi n\,x}\,J_{k-1}\left(\sqrt{2\pi n\,x}\,z\right),\,\,\text{Re}(x)>0,\,\,z\in\mathbb{C}.\label{Jacobi theta function cusp forms lemmata section}
\end{equation}
\end{lemma}

\begin{lemma}\label{lemma 4.2}
Let $f(\tau)$ be a holomorphic cusp form of weight $k$ for the full
modular group and $\psi_{f}(x,z)$ be the generalized Jacobi's $\psi-$function
attached to it, (\ref{Jacobi theta function cusp forms lemmata section}). Assume
that $z\in\mathbb{R}$ satisfies the condition
\begin{equation}
-\frac{\sqrt{\pi}}{3}\leq z\leq\frac{\sqrt{\pi}}{3}.\label{condition cusp forms 1/6}
\end{equation}
Then there is an absolute constant $A$ and a  constant $C$ depending
only on $k$ such that, for any $0<u<\frac{\pi}{4}$ and every $p\in \mathbb{N}_{0}$,
\begin{equation}
\left|\frac{d^{p}}{du^{p}}\psi_{f}\left(ie^{-2iu},z\right)\right|<C\,\frac{2^{7p}p!}{u^{k+p}}e^{-\frac{A}{u}}.\label{bound derivative only in direction of cusp forms}
\end{equation}

\end{lemma}

\begin{proof}
The proof is exactly the same as in Lemma \ref{lemma 2.2}, so we only give a brief
sketch. Starting with Cauchy's formula,
\begin{equation}
\left[\frac{d^{p}}{du^{p}}\psi_{f}\left(ie^{-2iu},z\right)\right]_{u=u_{0}}=\frac{p!}{2\pi i}\,\intop_{C_{\lambda u_{0}}(u_{0})}\,\frac{\psi_{f}\left(ie^{-2iw},z\right)}{(w-u_{0})^{p+1}}\,dw,\label{Cauchy integral as starting point}
\end{equation}
we need to find a suitable bound for $\psi_{f}\left(ie^{-2iw},z\right),\,w\in C_{\lambda u_{0}}(u_{0})$.
This is done through Lemma \ref{lemma 4.1} above: replacing there $x$ by $2e^{-iw}\sin(w)$ and $y$ by $e^{i\left(\frac{\pi}{4}-w\right)}\,z$,
we obtain the transformation formula
\begin{align*}
\psi_{f}\left(ie^{-2iw},z\right) &=(k-1)!\,\left(\sqrt{\frac{\pi ie^{-2iw}}{2}}\,z\right)^{1-k}\sum_{n=1}^{\infty}a_{f}(n)n^{\frac{1-k}{2}}\,e^{-2\pi n\cdot2e^{-iw}\sin(w)}\,J_{k-1}\left(\sqrt{2\pi nie^{-2iw}}\,z\right)\\
 & =(k-1)!\,\left(\sqrt{\frac{\pi}{2}}e^{i\left(\frac{\pi}{4}-w\right)}z\right)^{1-k}\frac{(-1)^{k/2}\,e^{-\frac{ie^{-iw}z^{2}}{8\sin(w)}}}{2e^{-iw}\sin(w)}\,\sum_{n=1}^{\infty}(-1)^{n}a_{f}(n)\,n^{\frac{1-k}{2}}\,e^{-\frac{\pi n}{\tan(w)}}\,I_{k-1}\left(\sqrt{\frac{\pi n}{2}}\frac{e^{i\frac{\pi}{4}}z}{\sin(w)}\right).
\end{align*}
From this point on, we can easily adapt
the steps given in the proof of the bounds (\ref{first inequality in the combination}) and (\ref{bound for the modulus in fact}). Extracting the first term
of the series in the above expression,\footnote{without any loss of generality, we suppose that $a_{f}(1)\neq0$.
If $c$ is the smallest integer such that $a_{f}(c)\neq0$, then it
would be possible to enlarge the interval (\ref{interval z in cusp form case}) to $\left[-\frac{\sqrt{\pi}c}{3},\frac{\sqrt{\pi}c}{3}\right]$.} we obtain
\begin{align}
\left|\psi_{f}\left(ie^{-2iw},z\right)\right| & \leq d_{k}\,\frac{\pi^{k}e^{-\frac{z^{2}}{8}+\frac{u_{0}}{2}}}{(2u_{0})^{k}}\exp\left\{ -\frac{\pi\sin(2\text{Re}(w))}{\cosh(2\text{Im}(w))-\cos(2\text{Re}(w))}P_{w}^{\star}(1)\right\} \nonumber \\
 & \times\sum_{n=1}^{\infty}\left|\frac{a_{f}(n)}{a_{f}(1)}\right|\exp\left[-\frac{\pi\sin(2\text{Re}(w))}{\cosh(2\text{Im}(w))-\cos(2\text{Re}(w))}\left(P_{w}^{\star}(\sqrt{n})-P_{w}^{\star}(1)\right)\right],\label{bound intermediate for psi f cusp}
\end{align}
where $P_{w}^{\star}(X)$ is the real-valued polynomial
\[
P_{w}^{\star}\left(X\right):=X^{2}-\frac{|z|}{\sqrt{\pi}\sin(2\text{Re}(w))}\,\sqrt{\cosh(2\text{Im}(w))-\cos(2\text{Re}(w))}\,X+\frac{z^{2}}{8\pi}\,\frac{\sinh(2\text{Im}(w))}{\sin(2\text{Re}(w))}.
\]
Note that our expression (\ref{bound intermediate for psi f cusp}) is completely analogous to (\ref{bound for the modulus in fact}). Repeating
the steps in (\ref{bound for the series!}), we find that the series
on the right-hand side of (\ref{bound intermediate for psi f cusp})
is uniformly bounded for any $w\in C_{0.01u_{0}}(u_{0})$ and $z$ satisfying (\ref{condition cusp forms 1/6}). The resulting constant, $M_{k}$, will only dependent on $k$. Analogously to (\ref{intermediate bound})
and (\ref{bound isolated exponential term!}),
\begin{align*}
\left|\psi_{f}\left(ie^{-2iw},z\right)\right| & \leq\frac{d_{k}M_{k}e^{\frac{\pi}{8}}\,\pi^{k}e^{-\frac{z^{2}}{8}}}{(2u_{0})^{k}}\exp\left\{ -\frac{\pi\sin(2\text{Re}(w))}{\cosh(2\text{Im}(w))-\cos(2\text{Re}(w))}P_{w}^{\star}(1)\right\} \\
 & <\frac{d_{k}M_{k}e^{\frac{\pi}{8}}\,\pi^{k}e^{-\frac{z^{2}}{8}}}{(2u_{0})^{k}}\exp\left[-\frac{1}{u_{0}}\left\{ \frac{1-\lambda}{e^{\frac{\pi^{2}\lambda^{2}}{24}}(1+\lambda)^{2}}-\frac{\pi|z|}{2(1-\lambda)}-\frac{\pi z^{2}\lambda e^{\frac{\pi^{2}\lambda^{2}}{24}}}{16(1-\lambda)^{2}}\right\} \right]\\
 & <\frac{d_{k}M_{k}e^{\frac{\pi}{8}}\,\pi^{k}e^{-\frac{z^{2}}{8}}}{(2u_{0})^{k}}\exp\left[-\frac{1}{u_{0}}\left\{ \frac{1-\lambda}{e^{\frac{\pi^{2}\lambda^{2}}{24}}(1+\lambda)^{2}}-\frac{\pi^{\frac{3}{2}}}{6(1-\lambda)}-\frac{\pi^{2}\lambda e^{\frac{\pi^{2}\lambda^{2}}{24}}}{144(1-\lambda)^{2}}\right\} \right].
\end{align*}
Thus, if we pick $\lambda=0.01$, the term on the braces is greater than $0.03$: this implies that
\[
\left|\psi_{f}\left(ie^{-2iw},z\right)\right|<\frac{d_{k}M_{k}e^{\frac{\pi}{8}}\,\pi^{k}e^{-\frac{z^{2}}{8}}}{(2u_{0})^{k}}\exp\left[-\frac{0.03}{u_{0}}\right]=\frac{C}{u_{0}^{k}}\,e^{-\frac{A}{u_{0}}},
\]
where $C$ depends on $k$ and $A$ is an absolute constant, say $A=0.03$.
Using Cauchy's formula (\ref{Cauchy integral as starting point})
and mimicking the steps in (\ref{steps to get bound derivativeees}),
we find that (\ref{bound derivative only in direction of cusp forms})
must hold for every $p\in\mathbb{N}_{0}$.
\end{proof}

We finish this section with two lemmas analogous to Lemmas \ref{lemma 2.3} and
\ref{lemma 2.4}. Since their proofs are exactly the same, we shall omit them. 

\begin{lemma}\label{lemma 4.3}
Let $g:\,\mathbb{C}\longmapsto\mathbb{C}$ be an analytic function.
If $z\in\left[-\frac{1}{6}\sqrt{\frac{\pi\alpha}{2}},\frac{1}{6}\sqrt{\frac{\pi\alpha}{2}}\right]$
then, for arbitrary $p\in\mathbb{N}_{0}$ and $0<u<\frac{\pi}{4}$,
the following relation holds
\begin{equation}
\left|\frac{d^{2p}}{du^{2p}}\left\{ g(u)\,\psi_{f}\left(ie^{-2iu},z\right)\right\} \right|<D\frac{2^{14p}(2p)!e^{-\frac{A}{u}}}{u^{k+2p}}\left\Vert g\right\Vert _{L^{\infty}(C_{1}(0))},\label{bound derivative with analytic function cussp form case}
\end{equation}
where $A$ is some absolute constant and $D$ only depends on $k$.
\end{lemma}

\begin{lemma} \label{lemma 4.4}
Let $-\frac{\pi}{4}<\omega<\frac{\pi}{4}$, $p\in\mathbb{N}$ and
$\tilde{G}_{z,f}\left(\frac{k}{2}+it\right)$ be the function defined
by (\ref{function defining coooombinations}). Then the following
integral representations hold\footnote{Note that the integral is written in terms of $i^{-k/2}\tilde{G}_{z,f}\left(\frac{k}{2}+it\right)$
because it defines a real function of $t\in\mathbb{R}$ which is even if $k\equiv0 \mod4$ or odd if $k\equiv 2\mod4$.}
\begin{equation}
\intop_{0}^{\infty}i^{-\frac{k}{2}}\,t^{2p}\tilde{G}_{z,f}\left(\frac{k}{2}+it\right)\,\cosh\left(2\omega t\right)dt=\frac{2\pi\,e^{\frac{z^{2}}{4}}}{2^{2p}}\,\text{Re}\,\left(i^{-\frac{k}{2}}\,\frac{d^{2p}}{d\omega^{2p}}\left\{ \sum_{j\neq0}c_{j}\,e^{i\omega k-2\omega\lambda_{j}}\,\psi_{f}\left(e^{2i\omega},z\right)\right\} \right),\,\,\,\,k\equiv 0\mod4\label{Formula in Lemma 4}
\end{equation}
and 
\begin{equation}
\intop_{0}^{\infty}i^{-\frac{k}{2}}\,t^{2p+1}\tilde{G}_{z,f}\left(\frac{k}{2}+it\right)\,\cosh\left(2\omega t\right)dt=\frac{2\pi}{2^{2p}}\,e^{\frac{z^{2}}{4}}\,\text{Re}\,\left(i^{-\frac{k}{2}}\,\frac{d^{2p+1}}{d\omega^{2p+1}}\left\{ \sum_{j\neq0}c_{j}\,e^{i\omega k-2\omega\lambda_{j}}\,\psi_{f}\left(e^{2i\omega},z\right)\right\} \right),\,\,\,\,k\equiv 2 \mod4,\label{formula odd case moments}
\end{equation}
where $\psi_{f}\left(x,z\right)$ is the ``cusp form analogue''
of Jacobi's $\psi-$function (\ref{Jacobi theta function cusp forms lemmata section}).
\end{lemma}

\section{Proof of Theorem \ref{theorem 1.2}}\label{proof of theorem 1.2 section}

In this section we follow closely the author's variant of de la Vall\'ee
Poussin's method \cite{RHALF, Poussin_zeros}. The details of the proof are given for the case where $k\equiv 0 \mod 4$ and a similar argument for $k\equiv 2 \mod4 $ will be devised at the end of the proof. 
Let $\left(\rho_{n}\right)_{n\in\mathbb{N}}$
be the sequence of zeros of odd order of $\tilde{G}_{z,f}(s)$ such
that $\text{Re}(\rho_{n})=\frac{k}{2}$. Then we can write $\rho_{n}:=\frac{k}{2}+i\tau_{n}$,
with $\tau_{n}>0$ being an increasing sequence\footnote{Note that if $\tilde{G}_{z,f}\left(\frac{k}{2}\right)=0$, we are
excluding this real zero from the sequence.}. If we show that there is some $h>0$ such that, for infinitely many
values of $n$, $\tau_{n}<h\,n^{2}$, our Theorem \ref{theorem 1.2} is proved. This
is the case because if we choose the sequence $T_{n}:=hn^{2}$, then
we find that $N_{f,z}\left(T_{n}\right)\geq N_{f,z}(\tau_{n})=n=\sqrt{\frac{T_{n}}{h}}$.
This establishes that
\begin{equation}
\limsup_{T\rightarrow\infty}\,\frac{N_{f,z}(T)}{\sqrt{T}}>\frac{1}{\sqrt{h}},\,\,\,\,\text{or equivalently }N_{f,z}(T)=\Omega\left(T^{\frac{1}{2}}\right).\label{explicit lower bound with explicit constant}
\end{equation}

Hence, for the sake of contradiction, let us assume that there is
some $N_{0}$ such that, for every $n\geq N_{0}$ and any $h>0$,
$\tau_{n}\geq h\,n^{2}$. We will now show that there exists some $h$ large enough that this assumption is contradicted. Indeed,
if we construct the real and entire function\footnote{The infinite product can be written because, due to Theorem 1.4 of
\cite{RYCE}, $\tilde{G}_{z,f}(s)$ has infinitely many zeros at the
critical line $\text{Re}(s)=\frac{k}{2}$.}
\begin{equation}
\varphi_{f,z}\left(y\right)=\prod_{\ell=1}^{\infty}\left(1-\frac{y^{2}}{\tau_{\ell}^{2}}\right)=\sum_{\ell=0}^{\infty}(-1)^{\ell}a_{2\ell}\,y^{2\ell},\label{Powert series for varphi (y)}
\end{equation}
we see that $a_{0}=1$ and, for $\ell\geq1$,
\begin{equation}
a_{2\ell}=\sum_{r_{1}\geq1}\sum_{r_{2}>r_{1}}...\sum_{r_{\ell}>r_{\ell-1}}\frac{1}{\tau_{r_{1}}^{2}\cdot...\cdot\tau_{r_{\ell}}^{2}}=\sum_{1\leq r_{1}<r_{2}<...<r_{\ell}}\frac{1}{\tau_{r_{1}}^{2}\cdot...\cdot\tau_{r_{\ell}}^{2}},\label{coefficient for the bound}
\end{equation}
where we are summing over $(r_{1},...,r_{\ell})\in\mathbb{N}^{\ell}\text{ such that }r_{1}<r_{2}<...<r_{\ell}$.
Note that, in the $k^{\text{th}}$ nested series in (\ref{coefficient for the bound}),
the index $r_{k}$ always satisfies $r_{k}\geq k$, due to the condition
$r_{k}>r_{k-1}>...>r_{1}\geq1$. From this point on, we just need
to find a suitable bound for $a_{2j}$. Considering the nested sum
above, we have two possibilities: if $1\leq k\leq\ell$ and $r_{k}\geq N_{0}$,
we know by the contradiction hypothesis that $\tau_{r_{k}}^{-2}\leq\frac{r_{k}^{-4}}{h^{2}}$,
while if $1\leq r_{k}\leq N_{0}-1$, we have that $\tau_{r_{k}}^{-2}\leq\frac{r_{k}^{-4}}{h^{\star2}}$
where $h^{\star}:=\min_{1\leq n\leq N_{0}-1}\left\{ \frac{\tau_{n}}{n^{2}}\right\} $.
But then if follows from the argument given in {[}\cite{RHALF}, p.
16, eq. (4.6){]} that for some $\mathscr{B}$ (only depending on $h$
and $N_{0}$), 
\begin{equation}
a_{2\ell}\leq\frac{\mathscr{B}}{h^{2\ell}}\,\sum_{1\leq r_{1}<r_{2}<...<r_{\ell}}\frac{1}{r_{1}^{4}\cdot...\cdot r_{\ell}^{4}}.\label{coefficient in the proof new}
\end{equation}
Defining the new coefficient on the right-hand side of (\ref{coefficient in the proof new})
\begin{equation}
b_{2\ell}:=\sum_{1\leq r_{1}<...<r_{\ell}}\,\frac{1}{r_{1}^{4}\cdot...\cdot r_{\ell}^{4}},\label{b2j definition}
\end{equation}
and following the steps given in [\cite{RHALF}, p. 17], this sequence
of numbers comes from Euler's infinite product representation of the
function 
\begin{equation}
\frac{\sinh(\pi\sqrt{y})\sin(\pi\sqrt{y})}{\pi^{2}y}=\prod_{\ell=1}^{\infty}\left(1+\frac{y}{\ell^{2}}\right)\prod_{j=1}^{\infty}\left(1-\frac{y}{\ell^{2}}\right)=\prod_{j=1}^{\infty}\left(1-\frac{y^{2}}{\ell^{4}}\right):=1+\sum_{\ell=1}^{\infty}(-1)^{\ell}b_{2\ell}\,y^{2\ell}.\label{sinh sin product}
\end{equation}
Thus, we can get precise information about $b_{2\ell}$ (and, consequently,
about $a_{2\ell}$) by interpreting them as the coefficients of the
Taylor series for the function on the left-hand side of (\ref{sinh sin product}). A standard computation of these coefficients gives 
\begin{equation}
b_{2\ell}:=\sum_{1\leq r_{1}<...<r_{\ell}}\,\frac{1}{r_{1}^{4}\cdot...\cdot r_{\ell}^{4}}=\frac{2^{2\ell+1}\pi^{4\ell}}{(4\ell+2)!}\implies a_{2\ell}\leq\mathscr{B}\,\frac{2^{2\ell+1}\pi^{4\ell}}{h^{2\ell}(4\ell+2)!}.\label{bound for this guy}
\end{equation}

Our proof will be concluded by seeing that (\ref{bound for this guy})
contradicts (\ref{Formula in Lemma 4}) and the bounds found
in Lemma \ref{lemma 4.2} above. Recall that, under the assumption $k\equiv 0 \mod4$, $i^{-k/2}\tilde{G}_{z,f}\left(\frac{k}{2}+it\right)$
is an even and real function and $\varphi_{f,z}(t)$ has the same
odd zeros as this function. Hence, by this construction, $i^{-k/2}\tilde{G}_{z,f}\left(\frac{k}{2}+it\right)\,\varphi_{f,z}(t)$
must be a real and even function with constant sign for any $t\in\mathbb{R}$.
Considering the continuous $Q_{f}:\,\left(0,\,\frac{\pi}{4}\right)\longmapsto\mathbb{R}$
defined by the integral 
\[
Q_{f}(u):=\intop_{0}^{\infty}i^{-\frac{k}{2}}\tilde{G}_{z,f}\left(\frac{k}{2}+it\right)\varphi_{f,z}(t)\,\cosh\left(\left(\frac{\pi}{2}-2u\right)t\right)\,dt,\,\,\,\,0<u<\frac{\pi}{4},
\]
we then have that $|Q_{f}(u)|$ will be positive decreasing. Our proof
will now show that this cannot be true if the bound (\ref{bound for this guy})
takes place.

If we use the power series (\ref{Powert series for varphi (y)}) for
$\varphi_{f,z}(t)$, we see that $Q_{f}(u)$ can be written as an
infinite series of the form
\begin{align}
Q_{f}(u)&=\intop_{0}^{\infty}i^{-\frac{k}{2}}\tilde{G}_{z,f}\left(\frac{k}{2}+it\right)\,\varphi_{f,z}(t)\,\cosh\left(\left(\frac{\pi}{2}-2u\right)t\right)\,dt\nonumber\\
&=\sum_{j=0}^{\infty}(-1)^{\ell}a_{2\ell}\,\intop_{0}^{\infty}i^{-\frac{k}{2}}\tilde{G}_{z,f}\left(\frac{k}{2}+it\right)\,t^{2\ell}\,\cosh\left(\left(\frac{\pi}{2}-2u\right)t\right)\,dt.\label{the interchange necessary for the proof}
\end{align}
Note that the interchange of the orders of summation and integral
in (\ref{the interchange necessary for the proof}) comes from Fubini's
theorem and the bounds (\ref{bound for this guy}): indeed 
\begin{align*}
\intop_{0}^{\infty}\sum_{\ell=0}^{\infty}|a_{2\ell}|t^{2\ell}\,\left|\tilde{G}_{z,f}\left(\frac{k}{2}+it\right)\right|\,\cosh\left(\left(\frac{\pi}{2}-2u\right)t\right)&\,dt \leq\mathscr{B}\,\intop_{0}^{\infty}\,\sum_{j=0}^{\infty}\frac{2^{2\ell+1}\pi^{4\ell}}{h^{2\ell}(4\ell+2)!}t^{2\ell}\,\left|\tilde{G}_{z,f}\left(\frac{k}{2}+it\right)\right|\,\cosh\left(\left(\frac{\pi}{2}-2u\right)t\right)\,dt\\
 & \leq2\mathcal{\mathscr{B}}\,\intop_{0}^{\infty}\sum_{j=0}^{\infty}\frac{1}{(4\ell)!}\,\left(\frac{2\pi^{2}t}{h}\right)^{2\ell}\left|\tilde{G}_{z,f}\left(\frac{k}{2}+it\right)\right|\,\cosh\left(\left(\frac{\pi}{2}-2u\right)t\right)\,dt\\
 & \leq2\mathscr{B}\,\intop_{0}^{\infty}\sum_{j=0}^{\infty}\frac{1}{\ell!}\,\left(\pi\sqrt{\frac{2t}{h}}\right)^{\ell}\left|\tilde{G}_{z,f}\left(\frac{k}{2}+it\right)\right|\,\cosh\left(\left(\frac{\pi}{2}-2u\right)t\right)\,dt\\
 & <2\mathscr{B}\,\intop_{0}^{\infty}\left|\tilde{G}_{z,f}\left(\frac{k}{2}+it\right)\right|\exp\left(\pi\sqrt{\frac{2t}{h}}+\left(\frac{\pi}{2}-2u\right)t\right)\,dt\\
 & \ll_{k,z,\Lambda}\,\intop_{0}^{\infty}|t|^{B(k)}\,\exp\left(-2ut+\left(\pi\sqrt{\frac{2}{h}}+|z|\right)\sqrt{t}\right)\,dt<\infty,
\end{align*}
where in the last step we have used the estimate (\ref{bound 1F1 and G cusp forms}).
Having assured that we can perform the operation (\ref{the interchange necessary for the proof}),
it now follows from Lemma \ref{lemma 4.4} above that   
\begin{align}
Q_{f}(u)=\sum_{j=0}^{\infty}(-1)^{\ell}a_{2\ell}\,\intop_{0}^{\infty}i^{-k/2}\tilde{G}_{z,f}\left(\frac{k}{2}+it\right)\,t^{2\ell}\,\cosh\left(\left(\frac{\pi}{2}-2u\right)t\right)\,dt\nonumber \\
=\sum_{\ell=0}^{\infty}(-1)^{\ell}a_{2\ell}\,\frac{2\pi}{2^{2\ell}}\,e^{\frac{z^{2}}{4}}\,\text{Re}\,\left(i^{-\frac{k}{2}}\frac{d^{2\ell}}{du^{2\ell}}\left\{ \sum_{j\neq0}c_{j}\,e^{i\left(\frac{\pi}{4}-u\right)k-\left(\frac{\pi}{2}-2u\right)\lambda_{j}}\,\psi_{f}\left(ie^{-2iu},z\right)\right\} \right).\label{formula for Qf(u)}
\end{align}
Using Lemma \ref{lemma 4.3} together with the estimate for $a_{2\ell}$ (\ref{bound for this guy}),
we can bound uniformly the previous series with respect to $u$. Indeed,
since the function
\[
g(u)=e^{i\left(\frac{\pi}{4}-u\right)k}\,\sum_{j\neq0}c_{j}\,e^{-\left(\frac{\pi}{2}-2u\right)\lambda_{j}}
\]
is analytic and $\left\Vert g(u)\right\Vert _{L^{\infty}(C_{0}(1))}\leq\mathcal{M}$,
with $\mathcal{M}$ only depending on the sequence $\sum_{j\neq0}|c_{j}|$
(see (\ref{bound for the h alpha Linfinity}) above), it follows from 
(\ref{bound derivative with analytic function cussp form case}) that
\begin{equation}
\left|\frac{d^{2\ell}}{du^{2\ell}}\left\{ g(u)\,\psi_{f}\left(ie^{-2iu},z\right)\right\} \right|<\mathcal{D}\frac{2^{14\ell}(2\ell)!e^{-\frac{A}{u}}}{u^{k+2\ell}},\label{estimate with respect to derivative lemma 2.3}
\end{equation}
for some constant $\mathcal{D}$ depending on the weight of the cusp
form and on the sequences $\left(c_{j}\right)_{j\in\mathbb{N}}$ and
on an upper bound for the sequence $\left(\lambda_{j}\right)_{j\in\mathbb{N}}$.
Thus, returning to (\ref{formula for Qf(u)}) and invoking (\ref{bound for this guy})
\begin{align}
|Q_{f}(u)| & <2\pi\mathcal{D}e^{\frac{z^{2}}{4}}\frac{e^{-A/u}}{u^{k}}\,\sum_{\ell=0}^{\infty}|a_{2\ell}|\frac{2^{12\ell}(2\ell)!}{u^{2\ell}}\leq2\pi\mathscr{C}e^{\frac{z^{2}}{4}}\frac{e^{-A/u}}{u^{k}}\sum_{\ell=0}^{\infty}\,\frac{2^{14\ell+1}\pi^{4\ell}(2\ell)!}{h^{2\ell}(4\ell+2)!}\cdot\frac{1}{u^{2\ell}}\nonumber \\
 & =\mathscr{C}\pi^{\frac{3}{2}}e^{\frac{z^{2}}{4}}\frac{e^{-A/u}}{u^{k}}\,\sum_{\ell=0}^{\infty}\frac{2^{10\ell}\pi^{4\ell}}{\Gamma\left(2\ell+\frac{3}{2}\right)\left(2\ell+1\right)\left(hu\right)^{2\ell}}<\mathscr{C}\pi^{\frac{3}{2}}e^{\frac{z^{2}}{4}}\frac{e^{-A/u}}{u^{k}}\,\cosh\left(\frac{32\pi^{2}}{hu}\right)\nonumber \\
 & <\mathscr{C}\frac{\pi^{\frac{3}{2}}e^{\pi/36}}{u^{k}}\exp\left(-\left(A-\frac{32\pi^{2}}{h}\right)\frac{1}{u}\right),\label{Qf(u) bound at lllas alsht}
\end{align}
where we have used the fact that $|z|<\frac{1}{3}\sqrt{\pi}$, (\ref{interval z in cusp form case}). From
(\ref{Qf(u) bound at lllas alsht}), if we choose the constant $h$
in such a way that $h>\frac{32\pi^{2}}{A}$, then
\[
\lim_{u\rightarrow0^{+}}\left|Q_{f}(u)\right|<\mathscr{C}\pi^{\frac{3}{2}}e^{\pi/36}\,\lim_{u\rightarrow0^{+}}u^{-k}\,\exp\left(-\left(A-\frac{32\pi^{2}}{h}\right)\frac{1}{u}\right)=0
\]
which contradicts the fact that $|Q_{f}(u)|$ is positive decreasing.
Consequently, we have found $h>\frac{32\pi^{2}}{A}$ such that $\tau_{n}<h\,n^{2}$ for infinitely many values of $n$. This shows (\ref{explicit lower bound with explicit constant})
and so we complete the proof of our Theorem. Finally, we note that we can replace the value of $d$ in (\ref{lim sup estimate}) by an explicit constant. Using the value $A=0.03$ (calculated in the proof of Lemma \ref{lemma 4.2}) and taking $h=\frac{36\pi^{2}}{0.03}$, we find from (\ref{explicit lower bound with explicit constant}) that $d=\frac{1}{36\pi}$ works. 
\bigskip{}

In order to consider the case where $k\equiv2\mod4$, the argument
is the same with a small modification. Instead of considering $Q_{f}(u)$
we will study the similar function,
\[
P_{f}(u):=\intop_{0}^{\infty}i^{-\frac{k}{2}}\,t\,\tilde{G}_{z,f}\left(\frac{k}{2}+it\right)\varphi_{f,z}(t)\,\cosh\left(\left(\frac{\pi}{2}-2u\right)t\right)\,dt,\,\,\,\,0<u<\frac{\pi}{4},
\]
where $\varphi_{f,z}(t)$ is given by (\ref{Powert series for varphi (y)}). Invoking the second integral representation (\ref{formula odd case moments}) and using the bounds for $a_{2j}$ found in (\ref{bound for this guy}), one can prove that $|P_{f}(u)|\rightarrow0$ as $u\rightarrow0^{+}$, which assures the result. $\blacksquare$

\section{Concluding Remarks}\label{concluding remarks section}
In this paper we have presented estimates for the number of critical
zeros of the arbitrary shifted combinations
\begin{equation}
\tilde{F}_{z,\alpha}(s):=\sum_{j\neq0}c_{j}\,\eta_{\alpha}\left(s+i\lambda_{j}\right)\,\left\{ _{1}F_{1}\left(\frac{\alpha}{2}-s-i\lambda_{j};\,\frac{\alpha}{2};\,\frac{z^{2}}{4}\right)+\,_{1}F_{1}\left(\frac{\alpha}{2}-\overline{s}+i\lambda_{j};\,\frac{\alpha}{2};\,\frac{z^{2}}{4}\right)\right\} ,\label{conclusion combination}
\end{equation}
\begin{equation}
\tilde{G}_{z,f}(s):=\sum_{j\neq0}c_{j}\,\eta_{f}\left(s+i\lambda_{j}\right)\,\left\{ _{1}F_{1}\left(k-s-i\lambda_{j};\,k;\,\frac{z^{2}}{4}\right)+\,_{1}F_{1}\left(k-\overline{s}+i\lambda_{j};\,k;\,\frac{z^{2}}{4}\right)\right\} .\label{conclusion combination cusp}
\end{equation}
The proofs of both estimates relied on the summation formulas (\ref{final formula for 1f1 theorem-1})
and (\ref{summation formula cusp forms lemma}). One may wonder how
to extend our results in different directions. For example, one may
ask if we can relax the class of hypergeometric functions appearing
in the combinations (\ref{conclusion combination}) and (\ref{conclusion combination cusp}).
Note that the second parameter of the hypergeometric functions above
depends on the Dirichlet series that preceeds them: in the first case,
it is equal to $\alpha/2$ and in the second is $k$. Therefore, one
may consider the study of the critical zeros of the functions,
\begin{equation}
\tilde{F}_{z,\alpha}(s;\nu):=\sum_{j\neq0}c_{j}\,\eta_{\alpha}\left(s+i\lambda_{j}\right)\,\left\{ _{1}F_{1}\left(\frac{\alpha}{2}-s-i\lambda_{j};\,\nu+1;\,\frac{z^{2}}{4}\right)+\,_{1}F_{1}\left(\frac{\alpha}{2}-\overline{s}+i\lambda_{j};\,\nu+1;\,\frac{z^{2}}{4}\right)\right\} ,\label{conclusion combination-1}
\end{equation}
\begin{equation}
\tilde{G}_{z,f}(s;\nu):=\sum_{j\neq0}c_{j}\,\eta_{f}\left(s+i\lambda_{j}\right)\,\left\{ _{1}F_{1}\left(k-s-i\lambda_{j};\,\nu+1;\,\frac{z^{2}}{4}\right)+\,_{1}F_{1}\left(k-\overline{s}+i\lambda_{j};\,\nu+1;\,\frac{z^{2}}{4}\right)\right\} ,\label{conclusion combination cusp-1}
\end{equation}
where $\nu$ is now a real parameter independent of $\alpha$ (in
the first case) and of $k$ in the second.
We remark that it is indeed possible to get analogues of Theorems \ref{theorem 1.1} and \ref{theorem 1.2} to the shifted combinations (\ref{conclusion combination-1})
and (\ref{conclusion combination cusp-1}) provided $\nu\geq\frac{\alpha}{2}-1$
in the first case and $\nu\geq k-1$ in the second. However, instead
of using the summation formulas (\ref{final formula for 1f1 theorem-1})
and (\ref{summation formula cusp forms lemma}), we will need more
general transformations. Indeed, it is not difficult to show\footnote{The proof of this summation formula is quite standard and it follows
the same steps as the ones given in {[}\cite{RYCE}, pp. 18-19{]}. Although
our formulas (\ref{first identity humber zeta alpha}) and (\ref{second identity humbert cusp forms})
do not have an intrinsic interest themselves, they can be used to
establish new summation formulas for $\sum_{n=1}^{\infty}r_{k}(n)\,n^{\frac{\nu-\mu}{2}}\,I_{\mu}\left(Y\sqrt{n}\right)\,K_{\nu}\left(X\sqrt{n}\right)$,
where $X>Y$ and $\text{Re}(\mu),\text{Re}(\nu)>0$. Such a formula
will be given in a forthcoming investigation and it constitutes a
generalization of a formula of Berndt, Dixit, Kim and Zaharescu \cite{sum_of_squares_berndt}.} that, when $\text{Re}(\nu)>-1$ and $\text{Re}(x)>0$, $y\in\mathbb{C}$,
the following summation formula takes place
\begin{align}
\sum_{n=1}^{\infty}(-1)^{n}r_{\alpha}(n)\,n^{-\nu/2}\,e^{-\pi nx}\,J_{\nu}(\sqrt{\pi n}y)=-\frac{y^{\nu}\pi^{\nu/2}}{2^{\nu}\Gamma(\nu+1)}\nonumber\\
+\frac{\pi^{\nu/2}y^{\nu}e^{-\frac{y^{2}}{4x}}}{2^{\nu}\Gamma(\nu+1)x^{\alpha/2}}\,\sum_{n=1}^{\infty}\tilde{r}_{\alpha}(n)e^{-\frac{\pi}{x}\left(n+\frac{\alpha}{4}\right)}\,\Phi_{3}\left(1-\frac{\alpha}{2}+\nu;\nu+1;\frac{y^{2}}{4x},\frac{\pi y^{2}}{4x^{2}}\left(n+\frac{\alpha}{4}\right)\right),\label{first identity humber zeta alpha}
\end{align}
where $\Phi_{3}(b;c;w,z)$ is the usual Humbert function,
\begin{equation}
\Phi_{3}(b;c;\,w,z)=\sum_{k,m=0}^{\infty}\frac{(b)_{k}}{(c)_{k+m}}\frac{w^{k}z^{m}}{k!\,m!},\label{definition humbert}
\end{equation}
whose series converges absolutely for any $w,z\in\mathbb{C}$. By using (\ref{first identity humber zeta alpha}) instead of (\ref{final formula for 1f1 theorem-1}),
we can establish a lower bound for the number of critical zeros of
the function (\ref{conclusion combination-1}). Analogously, Theorem \ref{theorem 1.2} can be extended to shifted combinations of the form (\ref{conclusion combination cusp-1})
by using the transformation formula
\begin{equation}
\sum_{n=1}^{\infty}a_{f}(n)\,n^{-\frac{\nu}{2}}\,e^{-2\pi nx}\,J_{\nu}\left(\sqrt{2\pi n}\,y\right)=\frac{(-1)^{k/2}\pi^{\nu/2}y^{\nu}x^{-k}}{2^{\nu/2}\Gamma(\nu+1)}e^{-\frac{y^{2}}{4x}}\,\sum_{n=1}^{\infty}a_{f}(n)e^{-\frac{2\pi n}{x}}\,\Phi_{3}\left(1-k+\nu;\nu+1;\frac{y^{2}}{4x},\frac{\pi y^{2}n}{2x^{2}}\right).\label{second identity humbert cusp forms}
\end{equation}

Note that (\ref{first identity humber zeta alpha}) and (\ref{second identity humbert cusp forms})
reduce to (\ref{final formula for 1f1 theorem-1}) and (\ref{summation formula cusp forms lemma})
when $\nu=\frac{\alpha}{2}-1$ and $\nu=k-1$ respectively. This is
the case due to the well-known reduction formula for Humbert's function
\[
\Phi_{3}\left(0;c;w,z\right)=\Gamma(c)\,z^{-\frac{c-1}{2}}I_{c-1}\left(2\sqrt{z}\right),
\]
which can be immediately established from the power series (\ref{definition humbert}).
Analogous theorems can be established when the combination of confluent
hypergeometric functions $_{1}F_{1}\left(a;c;z\right)$ is replaced
by $_{2}F_{2}\left(a,b;c,d;z\right)$ and Gauss' hypergeometric function
$_{2}F_{1}(a,b;c;z)$.

\bigskip{}

Another direction that this work can take concerns even more general
shifted combinations of completed Dirichlet series. To continue, recall
again that $\eta(s):=\pi^{-s/2}\Gamma(s/2)\,\zeta(s)$. In the paper
\cite{DKMZ}, the authors have posed the following interesting problem:
if one considers a combination of shifted products of the form
\begin{equation}
H_{z}(s):=\sum_{m,n=1}^{\infty}c_{m,n}\,\eta\left(s+i\lambda_{m}\right)\,\eta\left(s+i\lambda_{n}\right)\,\text{Re}\left(_{1}F_{1}\left(\frac{1-s-i\lambda_{m}}{2};\,\frac{1}{2};\,\frac{z^{2}}{4}\right)\right)\,\text{Re}\left(_{1}F_{1}\left(\frac{1-s-i\lambda_{n}}{2};\,\frac{1}{2};\,\frac{z^{2}}{4}\right)\right),\label{quadratic form in combination}
\end{equation}
under what circumstances will $H_{z}(s)$ have infinitely many zeros
on the critical line? Even when restrict this problem to $z=0$, it
seems complicated to formulate a result covering this case. Indeed,
consider the case where $c_{m,n}:=d_{m}\,\delta_{m,n}$, $d_{m}\geq0$
and $z=0$: then we are dealing with the study of the zeros of the function
\begin{equation}
\mathcal{H}(s)=\sum_{m=1}^{\infty}d_{m}\,\eta^{2}\left(s+i\lambda_{m}\right),\,\,\,\,d_{m}\geq0.\label{H(s) positivity consideration}
\end{equation}
It is simple to give an example where $\mathcal{H}(s)$ does not have
a single zero at the critical line $\text{Re}(s)=\frac{1}{2}$. For
some sufficiently large $M$, consider the finite sequence $\lambda_{m}=m,\,\,\,1\leq m\leq M$
and assume that $d_{j}=0$ when $m\geq M+1$. If $\mathcal{H}(s)=\sum_{1\leq m\leq M}d_{m}\,\eta^{2}\left(s+i\,m\right)$
has a zero on the critical line, say $s=\frac{1}{2}+i\tau$, then
due to the non-negativity of $d_{m}$, we have that 
\begin{equation}
\eta\left(\frac{1}{2}+i(\tau+m)\right)=0,\,\,\,\,\,\text{for }1\leq m\leq M.\label{condition periodic subsequence zeros}
\end{equation}
But this is impossible for $N$ large enough because every arithmetical
sequence contains infinitely many elements which are not zeros of
$\zeta(z)$.\footnote{This is a beautiful observation (with a simple proof) made by Putnam \cite{putnam_I, putnam_II}. Note that, if we pick $\lambda_{m}:=\frac{2\pi m}{\log(2)}$,
then our reasoning would also contradict a very interesting result
due to Steuding and Wegert \cite{steuding_wegert}, which states that
\[
\frac{1}{M}\sum_{0\leq m<M}\zeta\left(\frac{1}{2}+i\tau+i\frac{2\pi m}{\log(2)}\right)=\frac{1}{1-2^{-\frac{1}{2}-i\tau}}+O\left(\frac{\log(M)}{\sqrt{M}}\right),\,\,\,\,M\rightarrow\infty.
\]
} Therefore, it seems difficult to assure a theorem of the form given
here for combinations of shifted products. 
However, one may conjecture that a combination of the form (\ref{H(s) positivity consideration})
has infinitely critical zeros for ``infinitely many'' pairs of bounded
shifts.

\begin{conjecture}
Let $(c_{j,k})_{j,k\in\mathbb{\mathbb{N}}}$ be a double sequence
of non-zero real numbers such that $\sum_{j,k=1}^{\infty}\,|c_{j,k}|<\infty$
and $(\lambda_{j})_{j\in\mathbb{N}}$, $(\lambda_{k}^{\prime})_{k\in\mathbb{N}}$
be two bounded sequences of distinct real numbers that attain their
bounds. Then there are infinitely many $\tau\neq0$ such that the function
\[
G(s;\,\tau):=\sum_{j,k=1}^{\infty}c_{j,k}\,\eta_{\alpha}\left(s+i\lambda_{j}\right)\,\eta_{\alpha}\left(s+i\tau+i\lambda_{k}^{\prime}\right)
\]
has infinitely many zeros at the critical line $\text{Re}(s)=\frac{\alpha}{4}$.
\end{conjecture}

\bigskip{}

Finally, we should also mention a case where a zero counting problem
of the form considered in this paper has a better estimate than those
given by our Theorems. Under the conditions of Theorem \ref{theorem 1.1},
one may consider the problem of counting the zeros of the function
\begin{equation}
\mathcal{Z}\left(t\right):=\sum_{j=1}^{\text{\ensuremath{\infty}}}c_{j}\,Z\left(t+\lambda_{j}\right),\label{combinations Hardy Z function}
\end{equation}
where $Z(t)$ represents Hardy's $Z-$function \cite{ivic_hardyz},
\begin{equation}
Z(t):=\zeta\left(\frac{1}{2}+it\right)\left(\chi\left(\frac{1}{2}+it\right)\right)^{-1/2},\,\,\,\,\chi(s):=2^{s}\pi^{s-1}\sin\left(\frac{\pi s}{2}\right)\Gamma(1-s).\label{Definition Hardy Z function}
\end{equation}
Building on Ingham's work \cite{ingham} and following an approach very
similar to Atkinson's \cite{atkinson}, Hall {[}\cite{hall}, p. 103,
Theorem 5{]} deduced a curious integral formula involving the product
of $\mathcal{Z}(t)$ with a shifted version of itself. He proved the
asymptotic formula
\begin{equation}
\intop_{0}^{T}\mathcal{Z}\left(t\right)\mathcal{Z}\left(t+\frac{\beta}{\log(T)}\right)\,dt=\sum_{j=1}^{\infty}c_{j}^{2}\cdot\frac{\sin(\beta/2)}{\beta/2}\,T\log(T)+O(T),\label{formula of Hall similar to Inghams}
\end{equation}
where $0<\beta<1$. A formula like (\ref{formula of Hall similar to Inghams})
was derived for the first time by Atkinson \cite{atkinson} \footnote{The reader may check {[}{[}\cite{ivic_hardyz}{]}, p. 121{]} to see
the differences between the results of Hall and Atkinson.}and it was used by him to give a very short proof that the Riemann
zeta function has $\gg T/\log(T)$ critical zeros. Under the assumption
that $\sum c_{j}^{2}<\infty$ (which is actually implied by our condition
$\sum|c_{j}|<\infty$), Hall proved the analogous result to a shifted
combination of the form (\ref{combinations Hardy Z function}), establishing
that
\begin{equation}
\mathcal{N}_{0}(T)\gg\frac{T}{\log(T)},\label{estimate atkinson type}
\end{equation}
where $\mathcal{N}_{0}(T)$ denotes the number of the zeros of the form $s=\frac{1}{2}+it,\,\,0<t<T$
of the function $\mathcal{Z}\left(t\right)$.
Although the function (\ref{combinations Hardy Z function}) is very
different from the completed Dirichlet series $\eta(s)$, it is not
unreasonable to conjecture that lower bounds like (\ref{estimate atkinson type})
hold true in the setting of shifted combinations of our class of completed
Dirichlet series.

We state a conjecture which, if proven, would drastically improve on our Theorem \ref{theorem 1.1} when $z=0$. 
\begin{conjecture}
Assume the conditions of Theorem \ref{theorem 1.1} (with $z:=0$) and let $\mathcal{N}_{\alpha}(T)$
denote the number of zeros written in the form $s=\frac{\alpha}{4}+it,\,0\leq t\leq T$,
of the function
\[
\tilde{F}_{\alpha}(s):=\sum_{j\neq0}c_{j}\,\pi^{-(s+i\lambda_{j})}\Gamma\left(s+i\lambda_{j}\right)\,\zeta_{\alpha}\left(s+i\lambda_{j}\right).
\]
Then there exists some $c>0$ such that
\begin{equation}
\liminf_{T\rightarrow\infty}\,\frac{\mathcal{N}_{\alpha}(T)}{T}\geq c.\label{lim inf zeta alpha in general}
\end{equation}
\end{conjecture}

\bigskip{}

As it can be seen from the 4-square Theorem,
\[
\zeta_{4}(s)=8(1-2^{2-2s})\,\zeta(s)\,\zeta(s-1),
\]
the estimate (\ref{lim inf zeta alpha in general}) cannot be in general improved to $\mathcal{N}_{\alpha}(T)\gg T\,\log(T)$.
By Siegel's result \cite{Siegel_Contributions}, whenever $\alpha\in\mathbb{N}_{\geq4}$, $\zeta_{\alpha}\left(s\right)$
has $\asymp T$ zeros of the form $s=\frac{\alpha}{4}+it$, $0\leq t\leq T$
so, again, (\ref{lim inf zeta alpha in general}) seems to be the
best estimate that one may be able to prove. In any case, it can be
conjectured that, in the situation of $\zeta(s)$, one has the estimate
\begin{equation}
\liminf_{T\rightarrow\infty}\,\frac{\mathcal{N}_{1}(T)}{T\log(T)}\geq c,\label{positive proportion zeros combinations}
\end{equation}
which could act as an extension of Selberg's famous result \cite{selberg_zeros}
in a different direction than \cite{selberg_class}.

Analogously, we may formulate the following conjecture for holomorphic
cusp forms which, if true, would constitute a massive generalization
of Hafner's result \cite{Hafner_cusp forms} in a different setting than Selberg's result adapted to combinations of degree 2 $L-$ functions \cite{rezvyakova}.
\begin{conjecture}
Assume the conditions of Theorem \ref{theorem 1.2} (with $z:=0$) and let $\mathcal{N}_{f}(T)$
denote the number of zeros written in the form $s=\frac{k}{2}+it,\,\,\,0\leq t\leq T$,
of the function
\[
\tilde{G}_{f}(s):=\sum_{j\neq0}c_{j}\,\left(2\pi\right)^{-s-i\lambda_{j}}\Gamma(s+i\lambda_{j})\,L\left(s+i\lambda_{j},f\right).
\]
Then there exists some $d>0$ such that
\[
\liminf_{T\rightarrow\infty}\,\frac{N_{f,z}(T)}{T\log(T)}\geq d.
\]
\end{conjecture}

\bigskip{}

\textit{Acknowledgements:} This work was partially supported by CMUP, member of LASI, which is financed by national funds through FCT - Fundação para a Ciência e a Tecnologia, I.P., under the projects with reference UIDB/00144/2020 and UIDP/00144/2020. We also acknowledge the support from FCT (Portugal) through the PhD scholarship 2020.07359.BD. The author would like to thank to Semyon Yakubovich for unwavering support and guidance throughout the writing of this paper.

\end{document}